\documentclass[11pt,letterpaper, reqno]{amsart}
\usepackage{amscd}
\usepackage{amssymb}
\usepackage{amsthm}
\usepackage{graphicx}
\usepackage{color}
\usepackage[all]{xy}
\usepackage{mathrsfs}
\usepackage{marvosym}
\usepackage{stmaryrd}
\usepackage{srcltx}
\usepackage{mathtools}
\usepackage{scalerel,amssymb}
\usepackage{tikz-cd}
\usepackage{adjustbox}
\usepackage{placeins}
\usepackage{floatrow}
\newfloatcommand{capbtabbox}{table}[][\FBwidth]
\definecolor{amaranth}{rgb}{0.9, 0.17, 0.31}
\usepackage[colorlinks=true,citecolor=amaranth,linkcolor=black]{hyperref}%

\definecolor{shadecolor}{rgb}{1,0.9,0.7}

\newcommand\restr[2]{{
  \left.\kern-\nulldelimiterspace 
  #1 
  \vphantom{\big|} 
  \right|_{#2} 
  }}

\let\oldtocsection=\tocsection
 
\let\oldtocsubsection=\tocsubsection
 
\let\oldtocsubsubsection=\tocsubsubsection
 
\renewcommand{\tocsection}[2]{\hspace{0em}\oldtocsection{#1}{#2}}
\renewcommand{\tocsubsection}[2]{\hspace{1em}\oldtocsubsection{#1}{#2}}
\renewcommand{\tocsubsubsection}[2]{\hspace{2em}\oldtocsubsubsection{#1}{#2}}

\newtheorem{theorem}{Theorem}[section]
\newtheorem{lemma}[theorem]{Lemma}

\newtheoremstyle{defstyle}
  {.6em} 
  {.1em} 
  {} 
  {} 
  {\bfseries} 
  {.} 
  {.5em} 
  {} 
\theoremstyle{defstyle} \newtheorem{definition}[theorem]{Definition}

\newtheorem{example}[theorem]{Example}

\theoremstyle{remark}
\newtheorem{remark}[theorem]{Remark}

\numberwithin{equation}{section}
\numberwithin{figure}{section}

\newcommand{\NN} {\mathbb{N}}

\newcommand{\QQ} {\mathbb{Q}}
\newcommand{\RR} {\mathbb{R}}

\newcommand{\PP} {\mathbb{P}}
\renewcommand{\AA} {\mathbb{A}}
\newcommand{\GG} {\mathbb{G}}

\newcommand{\cO} {\mathcal{O}}

\newcommand {\shM}  {\mathcal{M}}

\newcommand {\shX}  {\mathcal{X}}

\newcommand {\fod}  {\mathfrak{d}}

\newcommand {\foj}  {\mathfrak{j}}

\newcommand {\Aut}  {\operatorname{Aut}}

\newcommand {\cl}  {\operatorname{cl}}
\newcommand {\codim} {\operatorname{codim}}

\newcommand {\coker} {\operatorname{coker}}

\newcommand {\ev}  {\operatorname{ev}}

\newcommand {\gp}  {{\operatorname{gp}}}

\newcommand {\Hom}  {\operatorname{Hom}}

\newcommand {\Int}  {\operatorname{Int}}

\renewcommand {\ker } {\operatorname{ker}}
\newcommand {\kk} {\Bbbk}

\newcommand {\lra}  {\longrightarrow}

\newcommand {\M} {\mathcal{M}}

\newcommand{\C} {\mathrm{C}}

\renewcommand{\O}  {\mathcal{O}}

\newcommand {\out}  {\mathrm{out}}
\newcommand{\trop}{\mathrm{trop}}

\renewcommand{\P}  {\mathscr{P}}

\newcommand {\pr}  {\operatorname{pr}}

\newcommand {\rk} {\operatorname{rk}}

\newcommand {\scrP}  {\mathscr{P}}

\newcommand {\sat}  {{\operatorname{sat}}}

\newcommand {\Spec} {\operatorname{Spec}}

\newcommand\Tr{\operatorname{Tr}}
\newcommand {\ul} {\underline}

\newcommand {\virt} {\mathrm{virt}}

\def\log{\mathrm{log}}
\def\ev{\mathrm {ev}}

\def\virt{\mathrm{vir}}
\def\PP{\mathbb{P}}

\def\cL{{\mathcal{L}}}

\def\cO{\mathcal{O}}

\def\cM{{\mathcal{M}}}
\def\cK{{\mathcal{K}}}
\def\C{{\mathcal{C}}}

\def\Z{\mathbb{Z}}

\def\C{\mathbb{C}}
\def\Q{\mathbb{Q}}

\def\Aut{{\rm Aut}}

\def\E{\mathrm{E}}

\def\mydate{\ifcase\month \or January\or February\or March\or
April\or May\or June\or July\or August\or September\or October\or 
November\or December\fi \space\number\day,\space\number\year}

\newtheoremstyle{cited}%
  {3pt}
  {3pt}
  {\itshape}
  {}
  {\bfseries}
  {.}
  {.3em}
  {\thmname{#1} \thmnumber{#2}\thmnote{\normalfont#3}}

\theoremstyle{cited}
\newtheorem{citedthm}{Theorem}

\theoremstyle{cited}

\begin{document}

\title[Quivers, Flow Trees, and Log Curves]{Quivers, Flow Trees, and Log Curves}

\author[H.\,Arg\"uz]{H\"ulya Arg\"uz}
\address{University of Georgia, Department of Mathematics, Athens, GA 30605}
\email{Hulya.Arguz@uga.edu}

\author[P.\,Bousseau]{Pierrick Bousseau}
\address{University of Georgia, Department of Mathematics, Athens, GA 30605}
\email{Pierrick.Bousseau@uga.edu}

\date{}

\begin{abstract}
Donaldson--Thomas (DT) invariants of a quiver with potential can be expressed in terms of simpler attractor DT invariants by a universal formula. The coefficients in this formula are calculated combinatorially using attractor flow trees. 
In this paper, we prove that these coefficients are genus 0 log Gromov--Witten invariants of $d$-dimensional toric varieties, where $d$ is the number of vertices of the quiver. 
This result follows from a log-tropical correspondence theorem which relates $(d-2)$-dimensional families of tropical curves obtained as universal deformations of attractor flow trees, and rational log curves in toric varieties.
\end{abstract}

\maketitle

\setcounter{tocdepth}{2}
\tableofcontents
\setcounter{section}{-1}

\section{Introduction}
This paper relates two seemingly different ways to define enumerative invariants given a quiver with potential: the quiver Donaldson--Thomas (DT) invariants, defined using moduli spaces of quiver representations, and log Gromov--Witten invariants of toric varieties, defined as counts of rational log curves satisfying constraints specified by the combinatorics of the quiver.

To relate quiver DT invariants to log Gromov--Witten invariants, we use the combinatorial calculation of quiver DT invariants in terms of attractor flow trees following our previous work \cite{ABflow}. We explain how to interpret such trees as tropical curves, and then prove a new correspondence theorem between counts of tropical curves and log Gromov--Witten invariants of toric varieties. 
Crucially, our log-tropical correspondence theorem involves non-rigid tropical curves moving in non-trivial families, unlike previous correspondence theorems in the literature where only rigid tropical curves are usually considered.

After a brief description of DT
quiver invariants in \S\ref{sec_dt_intro} and of log Gromov--Witten invariants in \S\ref{sec_gw_intro}, we provide a more detailed overview of the log-tropical correspondence in \S \ref{sec_log_tropical_intro}, and the quiver DT-log Gromov--Witten correspondence in \S \ref{sec_main_intro}.

\subsection{DT quiver invariants}
\label{sec_dt_intro}

A quiver with potential $(Q,W)$ is given by a finite oriented graph $Q$, and a finite formal linear combination $W$ of oriented cycles in $Q$. 
We denote by 
$Q_0$ the set of vertices of $Q$, and by $Q_1$ the set of edges of $Q$.
For every dimension vector $\gamma \in N \coloneqq \Z^{Q_0}$
and stability parameter
\begin{equation} \theta \in (\gamma^{\perp})_\RR  \subset M_\RR \coloneqq \Hom(N,\RR)\,,\end{equation} 
where $(\gamma^{\perp})_\RR \coloneqq \{\theta\in M_\RR|\,\theta(\gamma)=0\}$, 
the theory of King's stability for quiver representations
\cite{MR1315461} defines a quasiprojective variety
$M_\gamma^{\theta}$ over $\C$, parametrizing S-equivalence classes of 
$\theta$-semistable representations of $Q$ of dimension $\gamma$. Moreover, the trace of the potential naturally defines a regular function 
$\Tr (W)_\gamma^\theta \colon M_\gamma^\theta \longrightarrow \C$.

A stability parameter $\theta \in (\gamma^\perp)_\RR$ is called generic if $\theta \in (\gamma^{'
\perp})_\RR$ for $\gamma' \in N$ implies that $\gamma'$ is collinear with $\gamma$. For a generic stability parameter $\theta$, the \emph{DT invariant} $\Omega_\gamma^{+,\theta}$ is an integer which is a virtual count of the critical points of $\Tr (W)_\gamma^\theta$.
We have $\Omega_\gamma^{+,\theta}=0$ if the $\theta$-stable locus in $M_\gamma^\theta$ is empty, and else,  $\Omega_\gamma^{+,\theta}$ is the Euler characteristic of $M_\gamma^\theta$ valued in the perverse sheaf obtained by applying the vanishing-cycle functor for $\Tr (W)_\gamma^\theta$ to the intersection cohomology sheaf $IC$  normalized to be the constant sheaf in degree $0$ when $M_\gamma^\theta$ is smooth \cite{davison2015donaldson, MR4132957, MR4000572}:
\begin{equation} \label{eq_dt_intro}
\Omega_\gamma^{+,\theta}
=e(M_\gamma^\theta, \phi_{\mathrm{Tr}(W)_\gamma^\theta}(IC))
\in \Z
\,.\end{equation}
For example, if $\gamma$ is primitive and $W=0$, then $\Omega_\gamma^{+,\theta}$ is simply the topological Euler characteristic of $M_\gamma^\theta$.
For every $\gamma \in N$, we set 
\begin{equation*}
\kappa(\gamma):=(-1)^{\chi(\gamma,\gamma)} \in \{\pm 1\}\,,\end{equation*}
where $\chi: N \times N \rightarrow \Z$ is the Euler form of $Q$, given by 
\begin{equation*} 
\chi(\gamma,\gamma')=\sum_{i \in Q_0} \gamma_i \gamma_i' - \sum_{(\alpha:i \rightarrow j)\in Q_1} \gamma_i \gamma_j' \,.
\end{equation*}
It is often convenient to repackage the DT invariants 
$\Omega_\gamma^{+,\theta}$ into the \emph{rational DT invariants} 
\begin{equation}
\overline{\Omega}_\gamma^{+,\theta} \coloneqq \sum_{\substack{\gamma' \in N\\ 
    \gamma=k \gamma',\, k\in \Z_{\geq 1}}} \frac{\kappa(\gamma')^{k-1}}{k^2} \Omega_{\gamma'}^{+,\theta} \in \Q \,,
\end{equation}
 which can also be defined using the motivic Hall algebra \cite{JoyceSong, kontsevich2008stability, MR2650811,MR2801406}.
In the literature on DT invariants, it is more common to work with DT invariants $\Omega_\gamma^{\theta}$ defined by 
\[
\Omega_\gamma^{\theta}
:=(-1)^{\dim M_\gamma^\theta}\, e(M_\gamma^\theta, \phi_{\mathrm{Tr}(W)_\gamma^\theta}(IC)) \in \Z\,,\]
when the $\theta$-stable locus is non-empty, 
and the rational DT invariants $\overline{\Omega}_\gamma^\theta$ defined 
by 
\begin{equation} \nonumber
\overline{\Omega}_\gamma^\theta :=
\sum_{\substack{\gamma' \in N_Q\\ \gamma=k \gamma', k\in \Z_{\geq 1}}} \frac{1}{k^2} \Omega_{\gamma'}^\theta \in \QQ\,. \end{equation}
When the $\theta$-stable locus is non-empty, we have $\dim M_\gamma^\theta = 1-\chi(\gamma, \gamma)$, and so we have 
\[ \Omega_\gamma^{+,\theta}=-\kappa(\gamma) \Omega_\gamma^\theta
\,\,\,\,\text{and}\,\,\,\,\overline{\Omega}_\gamma^{+,\theta}=-\kappa(\gamma) \overline{\Omega}_\gamma^\theta\,.\]
In this paper, we work with the DT invariants $\Omega_\gamma^{+,\theta}$ as they are more directly related to log Gromov--Witten invariants.

The DT invariants $\Omega_\gamma^{+,\theta}$ are locally constant functions of the generic stability parameter $\theta \in \gamma^{\perp}$ and their jumps across the loci of non-generic stability parameters are controlled by the wall-crossing formula of Joyce-Song and 
Kontsevich-Soibelman \cite{JoyceSong,kontsevich2008stability}. 
Using the wall-crossing formula, the DT invariants can be computed in terms of the simpler \emph{attractor DT invariants}, which are 
DT invariants at specific values of the stability parameter.
Let $\omega \colon N \times N \rightarrow \Z$ be the skew-symmetric form on $N$ given by
\begin{equation}
\label{Eq: Euler form}
 \omega(\gamma, \gamma'):= \sum_{i,j\in Q_0}(a_{ij}-a_{ji})\gamma_i\gamma_j'\,,   
\end{equation}
where $a_{ij}$ is the number of arrows in $Q$ from the vertex $i$ to the vertex $j$. The specific point $\iota_\gamma \omega:=  \omega(\gamma,-) \in \gamma^{\perp} \subset M_{\RR}$ is called the \emph{attractor point} for $\gamma$ \cite{AlexandrovPioline, mozgovoy2020attractor}. In general, the attractor point 
$\iota_\gamma \omega$ is not generic and we define
the attractor DT invariants $\Omega_\gamma^{+,*}$ by
\begin{equation} \Omega_\gamma^{+,*}
\coloneqq \Omega_\gamma^{+, \theta_\gamma}\,,\end{equation} 
where $\theta_\gamma$
is a small generic perturbation of 
$\iota_\gamma \omega$ in $\gamma^{\perp}$ \cite{AlexandrovPioline, mozgovoy2020attractor}. One can check that $\Omega_\gamma^{+,*}$ is independent of the choice of the small perturbation
\cite{AlexandrovPioline, mozgovoy2020attractor}.

Iteratively using the wall-crossing formula, the rational DT invariants $\overline{\Omega}_\gamma^{+,\theta}$
for generic stability parameters $\theta \in \gamma^{\perp}$
are expressed in terms of the rational attractor DT invariants $\overline{\Omega}_\gamma^{+,*}$ by a formula of the form
\begin{equation} \label{eq_reconstruction_intro}
    \overline{\Omega}_\gamma^{+,\theta} = \sum_{r\geq 1} \sum_{\substack{\{\gamma_i\}_{1\leq i\leq r}\\ \sum_{i=1}^r \gamma_i = \gamma}} \frac{1}{|\Aut(\{\gamma_i\}_i)|} 
    F_r^{\theta}(\gamma_1,\dots,\gamma_r) \prod_{i=1}^r \overline{\Omega}_{\gamma_i}^{+,*}\,,
\end{equation}
where the second sum is over the multisets $\{\gamma_i\}_{1\leq i\leq r}$ with $\gamma_i \in N$ and $\sum_{i=1}^r \gamma_i=\gamma$ \cite{AlexandrovPioline, KS}. Here, the denominator $|\Aut(\{\gamma_i\}_i)|$ is the order of the symmetry group of $\{\gamma_i\}$: if $m_{\gamma'}$ is the number of times that $\gamma' \in N$ appears in $\{\gamma_i\}_i$, then 
$|\Aut(\{\gamma_i\}_i)|=\prod_{\gamma'\in N}m_{\gamma'}!$.
The coefficients $F_r^{\theta}(\gamma_1,\dots,\gamma_r)$ 
are integers and are universal in the sense that they depend 
on $(Q,W)$ only through the skew-symmetric form 
$\omega$ on $N$. The dependence on the potential $W$ is entirely contained in the attractor DT invariants $\overline{\Omega}_{\gamma_i}^{+,*}$. 
Moreover \cite[\S 3.2]{KS}, the coefficients $F_r^\theta(\gamma_1,\dots,\gamma_r)$
have a natural decomposition 
\begin{equation}\label{eq_trees_intro} F_r^\theta(\gamma_1,\dots,\gamma_r) = \sum_h F_{r,h}^\theta(\gamma_1,\dots,\gamma_r)\end{equation}
indexed by \emph{attractor flow trees} $h : T \rightarrow M_\RR$, which are rooted balanced trees in $M_\RR$ with root at $\theta$ and with $r$ leaves decorated by $\gamma_1, \dots,\gamma_r$ -- see Definition \ref{def_attractor_tree} for details. The contribution $F_{r,h}^\theta(\gamma_1,\dots,\gamma_r)$ of $h : T \rightarrow M_\RR$
is obtained by applying the wall-crossing formula at each vertex of $T$.

In our previous paper \cite{ABflow}, we proved an explicit formula, called the 
\emph{flow tree formula} and conjectured by 
Alexandrov-Pioline
 \cite{AlexandrovPioline}, which computes the coefficients
$F_{r,h}^{\theta}(\gamma_1,\dots,\gamma_r)$
in \eqref{eq_trees_intro}  in terms of a sum over trivalent perturbed attractor flow trees obtained as small perturbations of $h:T \rightarrow M_\RR$. A different combinatorial formula for $F_r^{\theta}(\gamma_1,\dots,\gamma_r)$
was proved by Mozgovoy in  \cite{mozgovoy2022operadic}. 
The main goal of the present paper is to give a geometric interpretation of these coefficients $F_{r,h}^{\theta}(\gamma_1,\dots,\gamma_r)$ in terms of genus $0$ log Gromov--Witten invariants of toric varieties.

\subsection{Log Gromov--Witten invariants} \label{sec_gw_intro}
Log Gromov--Witten theory, developed by Abramovich--Chen and Gross--Siebert \cite{logGWbyAC,logGW}, provides a framework to enumerate algebraic curves in an algebraic variety $X$ having prescribed tangency conditions along a divisor $D \subset X$ which is allowed to have particular types of singularities -- typically we consider normal crossing divisors. Generalizing the notion of a stable map $f: C\to X$ to the log setting amounts to endowing all spaces with (fine, saturated) log structures and lift all morphisms to morphisms of log spaces -- see \S \ref{sec_log_trop_review} for a brief review of log schemes and stable log maps.  

The additional data of a log structure on a variety $X$ encodes combinatorial information which is used to define a cone complex associated to $X$, referred to as the \emph{tropicalization} of $X$. This makes it feasible to count log curves in terms of their tropical counterparts. In \cite{NS}, Nishinou-Siebert used toric degenerations and log deformation theory to prove 
a correspondence theorem between counts of genus $0$ complex curves in $d$-dimensional toric varieties and counts of tropical curves in $\RR^d$. More precisely, the main result of \cite{NS} provides a correspondence between counts of stable log maps and counts of rigid tropical trees. To obtain such a correspondence, one considers log curves generically contained in the big torus orbit of the toric variety and with generically trivial log structure. 

In the present paper, given a quiver with $d$ vertices, we prove a correspondence theorem between counts of log curves in $d$-dimensional toric varieties and tropical counts associated to $(d-2)$-dimensional families of tropical curves in $\RR^d$ obtained as universal deformations of the attractor flow trees. 
A key and major difference with the set-up of \cite{NS} is that for $d \geq 3$, the relevant log curves are entirely contained in the toric boundary divisor, have generically non-trivial log structures, and correspond to non-rigid tropical curves moving in non-trivial $(d-2)$-dimensional families.

\begin{figure}[h]
\center{\includegraphics{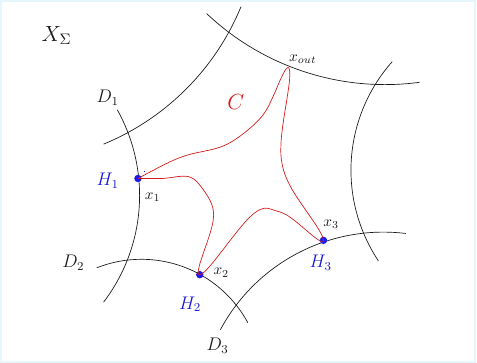}}
\caption{A stable log map in $\mathcal{M}_{\omega,\boldsymbol{\gamma},\mathbf{H}}^\log(X_\Sigma)$ for $d=2$ and $r=3$.\label{fig1}}
\end{figure}

We now describe more precisely our log-tropical correspondence theorem. As in \S\ref{sec_dt_intro}, we fix a quiver $Q$, with set of vertices $Q_0$ of cardinality $d$, and let $\omega$ be the skew-symmetric form on $N$ as in \eqref{Eq: Euler form}. We denote
\[N=\Z^{Q_0}\simeq \Z^d \,\ \mathrm{and} \,\ M_\RR=\Hom(N,\RR)\simeq \RR^d \,.\] 
We also fix a tuple $\boldsymbol{\gamma}=(\gamma_1,\dots,\gamma_r)$ of elements $\gamma_i \in N$ such that $\iota_{\gamma_i}\omega\neq 0$ for all $1\leq i\leq r$, and $\iota_\gamma \omega \neq 0$, where $\gamma=\sum_{i=1}^r \gamma_i$. We count curves in a $d$-dimensional projective toric variety $X_\Sigma$ defined by a fan in $M_\RR$ containing the rays $\RR_{\geq 0}\iota_{\gamma_i}\omega$, and such that the hyperplanes $(\gamma_i^\perp)_\RR$ in $M_\RR$ are unions of cones of $\Sigma$. We refer to such a fan as a $\boldsymbol{\gamma}$-fan (see Definition \ref{def_gamma_fan}). In particular, for every $1\leq i \leq r$, we have a toric divisor $D_i$ of $X_\Sigma$ corresponding to the ray $\RR_{\geq 0}\iota_{\gamma_i}\omega$. The toric variety $X_\Sigma$ has a natural divisorial log structure defined by its toric boundary $D_\Sigma$, and elements of $N$ naturally define tangency conditions along $D_\Sigma$  at marked point for stable log maps to $X_\Sigma$ -- see \S\ref{sec_log_trop_review} for details. 

Consider hypersurfaces $H_i$ in $D_i$, given by equation of the form $z^{\frac{\gamma_i}{|\gamma_i|}}=c_i$ for some constant $c_i$, where $|\gamma_i|$ is the divisibility of $\gamma_i$ in $N$, and let $\mathbf{H}=(H_1,\dots,H_r)$. We denote by
\[\mathcal{M}_{\omega,\boldsymbol{\gamma},\mathbf{H}}^\log(X_\Sigma)\] 
the moduli space of genus $0$ stable log maps $f :C \rightarrow X_\Sigma$ having $r+1$ marked points $x_1,\dots,x_r,x_{\mathrm{out}}$, with tangency condition along $D_\Sigma$ given by $\iota_{\gamma_i} \omega$ at $x_i$, $-\iota_\gamma \omega$ at $x_{\mathrm{out}}$, and such that $f(x_i) \in H_i$ for all $1\leq i \leq r$, see Figure \ref{fig1}. The moduli space $\mathcal{M}_{\omega,\boldsymbol{\gamma},\mathbf{H}}^\log(X_\Sigma)$ is naturally stratified by the combinatorics of the log structure, and is log smooth of dimension $d-2$ for general $H_i$'s -- see Theorem \ref{thm_enum}. If $f :C \rightarrow X_\Sigma$ is a general point in a codimension $k$-stratum of  $\mathcal{M}_{\omega,\boldsymbol{\gamma},\mathbf{H}}^\log(X_\Sigma)$, then the tropicalization of  $f :C \rightarrow X_\Sigma$ is naturally a $k$-dimensional face of the moduli space $\mathcal{M}_{\omega,\boldsymbol{\gamma},\mathbf{A}^0}^\trop(\Sigma)$ of genus $0$ tropical curves $h : \Gamma \rightarrow M_\RR$ in $M_\RR$ with $r+1$ legs $L_1,\dots,L_r, L_{\mathrm{out}}$, of weighted directions $\iota_{\gamma_i}\omega$ and $-\iota_\gamma \omega$ and such that $h(L_i) \subset A_i^0:=(\gamma_i^\perp)_\RR$, see Figure \ref{fig2}. In particular, given a $(d-2)$-dimensional face $\rho$ of $\mathcal{M}_{\omega,\boldsymbol{\gamma},\mathbf{A}^0}^\trop(\Sigma)$, one obtains a log Gromov--Witten 
\begin{equation}\label{eq_N_log_intro}
N_\rho^{\mathrm{toric}}(X_\Sigma)\end{equation}
by counting the $0$-dimensional strata of $\mathcal{M}_{\omega,\boldsymbol{\gamma},\mathbf{H}}^\log(X_\Sigma)$ with tropicalization $\rho$.

\begin{figure}[h]
\center{\includegraphics{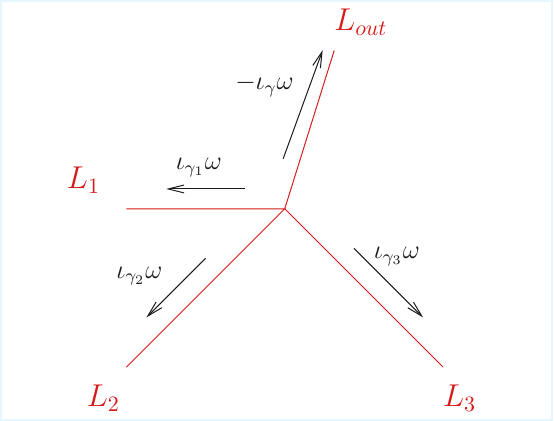}}
\caption{A tropical curve in $\mathcal{M}_{\omega,\boldsymbol{\gamma},\mathbf{A}^0}^\trop(\Sigma)$ for $d=2$ and $r=3$.\label{fig2}}
\end{figure}

\subsection{Log-tropical correspondence}
\label{sec_log_tropical_intro}
To compute the log Gromov--Witten invariant $N_\rho^{\mathrm{toric}}(X_\Sigma)$ tropically, one considers general affine perturbations $A_i$ of the linear hyperplanes $A_i^0=(\gamma_i^\perp)_\RR$, and the corresponding moduli space  \[\mathcal{M}_{\omega,\boldsymbol{\gamma},\mathbf{A}}^\trop\] of tropical curves in $M_\RR$ with $h(L_i) \subset A_i$ for all $1\leq i\leq r$. There is a finite $S_\rho$ of $(d-2)$-dimensional faces $\sigma$ of $\mathcal{M}_{\omega,\boldsymbol{\gamma},\mathbf{A}}^\trop$ which in the limit $\mathbf{A} \rightarrow \mathbf{A}^0$ recovers $\rho$. While a general tropical curve in the face $\rho$ might have vertices of arbitrary valency, general tropical curves in the faces $\sigma$ are trivalent
-- see Lemma \ref{lem_dimension} and Figure \ref{fig3}.

\begin{figure}[h]
\center{\includegraphics{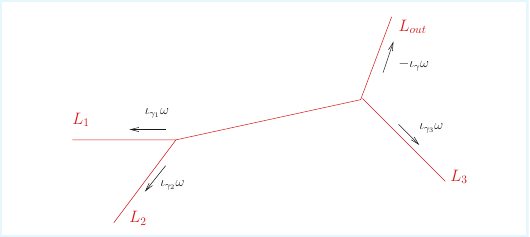}}
\caption{A tropical curve in $\mathcal{M}_{\omega,\boldsymbol{\gamma},\mathbf{A}^0}^\trop(\Sigma)$ for $d=2$ and $r=3$.\label{fig3}}
\end{figure}

In particular, one can define the tropical multiplicity $N_\sigma^\trop$ of $\sigma$ either as a lattice index -- see Definition \ref{def_trop_mult}, or as a product over the vertices of local multiplicities determined by the skew-symmetric form $\omega$ -- see Lemma \ref{lem: product2 for Ntrop}.
Moreover, if $\dim \fod_{L_{\mathrm{out}}}^\rho=d-1$, where $\fod_{L_{\mathrm{out}}}^\rho \subset M_\RR$ is the union of the loci $h(L_{\mathrm{out}})$ for $h \in \rho$, we define a tropical coefficient $k_\rho \in \Z_{\geq 1}$ as an explicit lattice index 
-- see \eqref{eq_coeff}. We also define a similar coefficient $k_\sigma$ for $\sigma \in S_\rho$.

We can now state the log-tropical correspondence theorem for the log Gromov--Witten invariants $N_\rho^{\mathrm{toric}}(X_\Sigma)$.

\begin{citedthm}
\label{thm_log_trop_intro}
Fix a skew-symmetric form $\omega$ on $N$, and a tuple $\boldsymbol{\gamma}=(\gamma_1,\dots,\gamma_r)$ of elements $\gamma_i \in N$ such that $\iota_{\gamma_i}\omega\neq 0$, and $\iota_\gamma \omega \neq 0$, where $\gamma=\sum_{i=1}^r \gamma_i$.  Fix also a $\boldsymbol{\gamma}$-fan $\Sigma$.
Then, for every general tuple $\mathbf{A}=(A_1,\dots,A_r)$ of affine hyperplane in $M_\RR$ parallel to the linear hyperplanes $(\gamma_i^\perp)_\RR$, and 
for every $(d-2)$-dimensional face $\rho$ of $\cM^\trop_{\omega,\boldsymbol{\gamma},\mathbf{A}^0}(\Sigma)$ such that $\dim \fod_{L_{\mathrm{out}}}^{\rho}=d-1$, 
the log Gromov--Witten invariant $N_\rho^{\mathrm{toric}}(X_\Sigma)$ of the toric variety $X_\Sigma$ 
satisfies
\begin{equation} \label{eq_log_trop_intro}
k_\rho N_\rho^{\mathrm{toric}}(X_\Sigma)=\sum_{\sigma \in S_\rho}  k_\sigma N_\sigma^\trop \,.\end{equation}
\end{citedthm}


We prove Theorem \ref{thm_log_trop_intro} using a toric degeneration of $X_\Sigma$ and the recent theory of punctured log maps of Abramovich--Chen--Gross--Siebert \cite{ACGSII} to construct stable log maps to the special fiber by gluing.

\subsection{Quiver DT-log Gromov--Witten correspondence}
\label{sec_main_intro}
Combining the flow tree formula of \cite{ABflow} with the log-tropical correspondence given by Theorem \ref{thm_log_trop_intro}, we obtain a geometric interpretation of the coefficients $F_{r,h}^\theta(\gamma_1,\dots,\gamma_r)$ in terms of genus $0$ log Gromov--Witten invariants of toric varieties. 

Let $(Q,W)$ be a quiver with potential with $d$ vertices as in 
\S \ref{sec_dt_intro} and $\omega$ be the corresponding skew-symmetric form on $N=\Z^{Q_0}$. We fix a dimension vector $\gamma \in N$
such that $\iota_\gamma \omega\neq 0$
and a decomposition $\gamma=\sum_{i=1}^r \gamma_i$ such that $\iota_{\gamma_i}\omega \neq 0$ for all $1\leq i\leq r$. We denote $\boldsymbol{\gamma}=(\gamma_1,\dots,\gamma_r)$ and we fix a $\boldsymbol{\gamma}$-fan $\Sigma$.
Then, every attractor flow tree $(h: T \rightarrow M_\RR)$ naturally defines a  $(d-2)$-dimensional face $\Tilde{\rho}_h$ of the moduli space of tropical curves $\mathcal{M}_{\omega,\boldsymbol{\gamma},\mathbf{A}^0}^\trop$. One first turns $h$ into a tropical curve $\overline{h}: \overline{T} \rightarrow M_\RR$ by extending its root to infinity in the direction $-\iota_\gamma \omega$. Then, $\Tilde{\rho}_h$ is the smallest face of $\mathcal{M}_{\omega,\boldsymbol{\gamma},\mathbf{A}^0}^\trop$ containing $\overline{h}$, that is, the universal family of deformations of $\overline{h}$ preserving its combinatorial type -- see Lemma \ref{lem_tree_dim}.
 In particular, one can consider as in \S\ref{sec_gw_intro} the genus $0$ log Gromov--Witten invariant $N_{\widetilde{\rho}_h}^{\mathrm{toric}}(X_\Sigma)$ of the $d$-dimensional toric variety $X_\Sigma$ of fan $\Sigma$, and the tropical coefficient $k_{\widetilde{\rho}_h}$.

\begin{citedthm} \label{thm_dt_gw_intro}
Let $(Q,W)$ be a quiver with potential. Fix a dimension vector $\gamma \in N=\Z^{Q_0}$ such that $\iota_\gamma \omega \neq 0$, a decomposition $\gamma=\sum_{i=1}^r \gamma_i$ such that $\iota_{\gamma_i}\omega \neq 0$ for all $1\leq i\leq r$, a generic stability parameter $\theta \in \gamma^\perp$, and an attractor flow tree $h: T \rightarrow M_\RR$ with root at $\theta$ and $r$ leaves decorated by $\gamma_1,\dots,\gamma_r$.
Then, for every $\boldsymbol{\gamma}$-fan $\Sigma$,
the coefficient  $F_r^\theta(\gamma_1,\dots,\gamma_r)$
in \eqref{eq_reconstruction_intro}-\eqref{eq_trees_intro}
satisfies
\begin{equation} \label{eq_dt_log_intro}
F_{r,h}^\theta(\gamma_1,\dots,\gamma_r)
= \frac{\prod_{i=1}^r |\gamma_i|}{|\gamma|}
k_{\widetilde{\rho}_h}
N_{\widetilde{\rho}_h}^{\mathrm{toric}}(X_\Sigma)\,.\end{equation}
\end{citedthm}

If $\iota_\gamma \omega =0$, then $\overline{\Omega}_\gamma^{+,\theta}$ does not experience any non-trivial wall-crossing and so $\overline{\Omega}_\gamma^{+,\theta}=\overline{\Omega}_\gamma^{+,\star}$ for every $\theta \in \gamma^\perp$. 
Moreover, $F_{r,h}^\theta(\gamma_1,\dots,\gamma_r)=0$ if $r\geq 2$ and if $\iota_{\gamma_i}\omega = 0$ for some $i$. In particular, Theorem \ref{thm_dt_gw_intro} applies to compute $F_{r,h}^\theta(\gamma_1,\dots,\gamma_r)$ in every non-trivial situation.

By the flow tree formula of \cite{ABflow}, one can compute $F_{r,h}^\theta(\gamma_1,\dots,\gamma_r)$ using generic trivalent perturbations of $h$. To prove Theorem \ref{thm_dt_gw_intro}, we show that the universal deformations of these perturbed attractor trees are exactly the $(d-2)$-dimensional families of tropical curves $\sigma$ appearing in the right-hand side of \eqref{eq_log_trop_intro}, and then we use the log-tropical correspondence of Theorem \ref{thm_log_trop_intro} to rewrite the corresponding tropical counts as the log Gromov--Witten invariant $N_{\widetilde{\rho}_h}^{\mathrm{toric}}(X_\Sigma)$.

\subsection{Related works}
Another way to construct a geometrical object out of a quiver is based on the theory of cluster algebras
\cite{FZ} and cluster varieties \cite{FG}.
From a quiver $Q$ with $d$ vertices, one can construct the corresponding $d$-dimensional cluster variety using cluster transformations in a way prescribed by the combinatorics of $Q$. A key result of Gross--Hacking--Keel--Kontsevich \cite{GHKK} shows that the algebra of regular functions of a cluster variety admits a canonical basis, which can be constructed using a combinatorial gadget, known as the cluster scattering diagram. A version of such a cluster scattering diagram is shown to capture quiver DT invariants in the work of Bridgeland \cite{BridgelandCluster}. On the other hand, Gross--Siebert in \cite{gross2021canonical} show how to construct a mirror to a given log Calabi--Yau variety from a canonical scattering diagram, defined using counts of log curves. In our previous paper \cite{ABclustermirror} we showed how the canonical scattering diagram for cluster varieties relates to the cluster scattering.

Building on the results in this paper, in a sequel work \cite{ABtrans} we provide a correspondence between quiver DT invariants, in the situation when attractor DT invariants are trivial, and counts of log curves in cluster varieties, using the comparison between cluster and canonical scattering diagrams. This generalizes the Kronecker-Gromov--Witten correspondence which was shown earlier in the case when the cluster variety is of dimension two \cite{bousseau2020quantum, GP, GPS,MR3033514, MR3004575}. For variants of the Kronecker-Gromov--Witten correspondence in dimension two see also \cite{bousseau2018example,reineke2021moduli}. Such a correspondence between quiver DT invariants and log curve counts in cluster varieties is compatible with the results of the present paper, as cluster varieties can be obtained by non-toric blow-ups of toric varieties as explained in \cite{GHKbirational}, and so admit  degenerations into a toric variety and other simpler varieties, as studied in \cite{HDTV,GPS}. In particular, log curve counts in cluster varieties are naturally related to log curve counts in toric varieties considered in this paper. 

\subsection{Acknowledgments}
H\"ulya Arg\"uz was supported by the NSF grant DMS-2302116, and Pierrick Bousseau was supported by the NSF grant DMS-2302117. Some of the material in this paper was originally written with the intention of appearing as part of \cite{HDTV} in an earlier form. We thank Mark Gross for his generosity to let us use this material. We also thank the anonymous referee for the time they spared to carefully read our manuscript and for the very useful comments they provided towards the improvement of it.

\section{The tropical enumerative problem}
\label{sec:tropical_enum}
Throughout this paper we work over an algebraically closed field $\kk$ of characteristic $0$. We denote by $M \simeq \Z^d$ a free abelian group of finite rank $d$, and  by $M_\RR:=M \otimes_\Z \RR \simeq \RR^d$ the associated $d$-dimensional real vector space. We let $N:=\Hom(M,\Z)$ denote the dual abelian group. 

\subsection{Tropical curves}
\label{sec_trop_curves}

We mainly follow the convention and notation of \cite{NS}.
In this paper, a \emph{graph} $\Gamma$ is  finite and connected, consisting of vertices, edges connecting pairs of vertices, and legs adjacent to single vertices. Moreover, we assume that legs are marked by distinct labels.
We denote by
\begin{align*}
V(\Gamma) & \coloneqq \text{the set of vertices of\,\,} \Gamma \\
E(\Gamma) & \coloneqq \text{the set of edges of\,\,} \Gamma \\
L(\Gamma) & \coloneqq \text{the set of legs of\,\,} \Gamma\,. 
\end{align*}
We denote the set of the two vertices adjacent to an edge $E$ by $\partial E$, and similarly the vertex adjacent to a leg $L$ by $\partial L$.

\begin{definition} A \emph{weighted graph} is a graph $\Gamma$ together with a weight function $w\colon E(\Gamma)\cup L(\Gamma) \rightarrow \Z_{\geq 1}$ assigning a weight $w(E) \in \Z_{\geq 1}$ to every edge or leg $E$ of $\Gamma$.
\end{definition}

\begin{definition} \label{def_tropical_curve}
A \emph{parametrized tropical curve in $M_\RR$} 
is a weighted graph $\Gamma$ without divalent vertices, together with a proper continuous map $h: \Gamma \rightarrow M_\RR$
satisfying the following conditions:
\begin{itemize}
    \item[(i)] for every edge or leg $E$, the restriction $h|_E$ is an embedding with image contained in an affine line with rational slope,
       \item[(ii)] for every vertex $v\in V(\Gamma)$, the following \emph{balancing condition} holds: Let $E_1,\dots,E_m \in E(\Gamma)\cup L(\Gamma)$ be the edges or legs adjacent to $v$, and let $\bar{u}_i \in M$ be the primitive integral vector emanating from $h(v)$ in the direction of $h(E_i)$, then
\[\sum_{i=1}^m w(E_i)\bar{u}_i=0\,.\]
\end{itemize}

\end{definition}

An isomorphism of parameterized tropical curves $h:\Gamma \to M_\RR$ and $h_0 :\Gamma_0 \to M_\RR$
is a homeomorphism $\Psi:\Gamma\to\Gamma_0$ respecting the marking of the legs, the weights of the edges and legs, and such that $h=h_0\circ \Psi$. A \emph{tropical curve} is an isomorphism class of parameterized tropical
curves. The genus of a tropical curve $h:\Gamma \to M_\RR$ is the first Betti number of its domain $\Gamma$. A rational tropical curve is a tropical curve of genus zero. In what follows we focus attention to rational tropical curves.

\begin{definition}\label{def_tropical_type}
A \emph{tropical type} for $M_\RR$ is the data $(\Gamma,\bar{u})$ of a weighted graph $\Gamma$ and of a map
\begin{align*}
    \bar{u} : F(\Gamma) & \longrightarrow M \\
    (v,E) & \longmapsto \bar{u}_{v,E} 
\end{align*} 
where $F(\Gamma)$ is the set of flags of $\Gamma$, that is, of pairs $(v,E)$ where $v \in V(\Gamma)$ is a vertex and $e\in E(\Gamma)\subset L(\Gamma)$ is an edge or leg adjacent to $v$.
\end{definition}

\begin{definition}\label{def_type}
The \emph{type} of a tropical curve $h \colon \Gamma \rightarrow M_\RR$ is the tropical type $(\Gamma, \bar{u})$ where for every $(v,E)\in F(\Gamma)$,  $\bar{u}_{v,E}$ is the primitive integral in $M$ emanating from $h(v)$ in the direction of $h(E)$.
\end{definition}

The moduli space of tropical curves of a given type $\tau=(\Gamma, \bar{u})$ is naturally the interior of a polyhedral cone. For example, when $\Gamma$ is of genus $0$, this moduli space is isomorphic to $M_\RR \times \RR_{>0}^{|E(\Gamma)|}$, where $M_\RR$ parametrizes the position of a chosen vertex of $\Gamma$, and $\RR_{>0}^{|E(\Gamma)|}$ parametrizes the affine lengths of the images $h(E)$ of the edges $E\in E(\Gamma)$.

\begin{definition}
The \emph{degree} of a tropical type $(\Gamma,\bar{u})$ is the tuple \[\Delta=(w(L)\bar{u}_{v,L})_{L\in L(\Gamma)} \in M^{L(\Gamma)}\]
of weighted directions of the legs, where $v$ is the unique vertex of $\Gamma$ adjacent to the leg $L$. The degree of a tropical curve is the degree of its type.
\end{definition}

By \cite[Proposition 2.1]{NS}, there are finitely many tropical types of given genus and degree which are realized by tropical curves. In particular, the moduli space $\cM^\trop(\Delta)$ of rational tropical curves of given degree $\Delta$ is naturally a finite cone complex, obtained by gluing together the cones obtained as the closure of the spaces of tropical curves of given types (when the length of an edge goes to $0$, a vertex of higher valency is produced).

\subsection{$(\omega,\boldsymbol{\gamma})$-marked tropical curves}
\label{sec_marked_tropical_curves}
Let $\omega \in \bigwedge^2 M$ be a skew-symmetric form on $N$. For every $n \in N$, we denote by $\iota_n \omega \in M$ the contraction of $\omega$ by $n$, that is, the linear form $\omega(n,-)$ on $N$.
Let $\boldsymbol{\gamma}=(\gamma_1,\dots,\gamma_r)$ be a $r$-tuple of elements $\gamma_i \in N$ such that $\iota_{\gamma_i} \omega \neq 0$ for all $1\leq i \leq r$ and $\iota_\gamma \omega \neq 0$, where $\gamma:=\sum_{i=1}^r \gamma_i$.

\begin{definition} \label{def_marked}
An $(\omega,\boldsymbol{\gamma})$-marked tropical curve in $M_\RR$ is a rational tropical curve $h\colon \Gamma \rightarrow M_{\RR}$ of 
degree \[\Delta_{\omega,\boldsymbol{\gamma}}:=(\iota_{\gamma_1}\omega,
\dots,\iota_{\gamma_r}\omega, -\iota_\gamma \omega)\,.\] In other words, $\Gamma$ is a genus $0$ graph, with $r+1$ legs $L_1,\dots, L_r$, $L_{\mathrm{out}}$, such that
$w(L_i)\bar{u}_{v_i,L_i}=-\iota_{\gamma_i} \omega$ for all 
$1\leq i \leq r$, and $w(L_{\mathrm{out}})\bar{u}_{v_{\mathrm{out}},L_{\mathrm{out}}}=-\iota_\gamma \omega$, where $v_i$ (resp.\ $v_{\mathrm{out}}$) is the unique vertex of $\Gamma$ adjacent to $L_i$ (resp.\ $L_{\mathrm{out}}$).
\end{definition}

Given an $(\omega,\boldsymbol{\gamma})$-marked tropical curve $h\colon \Gamma \rightarrow M_\RR$, we usually view $\Gamma$ as a  tree with root leg $L_{\mathrm{out}}$ and leaves $L_i$ for $1\leq i\leq r$. Correspondingly, we say that an edge $E$ is a child (resp.\ a parent) of a vertex $v$ if $E$ is adjacent to $v$ and is not (resp.\ is) contained in the unique path in $\Gamma$ between $v$ and the root. We similarly define descendants and ancestors. For every edge or leg $E$ of $\Gamma$,
we denote 
\begin{equation}\label{eq_u_E}u_E:=w(E)\bar{u}_{v,E}\,,\end{equation}
where $v$ is the unique vertex such that $E$ is a parent of $v$. In other words, $u_E$ is the weighted direction of $E$ following the flow on $\Gamma$ starting at the leaves and ending at the root.

\begin{definition} \label{def_class}
Let $h \colon \Gamma \rightarrow M_\RR$ be an $(\omega,\boldsymbol{\gamma})$-marked tropical curve in $M_\RR$. Then for every edge or leg $E$ of $\Gamma$, the \emph{class of $E$} is the element $\gamma_E \in N$ defined by 
\[ \gamma_E=\sum_{i} \gamma_i \]
where the sum is over the indices $1\leq i \leq r$ such that the leg $L_i$ is a descendant of $E$.
\end{definition}

\begin{lemma}  \label{lem_class}
Let $h \colon \Gamma \rightarrow M_\RR$ be an $(\omega,\boldsymbol{\gamma})$-marked tropical curve in $M_\RR$. Then for every edge or leg $E$ of $\Gamma$, we have $u_E= -\iota_{\gamma_E}\omega$. In particular, 
the class $\gamma_E$ is non-zero.
\end{lemma}

\begin{proof}
The equality $u_E= -\iota_{\gamma_E}\omega$ is true for leaves by Definition \ref{def_marked} of a $\boldsymbol{\gamma}$-marked tropical curve  and then follows for the other edges and the root leg by the balancing condition in 
Definition \ref{def_tropical_curve}(ii). This implies that $\gamma_E \neq 0$ because $u_E \neq 0$ by definition of a tropical curve in $M_\RR$ (see \eqref{eq_u_E} and Definition \ref{def_tropical_curve}(i)).
\end{proof}

To simplify the notation, instead of $\cM^{\trop}(\Delta_{\omega,\boldsymbol{\gamma}})$, in what follows we use
$\cM_{\omega,\boldsymbol{\gamma}}^{\trop}$ for the moduli space of $(\omega,\boldsymbol{\gamma})$-marked tropical curves in $M_\RR$. It follows from the discussion after Definition \ref{def_type} that $\cM_{\omega,\boldsymbol{\gamma}}^{\trop}$ is naturally a finite polyhedral cone.
For every cone $\sigma$ of $\cM_{\omega,\boldsymbol{\gamma}}^\trop$, the relative interior $\Int(\sigma)$ of $\sigma$
parametrizes tropical curves of a given type $(\Gamma_\sigma, \bar{u}_\sigma)$. Moreover, the dimension of $\sigma$ is given by 
\[ \dim \sigma=d+|E(\Gamma_\sigma)|\,,\]
which can also be written using that $\Gamma_\sigma$ has genus $0$ as 
\[ \dim \sigma=d-2+r-\mathrm{ov}(\Gamma_\sigma)\,.\]
Here, $\mathrm{ov}(\Gamma_\sigma)$ is the \emph{overvalence} of $\Gamma_\sigma$, which is the non-negative integer defined by 
\begin{equation}
\label{eq_overvalence}
    \mathrm{ov}(\Gamma_\sigma):=\sum_{v\in V(\Gamma)}\mathrm{ov}(v)\,,
\end{equation}
where $\mathrm{ov}(v):=k-3$ if $v$ is a $k$-valent vertex. In particular, $\sigma$ is of dimension at most $d-2+r$, and is of dimension $d-2+r$ if and only if $\Gamma_\sigma$ is trivalent.

\subsection{Tropical constraints}
\label{sec_tropical_constraints}

\begin{definition}
\label{Def affine constraint}
A \emph{$\boldsymbol{\gamma}$-constraint} is a $r$-tuple $\mathbf{A}=(A_1,\dots, A_r)$ of affine hyperplanes $A_i$ in $M_\RR$ whose associated linear hyperplanes are given by $(\gamma_i^{\perp})_\RR \subset M_\RR$.
\end{definition}

Choices of $\boldsymbol{\gamma}$-constraints $\mathbf{A}$ are in one-to-one correspondence with choices of points $[\mathbf{A}]=([A_i])_i \in \prod_{i=1}^r M_\RR/(\gamma_i^{\perp})_\RR\simeq \RR^r$.

\begin{definition} \label{def_matching}
An $(\omega,\boldsymbol{\gamma})$-marked tropical curve $h: \Gamma \rightarrow M_\RR$ \emph{matches} a $\boldsymbol{\gamma}$-constraint $\mathbf{A}$ if $h(L_i) \subset A_i$ for all $1\leq i \leq r$.
\end{definition}

We denote by $\cM_{\omega,\boldsymbol{\gamma}, \mathbf{A}}^\trop$ the moduli space of $(\omega,\boldsymbol{\gamma})$-marked tropical curves matching the $\boldsymbol{\gamma}$-constraint $\mathbf{A}$. 
In other words,
$\cM_{\omega,\boldsymbol{\gamma}, \mathbf{A}}^\trop\coloneqq (\ev^\trop)^{-1}([\mathbf{A}])$,
where $\ev^\trop$ is the evaluation map at the legs $L_i$, defined by
\begin{align}\label{eq_ev_trop}
\ev^\trop \colon \cM_{\omega,\boldsymbol{\gamma}}^\trop &\longrightarrow \prod_{i=1}^r M_\RR/(\gamma_i^{\perp})_\RR \simeq \RR^r
\\ 
\nonumber h &\longmapsto ([h(L_i)])_i\,,
\end{align}
which is well-defined because $\omega(\gamma_i,\gamma_i)=0$ implies that $u_{L_i}=\iota_{\gamma_i}\omega \in \gamma_i^{\perp}$.
As the evaluation map is affine in restriction to each cone of $\cM_{\omega,\boldsymbol{\gamma}}^\trop$, the moduli space $\cM_{\omega,\boldsymbol{\gamma}, \mathbf{A}}^\trop$ is naturally a finite polyhedral complex. For every face $\sigma$ of $\cM_{\omega,\boldsymbol{\gamma}, \mathbf{A}}^\trop$, the relative interior $\Int(\sigma)$ of $\sigma$
parametrizes tropical curves of a given type $(\Gamma_\sigma, \bar{u}_\sigma)$ and matching $\mathbf{A}$.

For every subset $I \subset \{1,\dots,r\}$, we denote $\boldsymbol{\gamma}_I:=(\gamma_i)_{i \in I}$ and we consider the $\boldsymbol{\gamma}_I$-constraint $\mathbf{A}_I:=(A_i)_{i\in I}$. 

\begin{definition}
\label{def_general_constraints}
A $\boldsymbol{\gamma}$-constraint $\mathbf{A}=(A_1,\dots,A_r)$ is \emph{general} if the following conditions are satisfied:
\begin{itemize}
    \item[(i)] $A_i \neq A_j$ for $i \neq j$,
    \item[(ii)] for every $I\subset \{1,\dots,r\}$, $\dim \cM_{\omega,\boldsymbol{\gamma}_I, \mathbf{A}_I}^\trop \leq d-2$,
    \item[(iii)] for every $I \subset \{1,\dots,r\}$, if $\sigma$ is a face of $\cM_{\omega,\boldsymbol{\gamma}_I, \mathbf{A}_I}^\trop$, then $\dim \sigma=d-2$ if and only if $\Gamma_\sigma$ is trivalent.
\end{itemize}

\end{definition}

\begin{lemma} \label{lem_dimension}
The set of general $\boldsymbol{\gamma}$-constraints is an open dense subset in the space $\prod_{i=1}^r M_\RR/(\gamma_i^\perp)_\RR \simeq \RR^r$ of $\boldsymbol{\gamma}$-constraints.
\end{lemma}

\begin{proof}
As there are only finitely many subsets $I \subset \{1,\dots,r\}$, it is enough to prove that the set of $\mathbf{A}$ satisfying Definition \ref{def_general_constraints}(ii)-(iii) for a given $I$ is open and dense.
Let $\tau=(\Gamma,\bar{u})$ be the type of 
an $(\omega,\boldsymbol{\gamma}_I)$-marked tropical curve. As $\Gamma$ is of genus $0$, $h$ is not superabundant by \cite[Proposition 2.5]{MR}. In this case, \cite[Proposition 2.10]{MR} implies that, for $\mathbf{A}$ in an open dense subset of the space of constraints, the space of $(\omega,\boldsymbol{\gamma})$-marked tropical curves of type $\tau$ matching $\mathbf{A}$ is a non-empty open convex polyhedron in a vector space of dimension 
\[ d+\overline{e}-\sum_{i\in I} \codim (A_i)=d+\overline{e}-|I|\,,\]
where by \cite[\S 2, Eq.\! (2)]{MR}, using that $\Gamma$ has $|I|+1$ legs,
\[ \overline{e}=(|I|+1)-3-\mathrm{ov}(\Gamma)\,,\]
where the overvalence $\mathrm{ov}(\Gamma)$ is given by \eqref{eq_overvalence} and is a non-negative integer.
Therefore, the space of $(\omega,\boldsymbol{\gamma}_I)$-marked tropical curves of type $\tau$ matching $\mathbf{A}$ is of dimension at most 
\[ d+(|I|+1)-3-\mathrm{ov}(\Gamma)-|I|=d-2-\mathrm{ov}(\Gamma)\,,\]
and of dimension $d-2$ if and only if $\mathrm{ov}(\Gamma)=0$, that is, if $\Gamma$ is trivalent.
\end{proof}

\subsection{Properties of $(\omega,\boldsymbol{\gamma})$-marked tropical curves}

Let $\mathbf{A}$ be a $\boldsymbol{\gamma}$-constraint and let $\sigma$ be a face of the moduli space
$\cM_{\omega,\boldsymbol{\gamma},\mathbf{A}}^\trop$ of $(\omega,\boldsymbol{\gamma})$-marked tropical curves in $M_\RR$ matching $\mathbf{A}$.
For each vertex $v$ of $\Gamma_\sigma$, we obtain a polyhedron 
$\foj_v^\sigma$ in $M_\RR$ defined by 
\begin{equation}
\label{Eq:jv}
\foj_v^\sigma \coloneqq \{h(v)\,|\, h\in \Int(\sigma)\}^{\cl}\subseteq M_{\RR} \,,   
\end{equation}
and for each edge or leg $E$ of $\Gamma_\sigma$, a polyhedron $\fod_E^\sigma$ in $M_\RR$ defined by
\begin{equation}
\label{Eq:Ev}
    \fod_E^\sigma \coloneqq \left(\bigcup_{h\in \Int(\sigma)} h(E)\right)^{\cl} \subseteq M_{\RR}\,,
\end{equation}
see Figure \ref{fig4}.
For every polyhedron $\fod$ in $M_\RR$, we denote by $\Lambda_{\fod} \subset M$ the subspace of integral tangent vectors to $\fod$.

\begin{figure}[h]
\center{\includegraphics{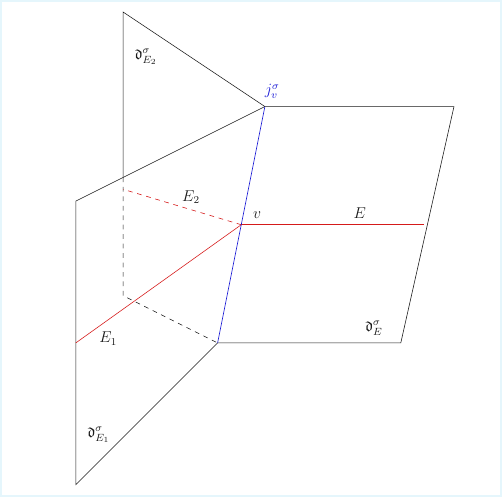}}
\caption{Polyhedra $\foj_v^\sigma$ and $\fod_E^\sigma$ for $d=3$.}
\label{fig4}
\end{figure}

By Definition \ref{Def affine constraint}, the affine hyperplanes $A_i$ of $\mathbf{A}$ have equations of the form $\gamma_i=\epsilon_i$ for some $\epsilon_i \in \RR$.
For every edge or leg $E$ of $\Gamma_\sigma$, we denote by $A_E$ the affine hyperplane in $M_\RR$ with equation 
\begin{equation}\label{eq_AE}
\gamma_E= \sum_i \epsilon_i\,,
\end{equation}
where the sum is over the indices $1\leq i \leq r$ such that the leg $L_i$ is a descendant of $E$. We have $\dim A_E= d-1$ because $\gamma_E\neq 0$ by Lemma \ref{lem_class}.

\begin{lemma} \label{lem_prep}
Let $\mathbf{A}$ be a $\boldsymbol{\gamma}$-constraint and let
$\sigma$ be a face of $\cM_{\omega,\boldsymbol{\gamma},\mathbf{A}}^\trop$.
Then,
\begin{itemize}
    \item[(i)] for every edge or leg $E$ of $\Gamma_\sigma$, we have $\fod_E^\sigma \subset A_E$, $\Lambda_{\fod_E^\sigma} \subset \gamma_E^{\perp}$, and in particular $\dim \fod_E^{\sigma} \leq d-1$,
    \item[(ii)] for every vertex $v$ of $\Gamma_\sigma$, we have $\foj_v^\sigma \subset \bigcap_{i=1}^k A_{E_i}$, and $\Lambda_{\foj_v^\sigma} \subset \bigcap_{i=1}^k \gamma_{E_i}^{\perp}$ where $E_1,\dots,E_k$ are the children edges of $v$.
\end{itemize}
\end{lemma}

\begin{proof}
We prove the result by induction following the flow on $\Gamma_\sigma$ starting at the leaves and ending at the root.

For the initial step of the induction, we consider a leg $E=L_i$ of $\Gamma_\sigma$. Every tropical curve $h \in \Int(\sigma)$ matches the $\boldsymbol{\gamma}$-constraint $\mathbf{A}$, and so $h(L_i) \subset A_i$.
It follows that $\fod_{L_i}^\sigma \subset A_i$ and so $\Lambda_{\fod_{L_i}^\sigma}\subset \gamma_i^{\perp}$ because $\Lambda_{A_i}=\gamma_i^{\perp}$.

For the induction step, let $v$ be a vertex of $\Gamma_\sigma$, $E$ the parent edge of $v$ and $E_1, \dots, E_k$ the children edges of $v$. We assume that Lemma \ref{lem_prep} holds for $E_1, \dots, E_k$, and we have to show that Lemma \ref{lem_prep} holds for $v$ and $E$. For every tropical curve $h \in \Int(\sigma)$, we have $h(v)\subset \cap_{i=1}^k h(E_i)$, and so $\fod_v^\sigma \subset \cap_{i=1}^k \fod_{E_1}^{\sigma}$.
By the induction hypothesis, we have
$\fod_{E_i}^\sigma \subset A_{E_i}$ and 
$\Lambda_{\fod_{E_i}^\sigma} =\gamma_{E_i}^{\perp}$ for all $1\leq i\leq k$, and so we conclude that 
$\foj_v^\sigma \subset \cap_{i=1}^k A_{E_i}$ and
$\Lambda_{\foj_v^\sigma} \subset \cap_{i=1}^k \gamma_{E_i}^\perp$. Similarly, we have $h(E) \subset h(v)+\RR_{\geq 0} u_E$. Using Lemma \ref{lem_class}, this implies that $\fod_E^\sigma \subset \foj_{v}^\sigma -
\RR_{\geq 0} \iota_{\gamma_E} \omega$. As we have already showed that $\Lambda_{\foj_v^\sigma} \subset \cap_{i=1}^k \gamma_{E_i}^{\perp}$, we conclude that
$\fod_E^\sigma \subset \cap_{i=1}^k A_{E_i}-
\RR_{\geq 0} \iota_{\gamma_E} \omega \subset A_E$ and 
$\Lambda_{\fod_E^\sigma} \subset \cap_{i=1}^k \gamma_{E_i}^{\perp}-
\RR_{\geq 0} \iota_{\gamma_E} \omega \subset \gamma_E^{\perp}$. Indeed, we have $\cap_{i=1}^k \gamma_{E_i}^{\perp}\subset \gamma_E^{\perp}$ because $\gamma_E=\sum_{i=1}^k \gamma_{E_i}$, and we 
have $\iota_{\gamma_E}\omega \in \gamma_E^{\perp}$ because $\omega(\gamma_E,\gamma_E)=0$.
\end{proof}

Let $(\Gamma,\bar{u})$ be a tropical type and let $v \in V(\Gamma)$ be a vertex. 
We denote by $T_v \Gamma$ the $\RR$-span in $M_\RR$ of all the weighted directions $u_E$ of the edges or legs of $\Gamma$ adjacent to $v$.
By the balancing condition in Definition \ref{def_tropical_curve}(ii), it is also the span of the weighted directions $u_{E_i}$ of the children edges or legs $E_i$ of $v$.

\begin{lemma} \label{lem_locally_planar}
Let $\mathbf{A}$ be a $\boldsymbol{\gamma}$-constraint and let
$\sigma$ be a face of $\cM_{\omega,\boldsymbol{\gamma},\mathbf{A}}^\trop$. Let $v \in V(\Gamma_\sigma)$ be a vertex such that $\dim \fod^\sigma_{E} =d-1$, where $E$ is the unique parent edge of $v$. Then, $\dim T_v \Gamma_\sigma \leq 2$.
\end{lemma}

\begin{proof}
Let $E_1,\dots, E_k$ be the children edges of $v$.
If $\dim T_v\Gamma_\sigma >2$, then, as $u_{E_k}=-\iota_{\gamma_{E_k}}\omega$ by Lemma \ref{lem_class}, the $\RR$-span in $N_\RR$ of the vectors $\gamma_{E_k}$ is also $>2$, and so $\dim \cap_{i=1}^k \gamma_{E_k}^{\perp}<d-2$. By Lemma \ref{lem_prep}, we have $\foj_v^\sigma \subset \cap_{i=1}^k \gamma_{E_i}^{\perp}$, and so $\dim \foj_v^\sigma <d-2$. Finally, as $\fod_E^\sigma \subset \foj_v^\sigma +\RR_{\geq 0}u_E$, this implies that $\dim \fod_E^\sigma <d-1$, contradiction.   
\end{proof}

The following lemma is similar to  \cite[Lemma 3.2.4]{KS} and 
\cite[Lemma C.2]{GHKK}.

\begin{lemma}\label{lem_key}
Let $\mathbf{A}$ be a $\boldsymbol{\gamma}$-constraint and let
$\sigma$ be a face of $\cM_{\omega,\boldsymbol{\gamma},\mathbf{A}}^\trop$. Let $v \in V(\Gamma_\sigma)$ be a  vertex such that 
$\dim T_v \Gamma_\sigma >1$ and
$\dim \fod^\sigma_{E} =d-1$, where $E$ is the unique parent edge of $v$. Then, we have $\dim \fod_{E_1}^\sigma=d-1$, for any child edge $E_1$ of $v$.
\end{lemma}

\begin{proof}
By Lemma \ref{lem_prep}(i), it is enough to prove that $\dim \fod_{E_1}^\sigma \geq d-1$ in order to show that $\dim \fod_{E_1}^\sigma = d-1$.

Assume by contradiction that $\dim \fod_{E_1}^\sigma \leq d-2$ for a child edge or leg $E_1$ of $v$. Denote by $E_2,\dots, E_k$ the other child edges or legs of $v$.
For every tropical curve $h \in \Int(\sigma)$, we have $h(v)\subset \cap_{i=1}^k h(E_i)$, and so $\fod_v^\sigma \subset \cap_{i=1}^k \fod_{E_i}^{\sigma}$. By assumption, we have $\dim \fod_{E_1}^\sigma \leq d-2$. On the other hand, by Lemma \ref{lem_prep}(i), we also have $\dim \fod_{E_i}^\sigma \leq d-1$ for all $2 \leq i\leq k$.
If $\Lambda_{\fod_{E_1}^\sigma}$ is not included in  $\Lambda_{\fod_{E_i}^\sigma}$ for some $2\leq i \leq k$, it follows that $\dim \foj_v^{\sigma} \leq d-3$, and so, as $\fod_E^\sigma \subset \foj_v^\sigma +\RR_{\geq 0}u_E$, we conclude that $\dim \fod_{E}^\sigma \leq d-2$, contradiction. 

Therefore, to end the proof, it is enough to consider the case where $\Lambda_{\fod_{E_1}^\sigma} \subset \cap_{i=2}^k \Lambda_{\fod_{E_i}^\sigma}$ . In this case, the direction $u_{E_1}$ is contained in $\cap_{i=2}^k \Lambda_{\fod_{E_2}^\sigma}$, and so in $\cap_{i=2}^k \gamma_{E_2}^{\perp}$ by Lemma \ref{lem_prep}(i). By Lemma \ref{lem_class}, we have $u_{E_1}=-\iota_{\gamma_{E_1}}\omega$, and so $u_{E_1}\in 
\cap_{i=2}^k \gamma_{E_i}^{\perp}$
implies $\omega(\gamma_{E_1},\gamma_{E_i})=0$ for all $2 \leq i \leq k$. 

As we are assuming that $\dim T_v \Gamma_\sigma >1$ and $\dim \fod_E^\sigma=d-1$, we actually have $\dim T_v \Gamma_\sigma =2$ by Lemma \ref{lem_locally_planar}. Hence, there exists $2 \leq i \leq k$ such that $u_{E_1}$ and $u_{E_i}$ generate $T_v 
\Gamma_\sigma$. Hence, $\omega(\gamma_{E_1},\gamma_{E_i})=0$ implies that $T_v \Gamma_\sigma \subset \gamma_{E_1}^{\perp} \cap \gamma_{E_i}^{\perp}$.
In particular, we obtain that $u_E \in \gamma_{E_1}^{\perp}\cap \gamma_{E_i}^{\perp}$. On the other hand, by Lemma \ref{lem_prep}(ii), we also have $\Lambda_{\foj_v^\sigma} \subset \gamma_{E_1}^{\perp} \cap \gamma_{E_i}^{\perp}$. We conclude that 
$\fod_E^\sigma \subset \foj_v^\sigma +\RR_{\geq 0}u_{E} \subset \gamma_{E_1}^{\perp} \cap \gamma_{E_i}^{\perp}$. On the other hand, as we are assuming that $u_{E_1}$ and $u_{E_i}$ span the 2-dimensional space $T_v \Gamma_\sigma$, $u_{E_1}$ and $u_{E_i}$ are not proportional, and so the charges $\gamma_{E_1}$ and $\gamma_{E_i}$ are also not proportional. Hence, the hyperplanes $\gamma_{E_1}^\perp$ and $\gamma_{E_i}^{\perp}$ are distinct, and so $\dim \fod_E^\sigma \leq d-2$, contradiction. 
\end{proof}

\begin{lemma} \label{lem_prep1}
Let $\mathbf{A}$ be a $\boldsymbol{\gamma}$-constraint and let
$\sigma$ be a face of $\cM_{\omega,\boldsymbol{\gamma},\mathbf{A}}^\trop$
such that $\dim \sigma=d-2$ and $\dim \fod_{L_{\mathrm{out}}}^\sigma =d-1$.
Then, we have $\dim T_v \Gamma_\sigma=2$ for all vertices $v \in V(\Gamma_\sigma)$, and $\dim \fod_{E}^\sigma=d-1$ for all edges or legs $E \in E(\Gamma_\sigma) \cup L(\Gamma_\sigma)$.
\end{lemma}

\begin{proof}
We first show that $\dim T_v \Gamma_\sigma >1$ for all $v \in V(\Gamma_\sigma)$.
Assume by contradiction that there exists $v \in V(\Gamma_\sigma)$ such that $\dim T_v \Gamma_\sigma =1$. Then, there exists at least one vertex $v' \in \Gamma_\sigma$ such that $\dim T_{v'} \Gamma_\sigma =1$ and such that $\dim T_{v''}\Gamma_\sigma >1$ for all vertices $v'' \neq v'$ on the unique path in $\Gamma_\sigma$ from $v'$ to the root. 

If $v'$ is the unique vertex adjacent to $L_{\mathrm{out}}$, then, as $\dim T_{v'}\Gamma_\sigma=1$, one can move $h(v')$ is the direction of $h(L_{\mathrm{out}})$, and so $\foj_{v'}^\sigma = \fod_{L_{\mathrm{out}}}^\sigma$, so $\dim \foj_{v'}^\sigma=d-1$. This contradicts the assumption that $\dim \sigma=d-2$ because $\foj_{v'}^\sigma$ is the image of $\sigma$ by the evaluation map $h \mapsto h(v')$.

If $v'$ is not adjacent at $L_{\mathrm{out}}$, then $v'$ admits a unique parent vertex $v$. Denote by $E_1$ the edge connecting $v'$ and $v$. As $\dim T_{v''}\Gamma_\sigma >1$ for every vertex $v''$ on the path from $v$ to the root, applying Lemma \ref{lem_key} iteratively from the root down to $v$  shows that the assumption $\dim \fod_{L_{\mathrm{out}}}^\sigma=d-1$ implies $\dim \fod_{E_1}^\sigma  =d-1$, and so $\dim \foj_v^{\sigma}\geq d-2$.
On the other hand, as $\dim T_{v'}\Gamma_\sigma=1$, one can move $h(v')$ is the direction of $h(E_1)$ without moving $h(v)$, and so $\dim \sigma \geq 1+\dim \foj_v^{\sigma} \geq d-1$, contradiction with the assumption that $\dim \sigma=d-2$.

Once one knows that $\dim T_v \Gamma_\sigma >1$ for all $v \in V(\Gamma_\sigma)$, applying Lemma \ref{lem_key} iteratively starting from the root  shows that the assumption $\dim \fod_{L_{\mathrm{out}}}^\sigma=d-1$ implies $\dim \fod_{E}^\sigma  =d-1$ for all edges or legs $E$. We conclude that $\dim T_v \Gamma_\sigma=2$ for all $v \in V(\Gamma_\sigma)$ by Lemma \ref{lem_locally_planar}.
\end{proof}

\begin{theorem} \label{thm_trop}
Let $\mathbf{A}$ be a $\boldsymbol{\gamma}$-constraint and let $\sigma$ be a $(d-2)$-dimensional face of $\cM_{\omega,\boldsymbol{\gamma},\mathbf{A}}^\trop$
such that $\dim \fod_{L_{\mathrm{out}}}^\sigma =d-1$.
Then, 
\begin{itemize}
\item[(i)]for every edge or leg $E$ of $\Gamma_\sigma$, we have $\dim \fod_{E}^\sigma=d-1$ and $\Lambda_{\fod_{E}^\sigma} =\gamma_E^{\perp}$,
\item[(ii)] for every vertex $v\in V(\Gamma_\sigma)$, we have $\dim T_v \Gamma_\sigma=2$, $\dim \foj_v^\sigma=d-2$ and $\Lambda_{\foj_v^\sigma}=\cap_{i=1}^k \gamma_{E_i}^\perp$ where $E_1,\dots, E_k$ are the children edges of $v$,
\item[(iii)] for every edge or leg $E$ of $\Gamma_\sigma$ with parent vertex $v$, we have $u_E \notin \Lambda_{\foj_v^\sigma}$.
\end{itemize}
\end{theorem}

\begin{proof}
(i) follows from Lemma \ref{lem_prep1} and Lemma \ref{lem_prep}(i). 
The first part of (ii)
also follows from Lemma \ref{lem_prep1}. Moreover, denoting by $E$ the parent edge of $v$, we have $\fod_{E}^\sigma=\foj_v^\sigma +\RR_{\geq 0} u_E$, and so the last part of (ii) follows from (i) and Lemma \ref{lem_prep}(ii). Finally, for (iii), if one had $u_E \in \Lambda_{\foj_v^\sigma}$, one would have $\Lambda_{\fod_E^\sigma}=\Lambda_{\foj_v^\sigma}$, in contradiction with (i) and (ii).
\end{proof}

\subsection{Tropical multiplicities and coefficients}
\subsubsection{Tropical multiplicities}
Let $\mathbf{A}$ be a $\boldsymbol{\gamma}$-constraint and let $\sigma$ be a face of $\cM_{\omega,\boldsymbol{\gamma},\mathbf{A}}^\trop$.
For every leg $L\in L(\Gamma_\sigma)$, there exists a  unique vertex $\partial^- L$ adjacent to $L$.
For every edge $E \in E(\Gamma_\sigma)$, we denote by $\partial^+ E$ the parent vertex of $E$ and by $\partial^-E$ the child vertex of $E$. 

\begin{definition} \label{def_gluing_map}
Let $\mathbf{A}$ be a $\boldsymbol{\gamma}$-constraint. 
For every face $\sigma$ of $\cM_{\omega,\boldsymbol{\gamma},\mathbf{A}}^\trop$, the
 \emph{gluing map}
$\mathrm{gl}_\sigma$ is the map
\begin{align}\label{eq_gl} \mathrm{gl}_\sigma: \prod_{v \in V(\Gamma_\sigma)} M &\longrightarrow \prod_{e \in E(\Gamma_\sigma)} M/\Z u_E \times \prod_{i=1}^r M/\gamma_i^\perp
\\ \nonumber
(m_v)_{v \in V(\Gamma_\sigma)} &\longmapsto ((m_{\partial^+ E}-m_{\partial^- E})_{E\in E(\Gamma_\sigma)}, m_{\partial^- E}) \,.
\end{align}
\end{definition}

By construction, the kernel $T_\sigma:=\ker\, \mathrm{gl}_\sigma$ is the integral tangent space to $\sigma$ (see \cite[Proposition 2.10]{MR}): an element $(m_v)_{v\in V(\Gamma_\sigma)} \in \prod_{v\in V(\Gamma_\sigma)} M$ contained in the kernel of $\mathrm{gl}_\sigma$ defines an integral way to move the vertices of $\Gamma_\sigma$ in $M_\RR$ while preserving the directions of the edges and the matching of the constraint $\mathbf{A}$.

\begin{lemma} \label{lem_finite}
Let $\mathbf{A}$ be a general $\boldsymbol{\gamma}$-constraint as in Definition \ref{def_general_constraints}. For every $(d-2)$-dimensional face  $\sigma$ of $\cM_{\omega,\boldsymbol{\gamma},\mathbf{A}}^\trop$, the cokernel $\coker \mathrm{gl}_\sigma$ of the gluing map $\mathrm{gl}_\sigma$ is finite. 
\end{lemma}

\begin{proof}
It is enough to show that $\mathrm{gl}_\sigma$ is surjective after tensoring with $\QQ$. This follows from a dimension count: the domain of $\mathrm{gl}_\sigma \otimes \QQ$ is of dimension $d |V(\Gamma_\sigma)|$, its codomain is of dimension $(d-1)|E(\Gamma_\sigma)|+r$, and its kernel is of dimension $d-2$.
Hence, 
\[ \dim \coker\,\mathrm{gl}_\sigma \otimes \QQ=(d-1)|E(\Gamma_\sigma)|+r-d|V(\Gamma_\sigma)|+d-2\]
\[=d(|E(\Gamma_\sigma)|-|V(\Gamma_\sigma)|+1)-|E(\Gamma_\sigma)|+r-2\,.\]
As $\Gamma_\sigma$ is of genus $0$, we have $|E(\Gamma_\sigma)|-|V(\Gamma_\sigma)|+1=0$, and so 
\begin{equation} \label{eq_coker}
\dim \coker\,\mathrm{gl}_\sigma \otimes \QQ =-|E(\Gamma_\sigma)|+r-2\,.\end{equation}
Moreover,
as $\mathbf{A}$ is general and $\dim \sigma=d-2$, the graph $\Gamma_\sigma$ is trivalent by Definition \ref{def_general_constraints}(iii),
with $r+1$ legs, and so we have $3|V(\Gamma_\sigma)|=2|E(\Gamma_\sigma)|+r+1$. Combining these two relations, we obtain \[ 3(|E(\Gamma_\sigma)|+1)=2|E(\Gamma_\sigma)|+r+1\,,\]
and so $|E(\Gamma_\sigma)|=r-2$, which implies that $\dim \coker\,\mathrm{gl}_\sigma \otimes \QQ=0$ by \eqref{eq_coker}.
\end{proof}

\begin{definition}\label{def_trop_mult}
Let $\mathbf{A}$ be a general $\boldsymbol{\gamma}$-constraint. For every $(d-2)$-dimensional face  $\sigma$ of $\cM_{\omega,\boldsymbol{\gamma},\mathbf{A}}^\trop$, the \emph{tropical multiplicity} $N_\sigma^\trop$ is the cardinality of the cokernel of the gluing map $\mathrm{gl}_\sigma$, which is finite by Lemma \ref{lem_finite}:
\begin{equation} \label{eq_N_trop} N_\sigma^\trop:=|\coker\, \mathrm{gl}_\sigma|\,.
\end{equation}
\end{definition}

\subsubsection{Tropical coefficients}
\label{sec_trop_coeff}

Let $\mathbf{A}$ be a $\boldsymbol{\gamma}$-constraint and $\sigma$ be a $(d-2)$-dimensional face of $\cM_{\omega,\boldsymbol{\gamma},\mathbf{A}}^\trop$ such that $\dim \fod_{L_{\mathrm{out}}}^\sigma=d-1$.
Let $p : T_\sigma \longrightarrow M$
be the composition of the inclusion $T_\sigma=\ker\, \mathrm{gl}_\sigma \subset \prod_{v \in V(\Gamma_\sigma)} M$ with the projection $\prod_{v \in V(\Gamma_\sigma)} M \rightarrow M$ on the factor corresponding to the vertex $v_{\mathrm{out}} \in V(\Gamma_\sigma)$ adjacent to $L_{\mathrm{out}}$. In other words, $p$ is the integral derivative of the evaluation map 
\begin{align*} \sigma &\longrightarrow  \foj_{v_{\mathrm{out}}}^\sigma \\
h &\longmapsto h(v_{\mathrm{out}})\,.
\end{align*}
This evaluation map is an affine map which is surjective by definition.
By Theorem \ref{thm_trop}(ii), we have $\dim \foj_{v_{\mathrm{out}}}^\sigma =d-2$, and so, as we also have $\dim \sigma=d-2$, this evaluation map is in fact bijective, and we have an
inclusion $p(T_\sigma) \subset \Lambda_{\foj_{v_{\mathrm{out}}}^\sigma }$ of finite index. 
By Theorem \ref{thm_trop}(iii), we have $u_{L_{\mathrm{out}}} \notin \Lambda_{\foj_{v_{\mathrm{out}}}^\sigma}$, and so the lattice 
\[ \cL:= p(T_\sigma)+\Z u_{L_{\mathrm{out}}} \]
is of rank $d-1$ and the natural inclusion $\cL \subset \Lambda_{\fod_{L_{\mathrm{out}}}^\sigma}$ is of finite index.

\begin{definition}
Let $\mathbf{A}$ be a $\boldsymbol{\gamma}$-constraint and $\sigma$ be a $(d-2)$-dimensional face of $\cM_{\omega,\boldsymbol{\gamma},\mathbf{A}}^\trop$ such that $\dim \fod_{L_{\mathrm{out}}}^\sigma=d-1$. The \emph{tropical coefficient} $k_\sigma$ is the index of the inclusion $\cL \subset \Lambda_{\fod_{L_{\mathrm{out}}}^\sigma}$, that is,
\begin{equation}
\label{eq_coeff}
k_{\sigma}:=|\Lambda_{\fod_{L_{\mathrm{out}}}^\sigma}/\cL|\,.\end{equation}
\end{definition}

Note that $\Lambda_{\fod_{L_{\mathrm{out}}}^\sigma}$ is also equal to the saturation $\cL^{\sat}$ of $\cL$ in $M$, and so one also have 
$k_\sigma=|\cL^\sat/\cL|$.

\subsection{Product formula for the tropical multiplicities}
\label{subsec: product formula}

In this section, we prove in Lemma \ref{lem: product for Ntrop} and Lemma \ref{lem: product2 for Ntrop}
two product formulas for the tropical multiplicity $N_\sigma^\trop$ introduced in Definition \ref{def_trop_mult}.

Let $\mathbf{A}$ be a general $\boldsymbol{\gamma}$-constraint and $\sigma$ a $(d-2)$-dimensional face of $\cM_{\omega,\boldsymbol{\gamma},\mathbf{A}}^\trop$ such that $\dim \fod_{L_{\mathrm{out}}}^\sigma =d-1$. By Definition \ref{def_general_constraints}(iii) of a general $\boldsymbol{\gamma}$-constraint, the graph $\Gamma_\sigma$ is trivalent.
In particular, every vertex $v \in \Gamma_\sigma$ has two children edges or legs, that we denote by $E_{1,v}$ and $E_{2,v}$, and one parent edge $E_{\mathrm{out},v}$.
We denote by $I_{k,v}$ (resp.\ $I_{\mathrm{out},v}$) the subset of $\{1,\dots,r\}$ consisting of indices $1\leq i \leq r$ such that $L_i$ is a descendant leg of $E_{k,v}$
(resp.\ $E_{\mathrm{out},v}$).

For $k=1, 2$, if $E_{k,v}$ is an edge, we denote by 
$\tau_{k,v}$ the type of tropical curves obtained from tropical curves $h \in \Int(\sigma)$
by removing the vertex $v$ and extending the connected component containing $E_{k,v}$ to infinity.   
We denote by $\sigma_{k,v}$ the face of $\cM^\trop_{\omega,\boldsymbol{\gamma}_{I_{k,v}}, \mathbf{A}_{I_{k,v}}}$ whose relative interior parametrizes tropical curves of type $\tau_{k,v}$ matching $\mathbf{A}_{I_{k,v}}$.
Similarly, we denote by $\tau_{\mathrm{out},v}$ the type of tropical curves obtained from tropical curves $h \in \Int(\sigma)$
by cutting the edge (or leg ) $E_{\mathrm{out},v}$ and extending the connected component containing $E_{\mathrm{out},v}$ to infinity.   
We denote by $\sigma_{\mathrm{out},v}$
the face of $\cM^\trop_{\omega,\boldsymbol{\gamma}_{I_{\mathrm{out},v}}, \mathbf{A}_{I_{\mathrm{out},v}}}$ whose relative interior parametrizes tropical curves of type $\tau_{\mathrm{out},v}$ matching $\mathbf{A}_{I_{\mathrm{out},v}}$. By a small abuse of notation, we still denote by $E_{k,v}$ and $E_{\mathrm{out},v}$ the new legs of $\Gamma_{\sigma_{k,v}}$ and $\Gamma_{\sigma_{\mathrm{out},v}}$ obtained by cutting and extending the edges
$E_{k,v}$ and $E_{\mathrm{out},v}$ of $\Gamma_\sigma$.

\begin{lemma} \label{lem_dim_new}
Let $\mathbf{A}$ be a general $\boldsymbol{\gamma}$-constraint and $\sigma$ a $(d-2)$-dimensional face of $\cM_{\omega,\boldsymbol{\gamma},\mathbf{A}}^\trop$ such that $\dim \fod_{L_{\mathrm{out}}}^\sigma =d-1$.  Then, for every vertex $v \in V(\Gamma)$, we have 
\begin{itemize}
    \item[(i)] for $k=1,2$, if $E_{k,v}$ is an edge, then $\dim \sigma_{k,v}=d-2$ and $\dim \fod_{E_{k,v}}^{\sigma_{k,v}}=d-1$.
    \item[(ii)] $\dim \sigma_{\mathrm{out},v}=d-2$ and $\dim \fod_{E_{\mathrm{out},v}}^{\sigma_{\mathrm{out},v}}=d-1$.
\end{itemize}
\end{lemma}
\begin{proof}
As the graph $\Gamma_{\sigma_{k,v}}$ is cut out form the trivalent graph $\Gamma_\sigma$, it is also trivalent, and so $\dim \sigma_{k,v}=d-2$ by Definition \ref{def_general_constraints}(iii) of a general $\boldsymbol{\gamma}$-constraint. In particular, $\dim \fod_{E_{k,v}}^{\sigma_{k,v}}\leq d-1$. But we also have 
$\fod_{E_{k,v}}^{\sigma} \subset \fod_{E_{k,v}}^{\sigma_{k,v}}$ and $\dim \fod_{E_{k,v}}^{\sigma}=d-1$ by Theorem \ref{thm_trop}(i), and so we conclude that $\dim \fod_{E_{k,v}}^{\sigma_{k,v}}= d-1$. This proves the case (i) of Lemma \ref{lem_dim_new}.
The proof of case (ii) is identical. 
\end{proof}
 
If $E_{k,v}$ is an edge, then by Lemma \ref{lem_dim_new}(i), $\sigma_{k,v}$ satisfies the assumptions of \S \ref{sec_trop_coeff}. In particular, following \S \ref{sec_trop_coeff}, we obtain a map $p_{k,v}\colon T_{\sigma_{k,v}} \rightarrow M$, a rank $(d-1)$-lattice 
\[ \cL_{k,v}:= p_{k,v}(T_{\sigma_{k,v}})+\Z u_{E_{k,v}} \subset M\,,\]
of finite index in $\Lambda_{\fod^{\sigma_{k,v}}_{E_{k,v}}}$, and a coefficient
\begin{equation}
\label{eq_coeff_1}
k_{\sigma_{k,v}}=|\Lambda_{\fod^{\sigma_{k,v}}_{E_{k,v}}}/\cL_{k,v}|=|\cL_{k,v}^{\sat}/\cL_{k,v}|\,.\end{equation}
If $E_{k,v}$ is a leg $L_i$ of $\Gamma_\sigma$ for some $1\leq i\leq r$, then we define 
\[ \cL_{k,v}:=\gamma_i^\perp \subset M\,.\]
By Theorem \ref{thm_trop}(ii), we have $\dim T_v \Gamma_\sigma=2$. In particular, the lattice 
$\cL_{1,v}+\cL_{2,v}$ is of rank $d$, and so of finite index in $M$.

Similarly by Lemma \ref{lem_dim_new}(ii), $\sigma_{\mathrm{out},v}$ satisfies the assumptions of \S \ref{sec_trop_coeff}. In particular, we obtain a map $p_{\mathrm{out},v}\colon T_{\sigma_{\mathrm{out},v}} \rightarrow M$, a rank $(d-1)$-lattice 
\[ \cL_{\mathrm{out},v}:= p_{\mathrm{out},v}(T_{\sigma_{\mathrm{out},v}})+\Z u_{E_{\mathrm{out},v}} \subset M\,,\]
of finite index in $\Lambda_{\fod^{\sigma_{\mathrm{out},v}}_{E_{\mathrm{out},v}}}$, and a coefficient
\begin{equation}
\label{eq_coeff_2}
k_{\sigma_{\mathrm{out},v}}=|\Lambda_{\fod^{\sigma_{\mathrm{out},v}}_{E_{\mathrm{out},v}}}/\cL_{\mathrm{out},v}|=|\cL_{\mathrm{out},v}^{\sat}/\cL_{\mathrm{out},v}|\,.\end{equation}
As a tropical curve in $\sigma_{\mathrm{out},v}$
is obtained by gluing at $v$ a tropical curve in $\sigma_{1,v}$ with a tropical curve in $\sigma_{2,v}$, we have $p_{\mathrm{out},v}(T_{\sigma_{\mathrm{out},v}})=\cL_{1,v}\cap \cL_{2,v}$, and so 
\begin{equation}\label{eq_cL_out}
\cL_{\mathrm{out},v}:= \cL_{1,v}\cap \cL_{2,v}+\Z u_{E_{\mathrm{out},v}} \subset M\,.\end{equation}

\begin{lemma}
\label{lem: product for Ntrop}
Let $\mathbf{A}$ be a general $\boldsymbol{\gamma}$-constraint and $\sigma$ a $(d-2)$-dimensional face of $\cM_{\omega,\boldsymbol{\gamma},\mathbf{A}}^\trop$ such that $\dim \fod_{L_{\mathrm{out}}}^\sigma =d-1$.
Then, we have
\[ N_\sigma^\trop=\prod_{v \in V(\Gamma_\sigma)}|M/(\cL_{1,v}+\cL_{2,v})|\,.\]
\end{lemma}

\begin{proof}
We prove by induction, following the flow on $\Gamma_\sigma$ starting at the leaves and ending at the root, that, for every vertex $v \in V(\Gamma_\sigma)$, we have
\[ N_{\sigma_{\mathrm{out},v}}^\trop=\prod_{v \in V(\Gamma_{\sigma_{\mathrm{out},v})}}|M/(\cL_{1,v}+\cL_{2,v})|\,.\]

To shorten the notation, denote $\Gamma_{\sigma_k}$, $E_{k}$, $\sigma_k$,
$\sigma_{\mathrm{out}}$, $p_k$, $\cL_k$ for $\Gamma_{\sigma_{k,v}}, E_{k,v}, \sigma_{k,v}, \sigma_{\mathrm{out},v}$, $p_{k,v}$, $\cL_{k,v}$
respectively.
If $E_{1}$ and $E_{2}$ are both edges, we have to show that
\begin{equation}\label{eq_1} N_{\sigma_{\mathrm{out}}}^\trop=|M/(\cL_{1}+\cL_{2})|
N_{\sigma_1}^{\trop} N_{\sigma_2}^\trop \,.
\end{equation}
Denote the domains of the gluing maps $\mathrm{gl}_{\sigma_{\sigma_{\mathrm{out}}}}$, 
$\mathrm{gl}_{\sigma_1}$, $\mathrm{gl}_{\sigma_2}$ of Definition 
\ref{def_gluing_map}
by
$\cM$, $\cM_1$, $\cM_2$
and the codomains by 
$\cK$, $\cK_1$, $\cK_2$.
We have $\cM=\cM_1 \oplus \cM_2 \oplus M_v$, where 
$M_v$ is the factor $M$ in $\cM$ corresponding to the vertex $v$, and 
$\cK=\cK_1 \oplus \cK_2 \oplus M/\Z u_{E_1}\oplus M/\Z u_{E_2}$.
With respect to 
these decompositions, one can write
\[ \mathrm{gl}_{\sigma_{\mathrm{out}}} (m_1,m_2,m_v)
=(\mathrm{gl}_{\sigma_1} (m_1),
\mathrm{gl}_{\sigma_2} (m_2),
m_{v_1}-m_v, m_{v_2}-m_v)\,,\]
where $v_1$ (resp.\ $v_2$) is the vertex of $\Gamma_{\sigma_1}$ (resp.\ 
 $\Gamma_{\sigma_2}$) adjacent to $E_1$ (resp.\ $E_2$).
In other words, we have a commutative diagram 
\[\begin{tikzcd}
0 
\arrow[r]
&
M_v \arrow[r]
\arrow[d,"\psi"]
&
\cM \arrow[r]\arrow[d,"\mathrm{gl}_{\sigma_{\mathrm{out}}}"]
&
\cM_1\oplus \cM_2 \arrow[r] \arrow[d,"\mathrm{gl}_{\sigma_1} \oplus \mathrm{gl}_{\sigma_2}"]
& 0\\
0 \arrow[r]& 
 M/\Z u_{E_1} \oplus M/\Z u_{E_2} \arrow[r] &
\cK \arrow[r]& \cK_1 \oplus \cK_2 \arrow[r]&
0\,,
\end{tikzcd}\]
where the rows are short exact sequences of abelian groups, and where $\psi$ is given by $\psi(m)=(-m,-m)$.
By Theorem \ref{thm_trop}(ii), we have $\dim T_v \Gamma_\sigma=2$, and so the vectors $u_{E_1}$ and $u_{E_2}$ are linearly independent.
This implies that the map $\psi$ is injective. Hence, by the snake lemma, we obtain a long exact sequence
\[ 0 \rightarrow T_{\sigma_{\mathrm{out}}} \rightarrow T_{\sigma_1}\oplus T_{\sigma_2} \xrightarrow{\xi} \coker\, \psi \rightarrow \coker\, \mathrm{gl}_{\sigma_{\mathrm{out}}}
\rightarrow \coker\, \mathrm{gl}_{\sigma_1} \oplus \coker\, \mathrm{gl}_{\sigma_2} \rightarrow 0\,, \]
and so, using \eqref{eq_N_trop},
\begin{equation} \label{eq_2} N_{\sigma_{\mathrm{out}}}^\trop=|\coker\, \xi|
N_{\sigma_1}^{\trop} N_{\sigma_2}^\trop \,.\end{equation}
Moreover, the connecting homomorphism $\xi$ is given by 
$\xi(m_1,m_2)=(\bar{p}_1(m_1),\bar{p}_2(m_2) )$,
where $\bar{p}_k(m_k)$ is the image of $p_k(m_k)$ in $M/\Z u_{E_k}$. As $\coker\, \psi = M/(\Z u_{E_1}+ \Z u_{E_2})$, we deduce that $\coker\, \xi=M/(\cL_{1}+\cL_{2})$, and so \eqref{eq_2} implies \eqref{eq_1}.

Consider now the case when one of $E_1$ or $E_2$ is an edge and the other is a leg. By symmetry, one can assume that $E_1$ is an edge, with adjacent vertex $v_1$ in $\Gamma_{\sigma_1}$, and $E_2$ is a leg $L_i$
 for some $1\leq i\leq r$. 
In this case, we have to show that  
 \begin{equation}\label{eq_11} N_{\sigma_{\mathrm{out}}}^\trop=|M/(\cL_{1}+\cL_{2})|
N_{\sigma_1}^{\trop}\,.
\end{equation}
We have $\cM=\cM_1\oplus M_v$, and 
$\cK=\cK_1 \oplus M/\Z u_{E_1} \oplus M/\gamma_i^\perp$, and so a commutative diagram 
\[\begin{tikzcd}
0 
\arrow[r]
&
M_v \arrow[r]
\arrow[d,"\psi"]
&
\cM \arrow[r]\arrow[d,"\mathrm{gl}_\sigma"]
&
\cM_1\arrow[r] \arrow[d,"\mathrm{gl}_{\sigma_1}"]
& 0\\
0 \arrow[r]& 
 M/\Z u_{E_1} \oplus M/\gamma_i^\perp \arrow[r] &
\cK \arrow[r]& \cK_1 \arrow[r]&
0\,,
\end{tikzcd}\]
where $\psi$ is given by $\psi(m)=(-m,m)$. 
By Theorem \ref{thm_trop}(ii), we have $\dim T_v \Gamma_\sigma=2$, and so $u_{E_1} \notin \gamma_i^\perp$.
Hence, $\psi$ is injective, and so by the snake lemma, we obtain a long exact sequence
\[ 0 \rightarrow T_{\sigma_{\mathrm{out}}} \rightarrow T_{\sigma_1} \xrightarrow{\xi} \coker\, \psi \rightarrow \coker\, \mathrm{gl}_{\sigma_{\mathrm{out}}}
\rightarrow \coker\, \mathrm{gl}_{\sigma_1} \rightarrow 0\,, \]
and so, using \eqref{eq_N_trop},
\begin{equation} \label{eq_3} N_\sigma^\trop=|\coker\, \xi|
N_{\sigma_1}^{\trop} \,.\end{equation}
Moreover, we have $\coker\, \psi = M/(\Z u_{E_1}+\gamma_i^\perp)$ and the connecting homomorphism $\xi$ is given by 
$\xi(m_1)=\bar{p}_1(m_1)$. As $\cL_{1}=p_1(T_{\sigma_1})+\Z u_{E_1}$ and
$\cL_{2} =\gamma_i^\perp$, we deduce that $\coker\, \xi=M/(\cL_{1}+\cL_{2})$, and so \eqref{eq_3} implies \eqref{eq_11}.

Finally, we consider the case where both $E_1$ and $E_2$ are legs
$L_i$ and $L_j$ of $\Gamma$ for some $1 \leq i,j \leq r$. In this case,
we have to show that 
\[N_\sigma^\trop=|M/(\cL_{1}+\cL_{2})|\,.\]
This follows because in this case the gluing map 
 $\mathrm{gl}_{\sigma_{\mathrm{out}}}$ is simply given by the  projection map
\[ M \longrightarrow M/\gamma_i^\perp \oplus M/\gamma_j^\perp\,,\]
and $\cL_{1}=\gamma_i^\perp$ and $\cL_{2}=\gamma_j^\perp$.
\end{proof}

For every non-zero $n\in N$, we denote by $|n|$ the divisibility of $n$ in $N$.

\begin{lemma}\label{lem: product2 for Ntrop}
Let $\mathbf{A}$ be a general $\boldsymbol{\gamma}$-constraint. 
Then, for every $(d-2)$-dimensional face $\sigma$ of $\cM_{\omega,\boldsymbol{\gamma},\mathbf{A}}^\trop$ such that $\dim \fod_{L_{\mathrm{out}}}^\sigma =d-1$, we have 
\begin{equation} \label{eq_product}
k_\sigma N_\sigma^{\trop} = \frac{|\gamma|}{\prod_{i=1}^r |\gamma_i|} \prod_{v \in V(\Gamma_\sigma)} |\omega (\gamma_{E_{1,v}},\gamma_{E_{2,v}})|\,,\end{equation}
where for every vertex $v$ we denote by $E_{1,v}$ and $E_{2,v}$ the two children edges of $v$.
\end{lemma}

\begin{proof}
We prove by induction, following the flow on $\Gamma_\sigma$ starting at the leaves and ending at the root, that, for every vertex $v \in V(\Gamma_\sigma)$, we have
\[ k_{\sigma_{\mathrm{out},v}} N_{\sigma_{\mathrm{out},v}}^\trop=\frac{|\gamma_{E_{\mathrm{out},v}}|}{\prod_{i \in I_v}|\gamma_i|}\prod_{v' \in V(\Gamma_{\sigma_{\mathrm{out},v}})}|\omega (\gamma_{E_{1,v'}},\gamma_{E_{2,v'}})|\,.\]
Given a vertex $v \in V(\Gamma_\sigma)$, we use the same simplified notations $\E_1$, $E_2$, $E_{\mathrm{out}}$ and so on, as in the proof of Lemma \ref{lem: product for Ntrop}.
If $E_1$ and $E_2$ are both edges, we have to show that 
\begin{equation} \label{eq_proof_1}k_{\sigma_{\mathrm{out}}}
N_{\sigma_{\mathrm{out}}}^\trop = \frac{|\gamma_{E_{\mathrm{out}}}|}{|\gamma_{E_1}||\gamma_{E_2}|}
|\omega(\gamma_{E_1},\gamma_{E_2})|k_{\sigma_1}
N_{\sigma_1}^\trop k_{\sigma_2}
N_{\sigma_2}^\trop  \,.
\end{equation}
By Lemma \ref{lem: product for Ntrop}, we have 
\[ N_{\sigma_{\mathrm{out}}}^\trop = |M/(\cL_1+\cL_2)|
N_{\sigma_1}^\trop 
N_{\sigma_2}^\trop  \,,\]
and so we have to show that
\begin{equation}\label{eq_proof_2}
|M/(\cL_1+\cL_2)|=
\frac{k_{\sigma_1}k_{\sigma_2}}{k_{\sigma_{\mathrm{out}}}}
\frac{|\gamma_{E_{\mathrm{out}}}|}{|\gamma_{E_1}||\gamma_{E_2}|} 
|\omega(\gamma_{E_1},\gamma_{E_2})|\,.
\end{equation}
First of all, we have 
\begin{equation}\label{eq_proof_3} |M/(\cL_1+\cL_2)|=|M/(\cL_{1}^\sat+\cL_2^\sat)|\,.
|(\cL_{1}^\sat+\cL_2^\sat)/(\cL_{1}+\cL_2)|\,.
\end{equation}
Then, the exact sequence
\begin{equation}\label{eq_sat}
0 \rightarrow 
\frac{\cL_1^\sat \cap \cL_2^\sat}{\cL_1\cap \cL_2} \rightarrow \cL_1^\sat/\cL_1 \oplus \cL_2^\sat/\cL_2 \rightarrow
\frac{\cL_1^\sat +\cL_2^\sat}{\cL_1+\cL_2} \rightarrow 0\,,\end{equation}
and \eqref{eq_coeff_1} imply that 
\begin{equation}\label{eq_proof_4}
|(\cL_1^\sat+\cL_2^\sat)/(\cL_1+\cL_2)|
=\frac{k_{\sigma_1}k_{\sigma_2}}{|(\cL_1^\sat \cap \cL_2^\sat)/(\cL_1\cap \cL_2)|} \,.\end{equation}
On the other hand, we have $\cL_{\mathrm{out}} =\cL_1 \cap \cL_2 + \Z u_{E_{\mathrm{out}}}$
by \eqref{eq_cL_out}, and so 
\begin{equation}\label{eq_proof_5}
    k_{\sigma_{\mathrm{out}}}=|\cL_{\mathrm{out}}^\sat/\cL_{\mathrm{out}}|
    =|(\cL_1^\sat \cap \cL_2^\sat)/(\cL_1\cap \cL_2)| |\overline{u}_{E_{\mathrm{out}}}|\,,
\end{equation}
where for every $u \in M$, $\overline{u}$ is the image of $u$ in the rank two lattice \[\overline{M}:=M/(\cL_1^\sat \cap \cL_2^\sat)\,.\]
Similarly, we also have 
\begin{equation} \label{eq_proof_6}
|M/(\cL_1^\sat +\cL_2^\sat)|=\frac{|\overline{u}_{E_1} \wedge \overline{u}_{E_2}|}{|\overline{u}_{E_1} ||\overline{u}_{E_2}| }\,.
\end{equation}
Combining \eqref{eq_proof_3}-\eqref{eq_proof_4}-\eqref{eq_proof_5}, we obtain
\[ |M/(\cL_1+\cL_2)|=\frac{k_{\sigma_1}k_{\sigma_2}}{k_{\sigma_{\mathrm{out}}}} 
\frac{|\overline{u}_{E_{\mathrm{out}}}|}{|\overline{u}_{E_1}| |\overline{u}_{E_2}|}
|\overline{u}_{E_1} \wedge \overline{u}_{E_2}|\,.\]
Hence, comparing with \eqref{eq_proof_2}, it remains to show that 
\begin{equation}\label{eq_proof_7}
\frac{|\overline{u}_{E_{\mathrm{out}}}|}{|\overline{u}_{E_1}| |\overline{u}_{E_2}|}
|\overline{u}_{E_1} \wedge \overline{u}_{E_2}|
= \frac{|\gamma_{E_{\mathrm{out}}}|}{|\gamma_{E_1}||\gamma_{E_2}|} 
|\omega(\gamma_{E_1},\gamma_{E_2})|\,.
\end{equation}
Denote $\overline{N}:=(\Z \gamma_{E_1}+\Z \gamma_{E_2})^{\sat}$, it is a rank two saturated sublattice of $N$. As $\cL_1=\gamma_{E_1}^\perp$ and $\cL_2^\perp=\gamma_{E_2}^{\perp}$, the duality between $N$ and $M$ implies that $\overline{N}$ and $\overline{M}$ are naturally dual to each other.
Let $(e_1, e_2)$ be a basis of $\overline{N}$ and $(e_1^\star, e_2^\star)$ the corresponding dual basis of $\overline{M}$. Then, we have 
$\iota_{e_1} \omega= \omega(e_1,e_2) e_2^\star$, and $\iota_{e_2} \omega = -\omega(e_1,e_2)e_1^\star$.
In particular, for every $n=\alpha e_1+\beta e_2 \in \overline{N}$, we have $\overline{\iota_n \omega}=\omega(e_1,e_2)(\alpha e_2^{\star}-\beta e_1^\star)$, and so $|\iota_n \omega|=|\omega(e_1,e_2))| \gcd(\alpha,\beta)=|\omega(e_1,e_2))| |n|$.
As $\overline{u}_{E_1}=\iota_{\gamma_{E_1}}\omega$,
$\overline{u}_{E_2}=\iota_{\gamma_{E_2}}\omega$,
and $\overline{u}_{E_{\mathrm{out}}}=\iota_{\gamma_{E_{\mathrm{out}}}}\omega$, we obtain 
\begin{equation}\label{eq_proof_8}
|\overline{u}_{E_1}|=\omega(e_1,e_2)|\gamma_{E_1}|\,,
|\overline{u}_{E_2}|=\omega(e_1,e_2)|\gamma_{E_2}|\,,
\text{and}\, |\overline{u}_{E_{\mathrm{out}}}|
=\omega(e_1,e_2)|\gamma_{E_{\mathrm{out}}}|\,.\end{equation}
On the other hand, writing $\gamma_{E_1}=ae_1+be_2$
and $\gamma_{E_2}=ce_1+de_2$, we have 
$|\overline{u}_{E_1} \wedge \overline{u}_{E_2}|=
\omega(e_1,e_2)^2 |ad-bc|$, whereas 
$|\omega(\gamma_{E_1},\gamma_{E_2})|=|\omega(e_1,e_2)||ad-bc|$, and so
\begin{equation}\label{eq_proof_9}
    |\overline{u}_{E_1} \wedge \overline{u}_{E_2}|
    =\omega(e_1,e_2)|\omega(\gamma_{E_1},\gamma_{E_2})|\,.
\end{equation}
Combining \eqref{eq_proof_8} and \eqref{eq_proof_9}, we obtain \eqref{eq_proof_7}, and this concludes the proof of \eqref{eq_proof_1}.

If $E_1$ is an edge and $E_2$ is a leg $L_i$ for some $1\leq i \leq r$, then we similarly show that 
\[ k_{\sigma_{\mathrm{out}}}
N_{\sigma_{\mathrm{out}}}^\trop = \frac{|\gamma_{E_{\mathrm{out}}}|}{|\gamma_{E_1}||\gamma_{E_2}|}
|\omega(\gamma_{E_1},\gamma_{E_2})|k_{\sigma_1}
N_{\sigma_1}^\trop \,.\]
The only difference is that in this case, we have $\cL_2=\gamma_{i}^{\perp}$, and so $|\cL_2^\sat/\cL_2|=1$.

Finally, if $E_1$ and $E_2$ are both legs $L_i$ and $L_j$ respectively for some $1\leq i,j \leq r$, then we similarly show that 
\[ k_{\sigma_{\mathrm{out}}}
N_{\sigma_{\mathrm{out}}}^\trop = \frac{|\gamma_{E_{\mathrm{out}}}|}{|\gamma_{E_1}||\gamma_{E_2}|}
|\omega(\gamma_{E_1},\gamma_{E_2})| \]
using that $\cL_1=\gamma_{i}^{\perp}$, $\cL_2=\gamma_j^{\perp}$, and so 
$|\cL_1^\sat/\cL_1|=1$ and $|\cL_2^\sat/\cL_2|=1$.
\end{proof}

\begin{remark} The tropical problem, the coefficient $k_\sigma$, and the tropical multiplicity $N_\sigma^\trop$ only depend on the vectors $\iota_{\gamma_i}\omega$ and on the hyperplanes $\gamma_i^\perp$. In particular, they are invariant under a common rescaling of $\gamma_i$ and $\omega$ of the form $\gamma_i \mapsto t \gamma_i$ and $\omega \mapsto t^{-1} \omega$. As a consistency check, one can verify directly that the right-hand side of \eqref{eq_product} in Lemma \ref{lem: product2 for Ntrop} is invariant under such rescaling: a trivalent tree with $r+1$ legs has $r-1$ vertices, and so the right-hand side of \eqref{eq_product} scales as $\frac{t(t^{-1}  t^2)^{r-1}}{t^r}=1$.
\end{remark}

\subsection{Polyhedral decompositions and tropical multiplicities}
\label{sec_polyh}

Until now, we considered tropical curves in $M_\RR$. In this section, we introduce tropical curves in $M_\RR$ endowed with a polyhedral decomposition.
The following definition can be found in \cite[Def 3.1]{NS}.
\begin{definition}
\label{polyhedral decomposition}
A \emph{polyhedral decomposition} of $M_{\RR}$ is a
covering $\P=\{\Xi\}$ of $M_{\RR}$ by a finite number of strongly
convex polyhedra satisfying the following properties:
\begin{enumerate}
\item[(i)]
If $\Xi \in \P$ and $\Xi' \subset \Xi$
is a face, then $\Xi' \in \P$.
\item[(ii)]
If $\Xi, \Xi' \in \P$, then $\Xi \cap \Xi'$
is a common face of $\Xi$ and $\Xi'$.
\end{enumerate}
\end{definition}

For the remaining of this section, we fix a polyhedral decomposition $\scrP$ of $M_\RR$.

\begin{definition} \label{def_tropical_curve_polyh}
A \emph{parametrized tropical curve in $(M_\RR, \scrP)$} 
is a weighted graph $\Gamma$ together with a proper continuous map $h: \Gamma \rightarrow M_\RR$
satisfying the following conditions:
\begin{itemize}
    \item[(i)] for every edge or leg $E$, the restriction $h|_E$ is an embedding with image contained in an affine line with rational slope,
       \item[(ii)] for every vertex $v\in V(\Gamma)$, the following \emph{balancing condition} holds: Let $E_1,\dots,E_m \in E(\Gamma)\cup L(\Gamma)$ be the edges or legs adjacent to $v$, and let $\bar{u}_i \in M$ be the primitive integral vector emanating from $h(v)$ in the direction of $h(E_i)$, then $\sum_{i=1}^m w(E_i)\bar{u}_i=0$.
       \item[(iii)] for every edge or leg $E$, there exists a cell of $\scrP$ containing $h(E)$,
       \item[(iv)] for every divalent vertex $v\in V(\Gamma)$, denoting by $\boldsymbol{\sigma}(v)$ the smallest cell of $\scrP$ containing $h(v)$, there exists an open subset $U$ of $\Gamma$ containing $v$ such that $h(U) \cap \boldsymbol{\sigma}(v)=\{h(v)\}$.
\end{itemize}
\end{definition}

We define as in \S 1.1 a \emph{tropical curve in $(M_\RR,\scrP)$} as an isomorphism class of parametrized tropical curves in $(M_\RR,\scrP)$. 
Compared with tropical curves in $M_
\RR$ of Definition \ref{def_tropical_curve}, tropical curves in $(M_\RR,\scrP)$ are required to be compatible with the polyhedral decomposition $\scrP$ in the sense of Definition \ref{def_tropical_curve_polyh}(iii), and can have divalent vertices, which are required to satisfy Definition \ref{def_tropical_curve_polyh}(iv).

\begin{definition}
A \emph{tropical type} for $(M_\RR,\scrP)$ is the data $(\Gamma,\bar{u}, \boldsymbol{\sigma})$ of a weighted graph $\Gamma$, a map
$\bar{u} : F(\Gamma)  \rightarrow M$ as in
Definition \ref{def_tropical_type}, and a map $\boldsymbol{\sigma} \colon V(\Gamma) \cup E(\Gamma) \cup L(\Gamma) \rightarrow \scrP$ assigning a cell of $\scrP$ to every vertex, edge or leg of $\Gamma$.
\end{definition}

\begin{definition}\label{def_type_polyh}
The type of a tropical curve $h \colon \Gamma \rightarrow M_\RR$ in $(M_\RR,\scrP)$ is the tropical type $(\Gamma, \bar{u},\boldsymbol{\sigma})$ where for every $(v,E)\in F(\Gamma)$,  $\bar{u}_{v,E}$ is the primitive integral in $M$ emanating from $h(v)$ in the direction of $h(E)$, and for every vertex, edge or leg $x$ of $\Gamma$, $\boldsymbol{\sigma}(x)$ is the smallest cell of $\scrP$ containing $x$.
\end{definition}

One defines as in \S\ref{sec_marked_tropical_curves}-\ref{sec_tropical_constraints} the notion of $(\omega,\boldsymbol{\gamma})$-tropical curves in $(M_\RR,\scrP)$, the notion of $(\omega,\boldsymbol{\gamma})$-tropical curves in $(M_\RR,\scrP)$ matching a $\boldsymbol{\gamma}$-constraint $\mathbf{A}$, and the moduli space $\cM_{\omega,\boldsymbol{\gamma},\mathbf{A}}^\trop(\scrP)$ of $(\omega,\boldsymbol{\gamma})$-tropical curves in $(M_\RR,\scrP)$ matching  $\mathbf{A}$. As the moduli space $\cM_{\omega,\boldsymbol{\gamma},\mathbf{A}}$ of tropical curves in $M_\RR$, the moduli space $\cM_{\omega,\boldsymbol{\gamma},\mathbf{A}}(\scrP)$ is a polyhedral complex: the relative interior $\Int(\sigma)$ of a face $\sigma \in \cM_{\omega,\boldsymbol{\gamma},\mathbf{A}}(\scrP)$ parametrizes $(\omega,\boldsymbol{\gamma})$-tropical curves $h \colon \Gamma_\sigma \rightarrow M_\RR$ in $(M_\RR,\scrP)$ matching
$\mathbf{A}$ and of a given type for $(M_\RR,\scrP)$.

In fact, there is a natural homeomorphism $\cM_{\omega,\boldsymbol{\gamma},\mathbf{A}}(\scrP) \simeq \cM_{\omega,\boldsymbol{\gamma},\mathbf{A}}$ realizing $\cM_{\omega,\boldsymbol{\gamma},\mathbf{A}}(\scrP)$ as a polyhedral refinement of $\cM_{\omega,\boldsymbol{\gamma},\mathbf{A}}$: every face  of $\cM_{\omega,\boldsymbol{\gamma},\mathbf{A}}$ is a union of faces $\cM_{\omega,\boldsymbol{\gamma},\mathbf{A}}(\scrP)$. Indeed, given a tropical curve in $(M_\RR,\scrP)$, we obtain a tropical curve in $M_\RR$ by ``erasing" the divalent vertices, and conversely, given a tropical curve in $M_\RR$ there is a unique minimal way to add divalent vertices so that Definition \ref{def_tropical_curve_polyh}(iii) is satisfied and we obtain a tropical curve in $(M_\RR,\scrP)$.

However, the identification $\cM_{\omega,\boldsymbol{\gamma},\mathbf{A}}(\scrP) \simeq \cM_{\omega,\boldsymbol{\gamma},\mathbf{A}}$ 
does not preserve the integral structures in general: if $\tilde{\sigma}$ is a face of $\cM_{\omega,\boldsymbol{\gamma},\mathbf{A}}(\scrP)$ contained in a face $\sigma$ of $\cM_{\omega,\boldsymbol{\gamma},\mathbf{A}}$
of the same dimension, then there is a natural lattice embedding of integral tangent spaces $T_{\tilde{\sigma}} \subset T_\sigma$
which can have a non-trivial finite cokernel.
The issue is that when adding a divalent vertex on a edge in a family of tropical curves, the lengths of the newly created edges are affine functions on the base of the family which might not be integral with respect to the original integral structure on the base: one needs in general to coarsen the integral structure on the base to make all length functions integral affine.  
In the general context of tropicalizations of stable log maps, this phenomenon is discussed in \cite[\S 4]{johnston2022birational}.
Let $\mathbf{A}$ be a general $\boldsymbol{\gamma}$-constraint, $\sigma$ be a $(d-2)$-dimensional face of $\cM_{\omega,\boldsymbol{\gamma},\mathbf{A}}^\trop$ such that $\dim \fod_{L_{\mathrm{out}}}^\sigma =d-1$, and $\tilde\sigma$ be a $(d-2)$-dimensional face of $\cM_{\omega,\boldsymbol{\gamma},\mathbf{A}}^\trop(\scrP)$
contained in $\sigma$.
We give below an explicit description of the difference between $T_{\Tilde{\sigma}}$ and $T_\sigma$.

Recall from \eqref{Eq:jv}-\eqref{Eq:Ev} that we defined for every vertex $v \in V(\Gamma_\sigma)$ the locus $\foj_v^\sigma$ as the closure in $M_\RR$ of the locus of $h(v)$ for $h\in \Int(\sigma)$, and for every edge or leg $E$ of $\Gamma_\sigma$ the locus $\fod_E^\sigma$ as the closure in $M_\RR$ of the locus of $h(E)$ for $h\in \Int(\sigma)$.
We define similarly $\foj_v^{\tilde{\sigma}}$ and $\fod_E^{\tilde{\sigma}}$ for every vertex $v$ and edge or leg $E$ of $\Gamma_{\tilde{\sigma}}$. To simplify the notations, we set $T_v$ for the integral tangent space to $\foj_v^\sigma$ or $\foj_v^{\tilde{\sigma}}$ depending if $v \in V(\Gamma_\sigma)$ or $v\in V(\Gamma_{\tilde{\sigma}})$, and $T_E$
for the integral tangent space to $\fod_E^{\sigma}$ or $\fod_E^{\tilde{\sigma}}$
depending if $E \in E(\Gamma_\sigma)\cup L(\Gamma_\sigma)$ or $E\in E(\Gamma_{\tilde{\sigma}})\cup V(\Gamma_{\tilde{\sigma}})$.

To compare $T_{\Tilde{\sigma}}$ and $T_\sigma$, we first give an alternative description of $T_\sigma$. Consider the map
\begin{equation}
    \label{Eq:Psi}
    \Psi_{\sigma}:\prod_{v\in V(\Gamma_\sigma)} T_v\times \prod_{E\in E(\Gamma_\sigma)} \Z
\rightarrow \prod_{E\in E(\Gamma_\sigma)} T_E\,,
\end{equation}
given by
\[
\Psi_\sigma \big((m_v)_v,(\ell_E)_E\big)=(m_{\partial^+ E}-m_{\partial^- E}+\ell_E u_E)_E \,.
\]
Then $T_{\sigma}=\ker\Psi_\sigma$. Recall that we also had $T_\sigma=\ker \mathrm{gl}_\sigma$, with the gluing map $\mathrm{gl}_\sigma$ defined by \eqref{eq_gl}, and that we defined the tropical multiplicity 
$N_\sigma^\trop=|\coker \mathrm{gl}_\sigma|$
in Definition \ref{def_trop_mult}. However,
$N_\sigma^\trop$ is in general distinct from $|\coker\Psi_\sigma|$. Rather, using the notations introduced in \S\ref{subsec: product formula}, we have the following lemma:

\begin{lemma}
\label{lem:coker psi}
Let $\mathbf{A}$ be a general $\boldsymbol{\gamma}$-constraint, and $\sigma$ a $(d-2)$-dimensional face of $\cM_{\omega,\boldsymbol{\gamma},\mathbf{A}}^\trop$ such that $\dim \fod_{L_{\mathrm{out}}}^\sigma =d-1$. Then, we have
\[|\coker\Psi_\sigma|=\prod_{v\in V(\Gamma_\sigma)} 
|(\cL_{1,v}^{\sat}+\cL_{2,v}^{\sat})/
(\cL_{1,v}+\cL_{2,v})|\,.\]
\end{lemma}

\begin{proof}
We prove by induction, following the flow on $\Gamma_\sigma$ starting at the leaves and ending at the root, that for every vertex $v \in V(\Gamma_\sigma)$, we have
\[ |\coker \Psi_{\sigma_{\mathrm{out},v}}|
=\prod_{v' \in V(\Gamma_{\sigma_{\mathrm{out}},v})}
|(\cL_{1,v'}^{\sat}+\cL_{2,v'}^{\sat})/
(\cL_{1,v'}+\cL_{2,v'})| \,.\]

Given a vertex $v \in V(\Gamma_\sigma)$, we use the same simplified notations $\E_1$, $E_2$, $E_{\mathrm{out}}$, $\cL_1$, $\cL_2$,
$\cL_{\mathrm{out}}$, and so on, as in the proof of Lemmas \ref{lem: product for Ntrop} and \ref{lem: product2 for Ntrop}.
If $E_1$ and $E_2$ are both edges, we have to show that 
\begin{equation} \label{eq_coker_1}
|\coker \Psi_{\sigma_{\mathrm{out}}}|
=
|(\cL_{1}^{\sat}+\cL_{2}^{\sat})/
(\cL_{1}+\cL_{2})| 
|\coker \Psi_{\sigma_1}||\coker \Psi_{\sigma_2}|\,.
\end{equation}

For $k=1$ and $k=2$, denote
\[
M_k=\prod_{v \in V(\Gamma_{\sigma_k})}T_v \times \prod_{E\in E(\Gamma_{\sigma_k})}
\Z,\quad\quad
K_k=\prod_{E\in E(\Gamma_{\sigma_k})} T_E,
\]
and consider the commutative diagram of exact sequences
\[
\xymatrix@C=30pt
{
0\ar[r]&T_v\times\Z^2\ar[d]_{\theta}\ar[r]& M_1\times M_2\times T_v\times
\Z^2\ar[r]\ar[d]_{{\Psi}_{\sigma_{\mathrm{out}}}}&M_1\times M_2\ar[d]_{\Psi_{\sigma_1}
\times \Psi_{\sigma_2}}\ar[r]&0\\
0\ar[r]&T_{E_1}\times T_{E_2}\ar[r]&
K_1\times K_2\times T_{E_1}\times T_{E_2}\ar[r]&
K_1\times K_2\ar[r]&0\,,
}
\]
where the map $\theta$ is given by 
\[
\theta(m,\ell_1,\ell_2)=(\ell_1 u_{E_1}-m,\ell_2 u_{E_2}-m).
\]
As $\ker \Psi_{\sigma_{\mathrm{out}}}=
T_{\sigma_{\mathrm{out}}}$ and $\ker \Psi_{\sigma_k}=T_{\sigma_k}$, we obtain by the snake lemma a long exact sequence
\[
T_{\sigma_{\mathrm{out}}} 
\rightarrow T_{\sigma_1}\times T_{\sigma_2}\xrightarrow{\xi} \coker\theta\rightarrow
\coker{\Psi}_{\sigma_{\mathrm{out}}}\rightarrow \coker{\Psi}_{\sigma_1}\times
\coker{\Psi}_{\sigma_2}\rightarrow 0.
\]
As $T_v=T_{E_1}\cap T_{E_2}$, the
cokernel of the map $T_v\rightarrow T_{E_1}\times T_{E_2}$
given by $m\mapsto (-m,-m)$ can be identified with 
$T_{E_1}+T_{E_2}\subseteq M$, and so 
\[
\coker\theta=(T_{E_1}+T_{E_2})/(\Z u_{E_1}+\Z u_{E_2})
=(\cL_1^{\sat}+\cL_2^{\sat})/(\Z u_{E_1}+\Z u_{E_2})\,.
\]
On the other hand, the connecting map $\xi:T_{\sigma_1}\times T_{\sigma_2}\rightarrow
\coker\theta$ is given by
$\xi(t_1,t_2)=(p_1(t_1), p_2(t_2))$,
and so
\[
\coker\xi=
(\cL_1^{\sat}+\cL_2^{\sat})/(\cL_1+\cL_2)\,,
\]
because $\cL_k=p_k(T_{\sigma_k})+\Z u_{E_k}$ by definition. This conclude the proof of 
\eqref{eq_coker_1}.

Consider now the case when one of $E_1$ or $E_2$ is an edge and the other is a leg. By symmetry, one can assume that $E_1$ is an edge, with adjacent vertex $v_1$ in $\Gamma_{\sigma_1}$, and $E_2$ is a leg $L_i$
 for some $1\leq i\leq r$. We have to show that 
 \begin{equation} \label{eq_coker_2}
|\coker \Psi_{\sigma_{\mathrm{out}}}|
=
|(\cL_{1}^{\sat}+\cL_{2}^{\sat})/
(\cL_{1}+\cL_{2})| 
|\coker \Psi_{\sigma_1}|\,.
\end{equation}
In this case, we consider the commutative diagram of exact sequences
\[
\xymatrix@C=30pt
{
0\ar[r]&T_v\times\Z\ar[d]_{\theta}\ar[r]& M_1\times T_v\times
\Z\ar[r]\ar[d]_{{\Psi}_{\sigma_{\mathrm{out}}}}&M_1\ar[d]_{\Psi_{\sigma_1}}\ar[r]&0\\
0\ar[r]&T_{E_1}\ar[r]&
K_1\times T_{E_1}
\ar[r]&
K_1\ar[r]&0\,,
}
\]
where the map $\theta$ is given by 
\[
\theta(m,\ell)=(\ell u_{E_1}-m).
\]
As $\ker \Psi_{\sigma_{\mathrm{out}}}=
T_{\sigma_{\mathrm{out}}}$ and $\ker \Psi_{\sigma_1}=T_{\sigma_1}$, we obtain by the snake lemma a long exact sequence
\[
T_{\sigma_{\mathrm{out}}} 
\rightarrow T_{\sigma_1}\xrightarrow{\xi} \coker\theta\rightarrow
\coker{\Psi}_{\sigma_{\mathrm{out}}}\rightarrow \coker{\Psi}_{\sigma_1}\rightarrow 0.
\]
As $T_v=T_{E_1} \cap \gamma_i^{\perp}$ and $\cL_2=\cL_2^{\sat}=\gamma_i^{\perp}$, we have
$\coker\theta=T_{E_1}/(T_v+\Z u_{E_1})
=\cL_1^\sat/(\cL_1^\sat \cap \cL_2^\sat +\Z u_{E_1})$.
On the other hand, the connecting map $\xi:T_{\sigma_1}\rightarrow
\coker\theta$ is given by
$\xi(t)=p_1$,
and so
\[
\coker\xi=
\cL_1^{\sat}/(\cL_1+\cL_1^\sat \cap \cL_2^\sat)\,,
\]
which is equal to $(\cL_1^\sat +\cL_2^\sat)/(\cL_1+\cL_2)$ because $\cL_2=\cL_2^\sat$. This concludes the proof of 
\eqref{eq_coker_2}.

Finally, if both $E_1$ and $E_2$ are legs $L_i$ and $L_j$, then $\Psi_{\sigma_{\mathrm{out}}}$ is the zero map $T_v \rightarrow 0$, so $\coker \Psi_{\sigma_{\mathrm{out}}}=0$. Moreover, $\cL_1=\cL_1^\sat=\gamma_i^\perp$ and $\cL_2=\cL_2^\perp=\gamma_j^\perp$, and so 
$(\cL_1^\sat +\cL_2^\sat)/(\cL_1+\cL_2)=0$.
In particular, we have $\coker \Psi_{\sigma_{\mathrm{out}}}=(\cL_1^\sat +\cL_2^\sat)/(\cL_1+\cL_2)$
and this ends the proof of Lemma \ref{lem:coker psi}.
\end{proof}

Similarly, the integral tangent space $T_{\tilde\sigma}$ can be viewed as
the kernel of the map
\begin{equation}
    \label{Eq: WidetildePsi}
\Psi_{\tilde\sigma}:
\prod_{v\in V(\Gamma_{\tilde\sigma})} T_v\times \prod_{E\in E(\Gamma_{\tilde\sigma})}
\Z\rightarrow \prod_{E\in E(\Gamma_{\tilde\sigma})} T_E    
\end{equation}
with $\Psi_{\tilde\sigma}$ defined analogously as $\Psi_\sigma$ in \eqref{Eq:Psi}. The 
comparison we need between $T_\sigma$
and $T_{\tilde\sigma}$
is provided by the following lemma.
We denote by $B$ the set of divalent vertices of $\Gamma_{\tilde\sigma}$ and for every $v \in B$, we denote $w_v := 
|\overline{u}_E|$, where $E$ is an edge adjacent to $v$
and $\overline{u}_E$ is the image of $u_E$ in the rank two lattice $M/T_v$. The definition of $w_v$ is independent of the choice of $E$  by the balancing condition at $v$.

\begin{lemma}
\label{lem:two gammas compare}
Let $\mathbf{A}$ be a general $\boldsymbol{\gamma}$-constraint, $\sigma$ a $(d-2)$-dimensional face of $\cM_{\omega,\boldsymbol{\gamma},\mathbf{A}}^\trop$ such that $\dim \fod_{L_{\mathrm{out}}}^\sigma =d-1$, and $\tilde\sigma$ a $(d-2)$-dimensional face of $\cM_{\omega,\boldsymbol{\gamma},\mathbf{A}}^\trop(\scrP)$
contained in $\sigma$,
Then, there is a long exact sequence
\begin{equation}\label{eq_les}
0\rightarrow T_{\tilde\sigma}\rightarrow T_{\sigma}
\rightarrow \prod_{v\in B} \Z/w_v\Z\rightarrow \coker\Psi_{\tilde\sigma}
\rightarrow \coker\Psi_\sigma\rightarrow 0 \,.
\end{equation}
\end{lemma}

\begin{proof}
We first introduce some notation. For every divalent vertex $v\in V(\Gamma_{\tilde\sigma})$, 
we denote by $E_1(v)$ and $E_2(v)$ the
two edges of $\Gamma_{\tilde\sigma}$ with vertex $v$, with $E_1(v)$ oriented towards $v$ and $E_2(v)$ oriented away from $v$.
We consider the map 
\[ \delta: \prod_{v\in B} T_v\times\prod_{v\in B}\Z 
\longrightarrow \prod_{v\in V(\Gamma_{\tilde\sigma})}T_v\times\prod_{E\in E(\Gamma_{\tilde\sigma})} \Z\,,\]
which is the obvious inclusion on $\prod_{v\in B}T_v\rightarrow \prod_{v\in V(\Gamma_{\tilde\sigma})} T_v$, and such that the component of $\delta\big((m_v),(\ell_v)\big)$ indexed by $E\in E(\Gamma_{\tilde\sigma})$ is given by 
\[ \sum_{\substack{v\in B\\ E_1(v)=E}} l_v
- \sum_{\substack{v\in B\\ E_2(v)=E}} l_v\,.\]

We also define a map 
\[ \delta': \prod_{v\in B} T_{E_1(v)} \longrightarrow
\prod_{E\in E(\Gamma_{\tilde\sigma})} T_E \]
by $\delta'((m_v)_{v\in B})=(n_E)_{E\in E(\Gamma_{\tilde\sigma})}$, where
\[ n_E = \sum_{\substack{v\in B\\ E_1(v)=E}}m_v - \sum_{\substack{v \in B\\ E_2(v)=E}} m_v \,.\]
Finally, we define
\begin{align} \pi : \prod_{v\in V(\Gamma_{\tilde\sigma})}T_v\times\prod_{E\in E(\Gamma_{\tilde\sigma})} \Z
&\longrightarrow 
\prod_{v\in V(\Gamma_{\sigma})}T_v\times\prod_{E\in E(\Gamma_{\sigma})}\Z \\
\big((n_v)_{v\in V(\Gamma_{\tilde\sigma})},(\ell_E)_{E\in E(\Gamma_{\tilde\sigma})}\big)
& \longmapsto 
\big((n_v)_{v\in V(\Gamma_\sigma)},(\sum_{\substack{E'\subset E\\ E'\in E(\Gamma_{\tilde\sigma})}} \ell_{E'})_{E
\in E(\Gamma_\sigma)}\big)\,,
\end{align}
and 
\begin{align} \pi' : \prod_{E\in E(\Gamma_{\tilde\sigma})}T_E
&\longrightarrow 
\prod_{E\in E(\Gamma_{\sigma})}T_E \\
(n_E)_{E\in E(\Gamma_{\tilde\sigma})}
& \longmapsto 
(\sum_{\substack{E'\subset E\\ E'\in E(\Gamma_{\tilde\sigma})}} n_{E'})_{E
\in E(\Gamma_\sigma)}\,,
\end{align}
where we view $\Gamma_{\tilde\sigma}$ as a refinement of $\Gamma_\sigma$,
so that each edge of $\Gamma_{\tilde\sigma}$ is contained in some edge of
$\Gamma_{\sigma}$.
These maps fit into  a commutative diagram of exact sequences
\[
\xymatrix@C=20pt
{
0\ar[r]&\prod_{v\in B} T_v\times\prod_{v\in B}\Z \ar[r]^{\delta}\ar[d]_{\Psi'}&
\prod_{v\in V(\Gamma_{\tilde\sigma})}T_v\times\prod_{E\in E(\Gamma_{\tilde\sigma})} \Z
\ar[r]^{\pi}\ar[d]_{\Psi_{\tilde\sigma}}&
\prod_{v\in V(\Gamma_{\sigma})}T_v\times\prod_{E\in E(\Gamma_\sigma)}\Z\ar[d]^{\Psi_\sigma}\ar[r]&0\\
0\ar[r]&\prod_{v\in B} T_{E_1(v)}\ar[r]_{\delta'}&
\prod_{E\in E(\Gamma_{\tilde\sigma})} T_E\ar[r]_{\pi'}&
\prod_{E\in E(\Gamma_{\sigma})}T_E\ar[r]&0
}
\]
where $\Psi'$ is defined by 
$\Psi'\big((m_v),(\ell_v)\big)=(\ell_v u_{E_1(v)}-m_v)$.
The map $\Psi'$ is injective and the
cokernel of $T_v\times\Z\rightarrow T_{E_1(v)}$, $(m,\ell)\mapsto
-m+\ell u_{E_1(v)}$, is just $\Z/w_v\Z$. Therefore, the 
snake lemma gives the desired long exact sequence \eqref{eq_les}.
\end{proof}

\section{Counts of $(\omega,\boldsymbol{\gamma})$-marked stable log maps in toric varieties}
\label{sec_log_gw}

In this section, after briefly reviewing in \S\ref{sec_log_trop_review}-\ref{subsubsec:stable punctured} the theory of stable log maps, we define in \S\ref{subsec: The moduli space of mw log maps}-\ref{sec_log_gw_invts} a particular class of genus $0$ log Gromov--Witten invariants of toric varieties.

\subsection{Log schemes and their tropicalizations}
\label{sec_log_trop_review}
We assume basic familiarity with log
geometry throughout this section \cite{Kk, Ogus}. We nonetheless roughly review definitions to establish notation. 

A \emph{log structure} on a scheme $X$ is a sheaf of commutative monoids $\shM_X$ together with a homomorphism of sheaves of multiplicative monoids
$\alpha_X : \shM_X \to \O_X$ inducing an isomorphism $\alpha_X^{-1}(\O_X^{\times})\rightarrow \O_X^{\times}$, allowing us to identify $\O_X^{\times}$ as a subsheaf
of $\shM_X$. The standard notation we use for a log scheme is $X:=(\ul{X},\shM_X,\alpha_X)$, where by $\ul X$ we denote the underlying scheme. By abuse of notation we also denote the underlying scheme just by $X$ when it is clear from the context. Throughout this paper, we assume that all log schemes are fine and saturated (fs) 
\cite[I,\S1.3]{Ogus}. A \emph{morphism of log schemes} $f:X \rightarrow Y$ consists of an
ordinary morphism $\ul{f}:\ul{X}\rightarrow \ul{Y}$ of schemes along with
a map $f^{\flat}:f^{-1}\shM_Y\rightarrow \shM_X$ which is compatible with $f^\#:f^{-1}\O_Y\rightarrow \O_X$ via the structure homomorphisms
$\alpha_X$ and $\alpha_Y$.

The \emph{ghost sheaf}, defined by
$\overline{\shM}_{X}:=\shM_X/\O_X^{\times}$, captures the key combinatorial
information about the log structure. In particular, it leads
to the description of the \emph{tropicalization} $\Sigma(X)$ of a log scheme $X$, as an abstract polyhedral
cone complex, see \cite[\S2.1]{ACGSI} for details. Briefly,
$\Sigma(X)$ is a collection of cones along with face maps between
them. There is one cone $\sigma_{\bar x}:=
\Hom(\overline{\shM}_{X,\bar x},\RR_{\ge 0})$ for every 
geometric point $\bar x\rightarrow X$,
and if $\bar x$ specializes to $\bar y$, there is a generization map
$\overline\shM_{X,\bar y}\rightarrow \overline\shM_{X,\bar x}$
which leads dually to a map
$\sigma_{\bar x}\rightarrow\sigma_{\bar y}$.
The condition of being fine and saturated implies this is an inclusion
of faces. Note that each cone $\sigma_{\bar x}\in \Sigma(X)$ comes with a
tangent space of integral tangent vectors $$\Lambda_{\sigma_{\bar x}}:=\Hom(\overline\shM_{X,\bar x},\Z)$$
and a set of integral points
\[
\sigma_{\bar x,\Z}:=\Hom(\overline\shM_{X,\bar x},\NN).
\]
Tropicalization is functorial, with $f:X\rightarrow Y$ inducing
$f_{\mathrm{trop}}:\Sigma(X)\rightarrow\Sigma(Y)$, with a map of cones
$\sigma_{\bar x}\rightarrow\sigma_{f(\bar x)}$ induced by
$\bar f^{\flat}:\overline\shM_{Y,f(\bar x)}\rightarrow \overline\shM_{X,x}$. In cases we consider in this paper, after identifying $\sigma_{\bar x}$ and
$\sigma_{\bar y}$ whenever $\sigma_{\bar x}\rightarrow\sigma_{\bar y}$
is an isomorphism, we obtain an ordinary polyhedral cone complex. 

\begin{example}
\label{Ex: divisorial}
Let $X$ be a toric variety associated to a fan $\Sigma$ in $M_\RR$. There is a canonical \emph{divisorial log structure} $\M_X=\M_{(X,D)}$, where $D\subset X$ is the toric boundary divisor \cite[ex. 3.8]{Mark}. Furthermore, the tropicalization $\Sigma(X_\Sigma)$ is naturally identified with $M_\RR$ endowed with the fan $\Sigma$.
\end{example}

\begin{example}
\label{Ex: log point}
Let $X$ be a log point, that is, a log scheme whose underlying scheme is a point $\underline{X}=\Spec \kk$ for some field $\kk$. Then, there exists a monoid $Q$ with $Q^\times=\{0\}$, such that $\mathcal{M}_X=Q \oplus \kk^\star$ and $\alpha_X$ is the map
\[
Q\oplus\kk^\times\lra \kk,\quad
(q,a)\longmapsto \begin{cases} a,&q=0\\ 0,&q\neq0\,.\end{cases}
\]
One recovers $Q$ as the ghost sheaf of $X$, that is, $\overline{\mathcal{M}}_X=Q$, and the tropicalization $\Sigma(X)$ of $X$ is the cone $Q^\vee_\RR:=\Hom(Q,\RR_{\geq 0})$.
\end{example}

\subsection{Stable log maps}
\label{subsubsec:stable punctured}
  
Following \cite{logGWbyAC, logGW},
a \emph{prestable log map} with
target $(X,\mathcal{M}_X) \rightarrow (S,\mathcal{M}_S)$
is a commutative diagram in the category of log schemes
\begin{equation}
\label{eq__log_map}
    \xymatrix@C=30pt
{
C\ar[r]^f\ar[d]_{\pi} & X\ar[d]\\
W\ar[r]&S
}
\end{equation}
where $W=(W,\mathcal{M}_W)$ is a log point\footnote{For families of stable log maps, $W$ can be an arbitrary scheme.}, $\pi$ is a log smooth family of curves, which is required to be an integral morphism all of whose geometric fibers are reduced curves. In this situation, one can describe locally the log structure on the log curve $C$ \cite{katoF}. In particular, $C/W$ comes with a set of disjoint sections $p_1,\ldots,p_n:\ul{W}\rightarrow
\ul{C}$, referred to as the \emph{marked points}, disjoint from the nodal locus of $C$, and such that away from the nodal locus, \[\overline{\shM}_C=\pi^*\overline{\shM}_W
\oplus\bigoplus_{i=1}^n p_{i*}\ul{\NN}\,.\] 
A \emph{stable log map} is a prestable log map whose underlying prestable marked map is a stable map.

If $p\in C$ is a marked point, we have
\[
\bar f^{\flat}:P_p:=\overline{\shM}_{X,f(p)}\longrightarrow
\overline{\shM}_{C,p}=\overline{\shM}_{W,\pi(p)}\oplus\NN
\stackrel{\pr_2}{\longrightarrow}\NN,
\]
which can be viewed as an element $u_p\in P_p^{\vee}:=\Hom(P_p,\NN)
\subseteq \sigma_{f(p)}$, called the \emph{contact order at $p$}. Similarly, if $x=q$ is a node,
there exists a homomorphism 
\begin{equation}
\label{eq:node contact order}
u_q:P_q :=\overline{\shM}_{X,f(q)} \lra \Z,
\end{equation}
called the \emph{contact order at $q$}, see \cite[(1.8)]{logGW} or \cite[\S2.3.4]{ACGSI}.
In the case the target space is $(X,D)$,
the contact order records tangency information with the irreducible
components of $D$. 

A key point in log Gromov-Witten theory is the tropical interpretation of stable log maps. Let $f:C \to X$ be a stable log map as in \eqref{eq__log_map}, where $W$ is the log point $(\Spec \kk, \kk^\star \oplus Q)$, as in Example \ref{Ex: log point}. Then by functoriality of tropicalization, 
we obtain a diagram
\begin{equation}
\label{eq:tropical diagram}
\xymatrix@C=45pt
{
\Gamma:=\Sigma(C)\ar[r]^>>>>>>{h=f_{\mathrm{trop}}}\ar[d]_{\pi_{\mathrm{trop}}} & \Sigma(X)\ar[d]\\
\Sigma(W)\ar[r]& \Sigma(S)
}
\end{equation}
Here $\Sigma(W)=\Hom(Q,\RR_{\ge 0})=Q^{\vee}_{\RR}$ is a rational
polyhedral cone, and for $q\in \Int(Q^{\vee}_{\RR})$, 
$\pi_{\mathrm{trop}}^{-1}(q)$ can be identified with the dual graph $G$
of the curve $C$, which is the graph with vertices corresponding to an irreducible components of $C$, edges corresponding to nodes of $C$, and legs corresponding to marked points. 

The map $h: \Sigma(C) \to \Sigma(X)$ defines a family of tropical maps to
$\Sigma(X)$ as defined in \cite[Def.\ 2.21]{ACGSI}. For
$s\in \Int(Q^{\vee}_{\RR})$, write
\[
h_s:G\rightarrow\Sigma(X)
\]
for the restriction of $h$ to $G=\pi_{\mathrm{trop}}^{-1}(s)$.

Associated to any family of tropical maps to $\Sigma(X)$ is the type,
recording which cones of $\Sigma(X)$ vertices, edges and legs of
$G$ are mapped to, and tangent vectors to the images of edges and legs:

\begin{definition}
\label{Def: type of the tropical curve}
A \emph{type} of tropical map to $\Sigma(X)$ is data of a triple
$\tau=(G,\boldsymbol{\sigma}, \mathbf{u})$ where $G$ is a graph,  $\boldsymbol{\sigma}$
is a map
\[
\boldsymbol{\sigma}:V(G)\cup E(G)\cup L(G)\rightarrow\Sigma(X)
\]
with the property that if $v$ is a vertex of an edge or leg $E$,
then $\boldsymbol{\sigma}(v)\subseteq \boldsymbol{\sigma}(E)$.
Next, $\mathbf{u}$ associates to each oriented edge $E\in E(G)$
a tangent vector $\mathbf{u}(E)\in \Lambda_{\boldsymbol{\sigma}(E)}$
and to each leg $L\in L(G)$ a tangent vector $\mathbf{u}(L)\in
\Lambda_{\boldsymbol{\sigma}(L)}$.
\end{definition}

Associated to a  type is a moduli space of tropical maps of the
given type, and this dually defines a monoid called the \emph{basic monoid}:

\begin{definition}
\label{def:basic monoid}
Given a  type $\tau=(G,\boldsymbol{\sigma}, \mathbf{u})$,
we define the
\emph{basic monoid} $Q_{\tau}=\Hom(Q_{\tau}^{\vee},\NN)$ of 
$\tau$ by defining its dual:
\begin{equation}
\label{eq:basic dual}
Q^{\vee}_{\tau}:=\big\{\big((p_v)_{v\in V(G)},(\ell_E)_{E\in E(G)}\big)\,|\,
\hbox{$p_{v'}-p_{v}=\ell_E \mathbf{u}(E)$ for all $E\in E(G)$}\big\},
\end{equation}
a submonoid of
\[
\prod_{v\in V(G)} \boldsymbol{\sigma}(v)_\Z \times \prod_{E\in E(G)}\NN.
\]
Here $\boldsymbol{\sigma}(v)_{\Z}$ denotes the set of integral points
of the cone $\boldsymbol{\sigma}(v)$, and
$v'$, $v$ are taken to be the endpoints of $E$ consistent
with the chosen orientation of the edge.
\end{definition}

Given a stable log map 
$f \colon C/W \rightarrow X/S$, there is a canonical map 
$Q_\tau \rightarrow \overline{M}_W$, where $Q_\tau$ is the basic monoid defined by the combinatorial type of $f$ \cite{logGW,logGWbyAC}. A stable log map 
$f \colon C/W \rightarrow X/S$
is said to be \emph{basic} if the natural map of monoids 
$Q_\tau \rightarrow \overline{\mathcal{M}}_W$ is an isomorphism. 
When $X \rightarrow S$ is log smooth and projective, \cite{logGWbyAC,logGW} shows that the moduli space of basic stable log maps with given genus, number of marked points, and degree, is a Deligne-Mumford stack, proper over $S$, and carrying a natural virtual fundamental class which can be used to define log Gromov--Witten invariants. From now on, we use stable log maps to mean basic stable log maps.

\subsection{Moduli spaces of $(\omega,\boldsymbol{\gamma})$-marked stable log maps}
\label{subsec: The moduli space of mw log maps}

Let $(\omega, \boldsymbol\gamma)$ be as in \S\ref{sec_marked_tropical_curves}:  $\omega \in \bigwedge^2 M$ is a skew-symmetric form on $N$, and $\boldsymbol{\gamma}=(\gamma_1,\dots,\gamma_r)$ is a $r$-tuple of elements $\gamma_i \in N$ such that $\iota_{\gamma_i} \omega \neq 0$ for all $1\leq i \leq r$ and $\iota_\gamma \omega \neq 0$, where $\gamma:=\sum_{i=1}^r \gamma_i$. 

\begin{definition}\label{def_gamma_fan}
A \emph{$\boldsymbol{\gamma}$-fan}
is a fan $\Sigma$ in $M_\RR$ of a $d$-dimensional projective toric variety $X_\Sigma$ over $\C$ such that:
\begin{itemize}
    \item[(i)] for every $1\leq i\leq r$, $\RR_{\geq 0}\iota_{\gamma_i}\omega$ is a ray of $\Sigma$,
    \item[(ii)] for every $1 \leq i \leq r$, the hyperplane $(\gamma_i^{\perp})_\RR$ is a union of $(d-1)$-dimensional cones of $\Sigma$.
\end{itemize}
Given a $\boldsymbol{\gamma}$-fan $\Sigma$, for every $1\leq i\leq r$, we denote by $D_i$ the toric divisor of $X_\Sigma$ corresponding to the ray $\RR_{\geq 0}\iota_{\gamma_i}\omega$.
\end{definition}

In what follows we focus attention on particular stable log maps, referred to as $(\omega,\boldsymbol{\gamma})$-marked stable log maps, defined as follows.

\begin{definition} \label{def_stable_log}
Let $\Sigma$ be a $\boldsymbol{\gamma}$-fan. 
An $(\omega,\boldsymbol{\gamma})$-marked stable log map to $X_\Sigma$ is a genus $0$ stable log map $f \colon C \rightarrow X_\Sigma$
endowed with a bijection between its set of marked points and 
\[ \{ x_i\,|\, 1\leq i\leq r \} \cup \{ x_{\mathrm{out}}\}\,,\]
and such that $u_{x_{i}}=\iota_{\gamma_i}\omega$ for all 
$1\leq i \leq r$, and $u_{x_{\mathrm{out}}}=-\iota_{\gamma}\omega$.
\end{definition}

Let $\Sigma$ be a $\boldsymbol{\gamma}$-fan.
We denote by $\cM_{\omega,\boldsymbol{\gamma}}^{\log}(X_\Sigma)$ the moduli space of $(\omega,\boldsymbol{\gamma})$-marked stable log maps to $X_\Sigma$. This moduli space is a proper Deligne-Mumford log stack \cite{logGW}. By \cite[Proposition 3.3.1]{ran2017toric},
$\cM_{\omega,\boldsymbol{\gamma}}^{\log}(X_\Sigma)$ is log smooth of the expected dimension $d-3+(r+1)=d-2+r$.
Moreover, $\cM_{\omega,\boldsymbol{\gamma}}^{\log}(X_\Sigma)$ is also irreducible by \cite[Proposition 3.3.5]{ran2017toric}.

It follows from Definition \ref{def_gamma_fan} that, for every $1\leq  i\leq r$, the projection 
\[M_\RR/\RR \iota_{\gamma_i}\omega \longrightarrow M_\RR/(\gamma_i^{\perp})_\RR \simeq \RR \]
is a morphism of fans from the fan $\Sigma(D_i)$ of $D_i$ in 
$M_\RR/\RR \iota_{\gamma_i} \omega$ to the fan $\Sigma(\PP^1)$ of $\PP^1$
in $M_\RR/(\gamma_i^{\perp})_\RR \simeq \RR$
(consisting of the cones $\RR_{\leq 0}$, $\{0\}$, $\RR_{\geq 0}$). Hence, there exists a corresponding toric morphism 
\[ \pi_i : D_i \longrightarrow \PP^1\,,\]
which in restriction to the dense torus orbit of $D_i$ is simply the monomial $z^{\frac{\gamma_i}{|\gamma_i|}}$.
For every point $[\mathbf{H}]=([H_1],\dots,[H_r]) \in (\PP^1)^r$, we denote $\mathbf{H}:=(H_1,\dots, H_r)$, where, for every $1\leq i \leq r$, $H_i$ is the hypersurface in $D_i$ given by the fiber of $\pi_i$ over $[H_i]\in \PP^1$, that is, $H_i:=\pi_i^{-1}([H_i])$. For $[H_i]=c_i \in \kk^{\star} \subset \PP^1$, $H_i$ is the closure in $D_i$ of the hypersurface $-c_i+z^{\frac{\gamma_i}{|\gamma_i|}}=0$
in $(\kk^{\star})^{d-1} \subset D_i$.

For every $1 \leq i \leq r$, we have a (scheme) evaluation morphism at the marked point $x_i$:
\begin{align*}\underline{\ev}_{i} \colon \cM_{\omega,\boldsymbol{\gamma}}^{\log}(X_\Sigma) &\longrightarrow D_i
\\ f &\longmapsto f(x_{i})\,,
\end{align*}
Composing with $\pi_i:  D_i \rightarrow \PP^1$, we obtain a (scheme) morphism
\begin{align} \label{eq_nu_underline}
\underline{\nu} \colon 
\cM_{\omega,\boldsymbol{\gamma}}^{\log}(X_\Sigma) &\longrightarrow (\PP^1)^r \\
\nonumber
f &\longmapsto (\pi_i(f(x_i)))_i \,.
\end{align}

For every $[\mathbf{H}]\in (\PP^1)^r$, the moduli space $\cM_{\omega,\boldsymbol{\gamma},\mathbf{H}}^\log(X_\Sigma)$ of $(\omega,\boldsymbol{\gamma})$-stable log maps matching $\mathbf{H}$ is defined by the fiber diagram (in the category of schemes)
\begin{equation} \label{eq_diag}
\begin{tikzcd}
\cM_{\omega,\boldsymbol{\gamma},\mathbf{H}}^\log(X_\Sigma)
\arrow[r]
\arrow[d]
&
\cM_{\omega,\boldsymbol{\gamma}}^{\log}(X_\Sigma)
\arrow[d,"(\underline{\ev}_{i})_{i}"]
\\
\prod_{i=1}^r H_i
\arrow[r,"\iota_H"]
& 
\prod_{i=1}^r D_i\,,
\end{tikzcd}
\end{equation}
where the bottom horizontal arrow $\iota$
is defined by the inclusion morphisms $H_{i}\subset D_i$,
or equivalently by the fiber diagram
\begin{equation}\label{eq_diag1}
\begin{tikzcd}
\cM_{\omega,\boldsymbol{\gamma},\mathbf{H}}^\log (X_\Sigma)
\arrow[r]
\arrow[d]
&
\cM_{\omega,\boldsymbol{\gamma}}^{\log}(X_\Sigma)
\arrow[d,"\underline{\nu}"]
\\
 \,[\mathbf{H}]\, \arrow[r,"\iota_{[\mathbf{H}]}"]
& 
(\PP^1)^r\,,
\end{tikzcd}\end{equation}
where $\iota_{[\mathbf{H}]}$ is the inclusion of the point $[\mathbf{H}]$ in $(\PP^1)^r$.

\begin{lemma}  \label{lem_log_lift}
For every $1\leq i\leq r$, the evaluation scheme morphism $\underline{\ev}_{i}$ lifts naturally to a log morphism 
\[ \ev_{i}: \cM_{\omega,\boldsymbol{\gamma}}^{\log}(X_\Sigma) \longrightarrow D_i\,,\]
where $D_i$ is given its toric log structure (as opposed to the log structure restricted from $X_\Sigma$).
\end{lemma}

\begin{proof} 
Let $(\pi: C \rightarrow T, f:C \rightarrow X_\Sigma)$ be an $(\omega,\boldsymbol{\gamma})$-marked stable log map. Evaluation at the marked point $x_i$ defines a scheme evaluation morphism $\underline{\ev}_i : T \rightarrow D_i$, and we want to show that $\underline{\ev}_i$ naturally lifts to a log morphism $\ev_i : T \rightarrow D_i$, where $D_i$ is given its toric log structure.

Let $T_i$ be the log scheme with underlying scheme $\underline{T}$ and with log structure $x_i^\star \cM_C$, where the marked point $x_i$ is viewed as a section $\underline{T} 
\rightarrow \underline{C}$. Then $x_i$ lifts as a log morphism $T_i \rightarrow C$, and composing with $f: C \rightarrow X_\Sigma$, we obtain a log lift of $\underline{\ev}_i$ as a log 
morphism $\ev_i: T_i \rightarrow D_i'$, where 
$D_i'$ is the log scheme with underlying scheme $\underline{D}_i$ and log structure restricted from $X_\Sigma$. We have $\cM_T \subset \cM_{T_i}$ and $\cM_{D_i} \subset \cM_{D_i'}$, so it is sufficient to show that 
$\ev_i^\flat: \ev_i^{*} \cM_{D_i'} \rightarrow \cM_{T_i}$
takes  $\ev_i^{*} \cM_{D_i}$ into $\cM_T$. 
To do this, it is sufficient to show this at the level of stalks of ghost sheaves. 

For any $p \in C$ in the image of the section $x_i$, we obtain an induced map
\[ \overline{\ev}^\flat_i: P_p :=\overline{\cM}_{X_\Sigma,f(p)}=\overline{\cM}_{D_i' ,f(p)} 
\longrightarrow \overline{\cM}_{C,p}=Q \oplus \NN \,,\]
where $Q=\overline{\cM}_{T,\pi(p)}$, and where the
composition with the second projection is given by the contact order $u_{x_{i}}=\iota_{\gamma_i}\omega \in P_p^\vee$. Necessarily $\Spec \kk [P_p]$ is an affine toric chart of $X_\Sigma$ containing $f(p)$, while $\Spec \kk [P_p \cap u_{x_i}^\perp]$ is the intersection of this toric affine chart with $D_i$. In particular, $\overline{\cM}_{D_i,f(p)}=P_p \cap u_{x_i}^\perp$ and so $\overline{\ev}^\flat_i$ induces a homomorphism between submonoids
\[ \overline{\cM}_{D_i,f(p)} \longrightarrow Q\,, \]
as desired.
\end{proof}

The toric morphisms $\pi_i: D_i  \rightarrow \PP^1$ naturally lifts to log morphisms when $D_i$ and $\PP^1$ are endowed with their toric log structures, and so by composition with the log lifts $\ev_{i}$ of the evaluation morphisms $\underline{\ev}_i$ given by Lemma \ref{lem_log_lift}, we obtain a log lift of the scheme morphism $\underline{\nu}$ introduced in \eqref{eq_nu}, 
\begin{equation}\label{eq_nu}
\nu \colon 
\cM_{\omega,\boldsymbol{\gamma}}^{\log}(X_\Sigma) \longrightarrow (\PP^1)^r\,,\end{equation}
where $(\PP^1)^r$ is endowed with its toric log structure. For every point $[\mathbf{H}] \in (\PP^1)^r$, we endow $[\mathbf{H}]$ with the log structure restricted from the toric log structure on $(\PP^1)^r$. In particular, this log structure is trivial if $[\mathbf{H}] \in (\kk^\star)^r \subset (\PP^1)^r$. Similarly, we endow $H_i$ with the log structure restricted from the toric log structure on $D_i$.
Then, the inclusions $\iota_H$ and $\iota_{[\mathbf{H}]}$ become a strict morphism of log schemes, and so the fiber diagrams of schemes
\eqref{eq_diag} and
\eqref{eq_diag1} lift to fiber diagrams of fs log schemes.  

For $[\mathbf{H}] \in (\kk^{\star})^r \subset (\PP^1)^r$, the tropicalization of $[\mathbf{H}]$  is the origin $0$ in the tropicalization $\Sigma((\PP^1)^r) \simeq \RR^r$ of $(\PP^1)^r$.
Hence, taking the tropicalization of \eqref{eq_diag1} viewed as a fiber diagram of fs log schemes, we obtain the fiber diagram  
of cone complexes
\begin{equation}\label{eq_diag_trop}
\begin{tikzcd}
\Sigma(\cM_{\omega,\boldsymbol{\gamma},\mathbf{H}}^\log(X_\Sigma))
\arrow[r]
\arrow[d]
&
\Sigma(\cM_{\omega,\boldsymbol{\gamma}}^\log(X_\Sigma))
\arrow[d,"\Sigma(\nu)"]
\\
0
\arrow[r]
& \RR^r\,,
\end{tikzcd}\end{equation}
where $\Sigma(\cM_{\omega,\boldsymbol{\gamma},\mathbf{H}}(X_\Sigma))$ and $\Sigma(\cM_{\omega,\boldsymbol{\gamma}}^\log(X_\Sigma))$ are the tropicalizations of $\cM_{\omega,\boldsymbol{\gamma},\mathbf{H}}(X_\Sigma)$ and $\Sigma(\cM_{\omega,\boldsymbol{\gamma}}^\log(X_\Sigma)$ respectively.

As discussed in \S\ref{subsubsec:stable punctured}, the tropicalization of an $(\omega,\boldsymbol{\gamma})$-marked stable log map to $X_\Sigma$ is a family of $(\omega,\boldsymbol{\gamma})$-marked tropical curve in $(M_\RR,\Sigma)$ as in Definition \ref{def_tropical_curve_polyh}, where one views the fan $\Sigma$ as a particular polyhedral decomposition of $M_\RR$.
This induces a map of cone complexes 
\begin{equation} \label{eq_cone1} T:\Sigma(\cM_{\omega,\boldsymbol{\gamma}}^{\log}(X_\Sigma)) \longrightarrow \cM_{\omega,\boldsymbol{\gamma}}^\trop(\Sigma)\,.\end{equation}
Moreover, the map $\Sigma(\nu)$ in \eqref{eq_diag_trop} is the composition of $T$ with the tropical evaluation map at the legs $\ev^\trop: \cM_{\omega,\boldsymbol{\gamma}}^\trop(\Sigma) \rightarrow \prod_{i=1}^r M_\RR/(\gamma_i^\perp)_\RR \simeq \RR^r$
given in \eqref{eq_ev_trop}:
\[ \Sigma(\nu) =  \ev^\trop \circ T \,.\]
Recall from Definition \ref{Def affine constraint} that $\prod_{i=1}^r M_\RR/(\gamma_i^\perp)_\RR \simeq \RR^r$ is naturally the space of $\boldsymbol{\gamma}$-constraints. In particular, the origin $0 \in \RR^r$ corresponds to the 
$\boldsymbol{\gamma}$-constraint
\begin{equation}
    \label{Eq: A0}
    \mathbf{A}^0:=(A^0_1,\cdots,A^0_r) \,
\end{equation}
where $A_i^0$ is the linear hyperplane $(\gamma_i^\perp)_\RR$ in $M_\RR$.
Therefore, it follows from the diagram \eqref{eq_diag_trop}
that the restriction of $T$ to 
$\Sigma(\cM_{\omega,\boldsymbol{\gamma},\mathbf{H}}^\log(X_\Sigma))$ defines a map
\begin{equation}
T_H : \Sigma(\cM_{\omega,\boldsymbol{\gamma},\mathbf{H}}^\log(X_\Sigma)) \longrightarrow \cM_{\omega,\boldsymbol{\gamma}, \mathbf{A}^0}^\trop (\Sigma)\,,
\end{equation}
where $\cM_{\omega,\boldsymbol{\gamma}, \mathbf{A}^0}^\trop (\Sigma)$ is the moduli space of $(\omega,\boldsymbol{\gamma})$-marked tropical curves to $(M_\RR,\Sigma)$ matching $\mathbf{A}^0$. In other words, the tropicalization of a stable log map matching $\mathbf{H}$ is a tropical curve matching $\mathbf{A}^0$.

\subsection{Log Gromov--Witten invariants}
\label{sec_log_gw_invts}
Let $\rho$ be a $(d-2)$-dimensional face of $\cM^\trop_{\omega,\boldsymbol{\gamma}, \mathbf{A}^0}(\Sigma)$. As $\cM^\trop_{\omega,\boldsymbol{\gamma}, \mathbf{A}^0}(\Sigma)$ is a sub-cone complex of $\cM^\trop_{\omega,\boldsymbol{\gamma}}(\Sigma)$,  $\rho$ is also a face of $\cM^\trop_{\omega,\boldsymbol{\gamma}}(\Sigma)$.
The relative interior $\Int(\rho)$ parametrizes $(\omega,\boldsymbol{\gamma})$-marked tropical curves in $(M_\RR,\Sigma)$ of a given type $\tau_\rho$.
Let $\cM^\log_{\rho}(X_\Sigma)$ be the closure in 
$\cM^\log_{\omega,\boldsymbol{\gamma}}(X_\Sigma)$
of the locus of $(\omega,\boldsymbol{\gamma})$-marked stable log maps of type $\tau_\rho$.

As reviewed in
\S\ref{subsec: The moduli space of mw log maps},
$\cM^\log_{\omega,\boldsymbol{\gamma}}(X_\Sigma)$ is log smooth of dimension $d-2+r$. By definition, $\cM^\log_{\rho}(X_\Sigma)$ is a union of log strata of $\cM^\log_{\omega,\boldsymbol{\gamma}}(X_\Sigma)$ where the ghost monoid is generically give by the basic monoid $Q_{\tau_\rho}$. As $\rk Q_{\tau_\rho}^\gp =d-2$, it follows that 
$\cM^\log_{\rho}(X_\Sigma)$ is of pure dimension $r$. In particular, the virtual fundamental class on $\cM^\log_{\rho}(X_\Sigma)$ constructed by log Gromov--Witten theory \cite{logGW} coincides with the usual fundamental class 
$[\cM^\log_{\rho}(X_\Sigma)]$.

For every $[\mathbf{H}] \in (\PP^1)^r$, let  $\cM_{\rho,\mathbf{H}}^\log(X_\Sigma)$ be 
the closure in $\cM^\log_{\omega,\boldsymbol{\gamma},\mathbf{H}}(X_\Sigma)$
of the locus of $(\omega,\boldsymbol{\gamma})$-marked stable log maps of type $\tau_\rho$.
Restricting the morphism $\nu$ of \eqref{eq_nu} to $\cM_{\rho}^\log(X_\Sigma)$, we obtain a fiber diagram
\begin{equation}\label{eq_diag3}
\begin{tikzcd}
\cM_{\rho,\mathbf{H}}^\log(X_\Sigma)
\arrow[r]
\arrow[d]
&
\cM_{\rho}^\log(X_\Sigma)
\arrow[d,"\nu"]
\\
\,[\mathbf{H}]\,
\arrow[r,"\iota_{[\mathbf{H}]}"]
& 
(\PP^1)^r\,,
\end{tikzcd}\end{equation}
and we define a 0-dimensional virtual fundamental class on $\cM^\log_{\rho,\mathbf{H}}(X_\Sigma)$
by 
\begin{equation}\label{eq_vclass1} [\cM^\log_{\rho,\mathbf{H}}(X_\Sigma)]^\virt:=\iota_{[\mathbf{H}]}^! [\cM^\log_\rho(X_\Sigma)]\,,
\end{equation}
where $\iota_{[\mathbf{H}]}^!$ is the Gysin pullback \cite[Chapter 6]{Fult} defined by the regular embedding $\iota_{[\mathbf{H}]}$ of codimension $r$.
As the moduli space $\cM_{\rho,\mathbf{H}}^\log (X_\Sigma)$ is proper, one can define a log Gromov--Witten invariant $N_\rho^{\mathrm{toric}}$
as the degree of this class:
\begin{equation}\label{eq_N_toric}
N_\rho^{\mathrm{toric}}(X_\Sigma)=\mathrm{deg}[\cM_{\rho,\mathbf{H}}^\log (X_\Sigma)]^\virt \,.
\end{equation}
By deformation invariance of the Gysin pullback $\iota_{[\mathbf{H}]}^!$, the log Gromov--Witten invariant $N_\rho^{\mathrm{toric}}$ is independent of the choice of $[\mathbf{H}] \in (\PP^1)^r$.

\begin{theorem} \label{thm_enum}
For general $[\mathbf{H}]\in (\PP^1)^l$,
the moduli stack $\mathcal{M}_{\omega,\boldsymbol{\gamma},\mathbf{H}}^\log(X_\Sigma)$ is log smooth of dimension $d-2$, and, for every $(d-2)$-dimensional face of $\cM^\trop_{\omega,\boldsymbol{\gamma}, \mathbf{A}^0}(\Sigma)$, 
the moduli stack $\cM_{\rho, \mathbf{H}}^\log(X_\Sigma)$ is reduced and of pure dimension $0$. In particular, the virtual fundamental class is in this case given by the usual fundamental class,
\begin{equation}\label{eq_N_toric_enum}
N_\rho^{\mathrm{toric}}(X_\Sigma)=\mathrm{deg}[\cM_{\rho,\mathbf{H}}^\log(X_\Sigma)]\,,\end{equation}
and $N_\rho^{\mathrm{toric}}(X_\Sigma)$ is an enumerative count with automorphisms.
\end{theorem}

\begin{proof}
As $\mathcal{M}_{\omega,\boldsymbol{\gamma}}^\log(X_\Sigma)$ is log smooth of dimension $d-2+r$ over the trivial log point, it follows from the generic log smoothness result in Theorem \ref{thm_appendix} that $\mathcal{M}_{\omega,\boldsymbol{\gamma},\mathbf{H}}^\log(X_\Sigma)$ is also log smooth of dimension $d-2$ over the trivial log point for general $[\mathbf{H}] \in (\PP^1)^l$. In particular, as
$\cM^\log_{\rho,\mathbf{H}}(X_\Sigma)$ is a union of log strata of $\cM^\log_{\omega,\boldsymbol{\gamma},\mathbf{H}}(X_\Sigma)$ where the ghost monoid is generically given by the basic monoid $Q_{\tau_\rho}$ with $\rk Q_{\tau_\rho}^\gp =d-2$, this implies that $\cM^\log_{\rho}(X_\Sigma)$ is a union of $0$-dimensional strata of $\mathcal{M}_{\omega,\boldsymbol{\gamma},\mathbf{H}}^\log(X_\Sigma)$, which are necessarily reduced.
\end{proof}

\section{The log-tropical correspondence}

In this section, we establish our main correspondence result between the tropical multiplicities $N_\sigma^\trop$ introduced in 
\S \ref{sec:tropical_enum} and the log Gromov--Witten invariants $N_\rho^{\mathrm{toric}}(X_\Sigma)$ defined in \S \ref{sec_log_gw}.

\subsection{Good polyhedral decompositions and toric degenerations}
\label{sec_good_polyh}

As in the work of Nishinou-Siebert \cite{NS}, we obtain our log-tropical correspondence theorem by the study of a toric degeneration 
defined by an appropriately chosen polyhedral decomposition of $M_\RR$. In this section, we define the notion of a ``good" polyhedral decomposition
and we prove that good polyhedral decompositions exist. The main difference with the set up of \cite{NS} is that we are considering families of tropical curves and not rigid tropical curves in general.

We fix a skew-symmetric form $\omega \in \bigwedge^2 M$
on $N$, and a $r$-tuple
$\boldsymbol{\gamma}=(\gamma_1,\dots,\gamma_r)$ of elements $\gamma_i \in N$ such that $\iota_{\gamma_i} \omega \neq 0$ for all $1\leq i \leq r$ and $\iota_\gamma \omega \neq 0$, where $\gamma:=\sum_{i=1}^r \gamma_i$ as in \S\ref{sec_marked_tropical_curves}. We also fix a general $\boldsymbol{\gamma}$-constraint $\mathbf{A}$ defined as in Definitions \ref{Def affine constraint}-\ref{def_general_constraints}. In what follows, the tuple 
\begin{equation}
\label{Eq:tropical data}
(\omega,\boldsymbol{\gamma},\mathbf{A})    
\end{equation}
is referred to as the \emph{tropical data}.
Moreover, we say that $\mathbf{A}$ is \emph{rational} if the affine hyperplanes $A_i$ are defined over $\Q$, that is, $[\mathbf{A}] \in \Q^r \subset \RR^r$.
Recall that we reviewed the notion of polyhedral decomposition in Definition \ref{polyhedral decomposition}. Then, a good polyhedral decomposition for a given tropical data is defined as follows.

\begin{definition}
\label{def:good polyhedral}
Let $(\omega,\boldsymbol{\gamma},\mathbf{A})$ be a tropical data with rational $\mathbf{A}$.
A \emph{good polyhedral decomposition} $\scrP$ of $M_{\RR}$ for $(\omega,\boldsymbol{\gamma},\mathbf{A})$  is a rational polyhedral decomposition of $M_{\RR}$ which satisfies the following conditions:
\begin{enumerate}
\item For each $d-2$-dimensional face $\sigma$ of $\cM_{\omega,\boldsymbol{\gamma}, \mathbf{A}}^{\trop}$, each $\foj_v^\sigma$ defined as in \eqref{Eq:jv} is a union of $d-2$-dimensional faces
of $\scrP$ and each $\fod_E^\sigma$ defined as in \eqref{Eq:Ev} is a union of $d-1$-dimensional faces of
$\scrP$.
\item For every $1\leq i \leq r$, the constraint affine hyperplane $A_i$ is a union of $d-1$ dimensional faces of $\scrP$
\end{enumerate}
\end{definition}

To show that a good polyhedral decomposition exists, we first need the following lemma.

\begin{lemma} \label{lem_polyh}
Let $P$ be a $k$-dimensional polyhedron in $\RR^n$. Then, there exists a polyhedral decomposition $\scrP$ of $\RR^n$ such that $P$ is a union of $k$-dimensional cell of $\scrP$.
\end{lemma}

\begin{proof}
We first remark that is enough to prove the result for $n=k$: indeed, if $\scrP'$ is a polyhedral decomposition of the $k$-dimensional hull of $P$ such that $P$ is a union of $k$-dimensional cells of $\scrP'$, then, choosing a $(n-k)$-dimensional linear subspace $V$ transverse to the affine hull of $P$, $\scrP=\scrP'+ \Sigma_{\PP^{n-k}}$ is a polyhedral decomposition of $\RR^n$ with the same property, where $\Sigma_{\PP^{n-k}}$ is the fan of $\PP^{n-k}$ viewed as a fan in $V$. 

We prove the result for $k$-dimensional polytopes in $\RR^k$ by induction on $k$.
For $k=0$, there is nothing to prove. We now treat the induction step.
By the induction hypothesis, for every codimension 1 face $F$ of $P$, there exists a polyhedral decomposition $\scrP_F$ of the $(k-1)$-dimensional affine hull of $F$ such that $F$ is a union of $(k-1)$-dimensional cells of $\scrP_F$. For every such $F$, choose $n_F \in \Z^k$ not tangent to $F$, and define the polyhedral decomposition $\widetilde{\scrP}_F$ of $\RR^k$ with cells $\sigma+\RR_{\geq 0}n_F$ and $\sigma+\RR_{\leq 0}n_F$ for all cells $\sigma$ of $\scrP_F$. Then, $P$ is a union of $k$-dimensional cells of the polyhedral decomposition $\scrP:=\cap_F \widetilde{\scrP}_F$ of $\RR^k$, where the  intersection is taken over the codimension 1 faces $F$ of $P$.  
\end{proof}

\begin{remark}
    We gave an elementary proof of Lemma \ref{lem_polyh} for completeness. Much stronger results exist: for example, one could assume that $P$ is actually a $k$-cell of $\scrP$ by considering the cone over $P\times \{1\}$ in $\RR^n\times \RR$ and then using the rather difficult result that a cone can always be completed in a complete fan, see \cite{fans} and references there. We will not use these non-trivial results.
\end{remark}

\begin{lemma} \label{lem_good_poly}
For every tropical data $(\omega,\boldsymbol{\gamma},\mathbf{A})$ with rational $\mathbf{A}$, a good polyhedral decomposition exists.
\end{lemma}
\begin{proof}
By Theorem \ref{thm_trop}, we have $\dim \foj_v^{\sigma}=d-2$ and $\dim \fod_E^\sigma=d-1$
for every $\foj_v^\sigma$ and $\fod_E^\sigma$ as in Definition \ref{def:good polyhedral}. 
Moreover, by Lemma \ref{lem_polyh}, for every $\foj_v^\sigma$, $\fod_E^\sigma$ or $A_i$ as in Definition \ref{def:good polyhedral}, there exists polyhedral decompositions of $M_{\RR}$ containing respectively $\foj_v^\sigma$, $\fod_E^\sigma$ or $A_i$ as a face. Taking the intersection of all these polyhedral decompositions for every $\foj_v^\sigma$, $\fod_E^\sigma$ and $A_i$, we obtain a good polyhedral decomposition.
\end{proof}

Let $(\omega,\boldsymbol{\gamma},\mathbf{A})$ be a tropical data with rational $\mathbf{A}$.
Given a good polyhedral decomposition $\scrP$ for $(\omega,\boldsymbol{\gamma},\mathbf{A})$, whose existence is guaranteed by Lemma \ref{lem_good_poly}, one constructs a toric degeneration as in \cite[\S 3]{NS}.
Let $\overline{M}:=M \oplus \Z$, and let $\overline{\Sigma}_\scrP$ be the fan in $\overline{M}_\RR:= \overline{M} \otimes \RR$, whose faces are the cones over the cells of $\scrP$ viewed in $M_\RR \times \{1\} \subset \overline{M}_\RR$. We denote by $X_{\scrP}$ the corresponding $(d+1)$-dimensional toric variety. The projection on the $\RR$-factor of $\overline{M}_\RR=M_\RR \oplus \RR$ defines a map of fans $\overline{\Sigma}_\scrP \rightarrow \RR_{\geq 0}=\Sigma(\AA^1)$, and so a toric morphism
\begin{equation}
\label{Eq: total space}
\pi_{\scrP}: \shX_{\scrP} \longrightarrow \AA^1  \end{equation}
whose fiber $\pi_{\scrP}^{-1}(z)$ for $z \in \GG_m=\AA^1 \setminus \{0\}$ is the toric variety $X_{\Sigma_{\scrP}}$ associated to the \emph{asymptotic fan} $\Sigma_{\scrP}$ of $\scrP$, defined by
\begin{equation}
\label{Eq: Asymptotic fan}
\Sigma_\P:=\big\{\lim_{a\to0} a\Xi\subset M_\RR\,\big|\,
\Xi\in\P\big\}\,.    
\end{equation}
Irreducible components of the central fiber $\pi_{\scrP}^{-1}(0)$ are in one-to-one correspondence with the vertices of $\scrP$.
By rescaling $M_{\RR}$ if necessary, we can assume all vertices of $\scrP$
lie in $M$ and then the central fiber $\pi_{\scrP}^{-1}(0)$ is reduced. 

We also obtain a degeneration of the constraints $\mathbf{H}=(H_1,\dots,H_r)$ determined by the choice of the tropical constraints $\mathbf{A}=(A_1,
\dots,A_r)$.
Up to rescaling $\scrP$ if necessary, one can assume that $M \cap A_i \neq \emptyset$ for every $1\leq i \leq r$.
Then, one chooses a point $P_i \in M \cap A_i$ for every $1\leq i \leq r$. The point $P_i \in M$
determines a point $(P_i,1)\in \overline{M}$ and hence a one-parameter
subgroup $\GG_m(P_i,1)\subseteq \overline{M}\otimes \GG_m$. 
This can be viewed as a section $\sigma_i$
of the projection $\overline{M}\otimes\GG_m\rightarrow
\GG_m$ onto the last coordinate. For every $z\in \GG_m$, $\sigma_i(z)$
acts on the fiber of $\shX_{\scrP}\rightarrow \AA^1$ over $z\in\GG_m
\subset \AA^1$. Recall that the inverse image of $\GG_m$ in $\shX_{\scrP}$
is $X_{\Sigma_{\scrP}}\times\GG_m$. 
We then define $\overline{H_i}$ to be the closure of
the subset
\[
\bigcup_{z\in\GG_m} (\sigma_i(z)H_i,z) \subseteq X_{\Sigma_{\scrP}}\times\GG_m
\]
in $\shX_{\scrP}$. 
By construction, $\overline{H}_i$ is an hypersurface in the divisor $\overline{D_i}$ of $\shX_\scrP$ corresponding to the ray $\RR_{\geq 0}(\iota_{\gamma_i}\omega, 1)$ in $\overline{\Sigma}$. We denote the tuple
$\overline{\mathbf{H}}:=(\overline{H}_1,\dots,\overline{H}_r)$.

\subsection{Decomposition formula}
\label{sec_decomp}
We fix a good polyhedral decomposition $\scrP$
for a tropical data $(\omega, \boldsymbol{\gamma},\mathbf{A})$ with rational $\mathbf{A}$ as in \S\ref{sec_good_polyh}. 
In this section, we express the log Gromov--Witten invariants $N_\rho^{\mathrm{toric}}(X_{\Sigma_\scrP})$ of the general fiber of the toric degeneration $\pi_{\scrP}: \shX_{\scrP} \longrightarrow \AA^1$ in terms of log Gromov--Witten invariants of the special fiber.

We define as in Definition \ref{def_stable_log} the notion of $(\omega,\boldsymbol{\gamma})$-marked stable maps to $\shX_\scrP$, the only difference being that the contact orders at the marked points are now $(\iota_{\gamma_i}\omega,0)$ and $(\iota_\gamma \omega, 0)$ in $\overline{M}=M \oplus \Z$. We denote by 
$\cM_{\omega,\boldsymbol{\gamma}}^\log(\shX_\scrP/\AA^1)$ the moduli space of $(\omega,\boldsymbol{\gamma})$-marked stable maps to $\pi_{\scrP}:\shX_\scrP \rightarrow \AA^1$. 
We have a natural morphism
$\cM_{\omega,\boldsymbol{\gamma}}^\log(\shX_\scrP/\AA^1)
\rightarrow \AA^1$
whose fiber over $z\in\GG_m$ is $\cM^\log_{\omega,\boldsymbol{\gamma}}(X_{\Sigma_{\scrP}})$.

\begin{lemma} \label{lem_log_smooth}
The moduli space $\cM_{\omega,\boldsymbol{\gamma}}^\log(\shX_\scrP/\AA^1)$ is log smooth over $\AA^1$ and is of pure dimension $d-2+r+1$.
\end{lemma}

\begin{proof}
The relative log tangent bundle to $\shX_\scrP/\AA^1$ is trivial, given by 
$T^\log_{\shX_\scrP/\AA^1}=M \otimes \cO_{\shX_\scrP}$, and so we have $H^1(C,f^{\star}T^\log_{\shX_\scrP/\AA^1})=0$
for every genus $0$ stable log maps  $f: C \rightarrow \shX_\scrP/\AA^1$.
It follows from \cite{logGW} that relative deformations of genus $0$ stable log maps to $\shX_\scrP/\AA^1$ with fixed domain are unobstructed. Hence, it is enough to show that the moduli stack $\mathfrak{M}^{\log}(\AA^1)$ of non-necessarily basic prestable log maps to $\AA^1$ is log smooth over $\AA^1$. We have  $\mathfrak{M}^{\log}(\AA^1)=\mathrm{Log}_{\mathfrak{M} \times \AA^1}$, where $\mathfrak{M}$ is the moduli stack of prestable curves and $\mathrm{Log}_{\mathfrak{M} \times \AA^1}$ is Olsson's stack \cite{Olsson03} of log schemes over $\mathfrak{M} \times \AA^1$.
As $\mathfrak{M}$ is log smooth over the trivial log point, $\mathfrak{M} \times \AA^1$ is log smooth over $\AA^1$, and so the result follows because $\mathrm{Log}_{\mathfrak{M} \times \AA^1} \rightarrow \AA^1$ factors into $\mathrm{Log}_{\mathfrak{M} \times \AA^1} \rightarrow \mathfrak{M} \times \AA^1 \rightarrow \AA^1$, and for any log stack $S$ the natural morphism $\mathrm{Log}_S \rightarrow S$ is log \'etale by \cite[Theorem 4.6(iii)]{Olsson03}
(the corresponding morphism of stacks $\mathrm{Log}_S \rightarrow \mathrm{Log}_S$ is the identity and so is in particular \'etale). 
\end{proof}

We also define as in \eqref{eq_diag} 
the moduli space $\cM_{\omega,\boldsymbol{\gamma},\overline{\mathbf{H}}}^\log (\shX_\scrP/\AA^1)$ of $(\omega,\boldsymbol{\gamma})$-marked stable maps to $\shX_\scrP \rightarrow \AA^1$ matching the constraint $\overline{\mathbf{H}}$ by the fiber diagram (both in the schemes and fs log schemes categories)
\begin{equation} \label{eq_diag_deg}
\begin{tikzcd}
\cM_{\omega,\boldsymbol{\gamma},\overline{\mathbf{H}}}^\log(\shX_\scrP/\AA^1)
\arrow[r]
\arrow[d]
&
\cM_{\omega,\boldsymbol{\gamma}}^{\log}(\shX_\scrP/\AA^1)
\arrow[d]
\\
\prod_{i=1}^r \overline{H}_i
\arrow[r,"\iota_{\overline{H}}"]
& 
\prod_{i=1}^r \overline{D}_i\,,
\end{tikzcd}
\end{equation}
We have a natural morphism
$\cM_{\omega,\boldsymbol{\gamma},\overline{\mathbf{H}}}^\log(\shX_\scrP/\AA^1)
\rightarrow \AA^1$
whose fiber over $z\in\GG_m$ is $\cM^\log_{\omega,\boldsymbol{\gamma},\{\sigma_i(z)\cdot H_i\}_i}(X_{\Sigma_{\scrP}})$.

For every $1 \leq i\leq r$, denote by $\overline{A}_i$ the half-hyperplane in $\overline{M}_\RR$ obtained  as the closure in $\overline{M}_\RR$  of $\RR_{\geq 0} (A_i,1)$. By construction, one has a natural projection $\overline{A}_i \rightarrow \RR_{\geq 0}$ whose fiber over $1$ is $A_i$, and whose fiber over $0$ is $A_i^0=(\gamma_i^\perp)_\RR$.
As in \S \ref{sec_tropical_constraints}, we define the moduli space $\cM_{\omega,\boldsymbol{\gamma},\overline{\mathbf{A}}}^\trop(\overline{\Sigma}_\scrP)
$ of $(\omega,\boldsymbol{\gamma})$-marked tropical curves in $(\overline{M}_\RR, \overline{\Sigma}_\scrP)$ matching the constraint $\overline{\mathbf{A}}:=(\overline{A}_1,\dots,\overline{A}_r)$. There is a natural map 
$\cM_{\omega,\boldsymbol{\gamma},\overline{\mathbf{A}}}^\trop(\overline{\Sigma}_\scrP) \rightarrow \RR_{\geq 0}$ whose fiber over $1$ is $\cM_{\omega,\boldsymbol{\gamma},\mathbf{A}}^\trop(\scrP)$ and whose fiber over $0$
is $\cM_{\omega,\boldsymbol{\gamma},\mathbf{A}^0}^\trop(\Sigma_\scrP)$.
Using this, we obtain a natural map
\begin{equation}
   \label{Eq: phi} 
   \tilde{\Phi} \colon \{\mathrm{Faces~of~} \cM^\trop_{\omega,\boldsymbol{\gamma},\mathbf{A}}(\scrP)   \} \longrightarrow \{\mathrm{Faces~of~} \cM^\trop_{\omega,\boldsymbol{\gamma},\mathbf{A}^0}(\Sigma_\scrP)   \}\,,
\end{equation}
defined as follows. Given a face $\tilde{\sigma}$ of $\cM^\trop_{\omega,\boldsymbol{\gamma},\mathbf{A}}(\scrP)$, 
$\tilde{\Phi}(\tilde{\sigma})$ is the intersection of the fiber $\cM_{\omega,\boldsymbol{\gamma},\mathbf{A}^0}^\trop(\Sigma_\scrP)$ over $0$ with the face $\overline{\RR_{\geq 0}(\tilde{\sigma},1)}$ of $\cM_{\omega,\boldsymbol{\gamma},\overline{\mathbf{A}}}^\trop(\overline{\Sigma}_\scrP)$ obtained as the closure of the cone over $\tilde{\sigma}$.

\begin{definition}
\label{def_very_good}
A good polyhedral decomposition $\scrP$ of $M_\RR$ is \emph{very good} if:
\begin{itemize}
    \item[(i)] every vertex of $\scrP$ lies in $M$,
    \item[(ii)] for every $(d-2)$-dimensional face $\tilde{\sigma}$ of $\cM^\trop_{\omega,\boldsymbol{\gamma},\mathbf{A}}(\scrP)$, the natural projection $T_{\overline{\RR_{\geq 0}(\tilde{\sigma},1)}} \rightarrow \Z$ is surjective, where $T_{\overline{\RR_{\geq 0}(\tilde{\sigma},1)}}$ is the integral tangent space to the face $\overline{\RR_{\geq 0}(\tilde{\sigma},1)}$ of  $\cM_{\omega,\boldsymbol{\gamma},\overline{\mathbf{A}}}^\trop(\overline{\Sigma}_\scrP)$
\end{itemize}
\end{definition}

\begin{lemma} \label{lem_very_good}
For every tropical data $(\omega,\boldsymbol{\gamma},\mathbf{A})$ with rational $\mathbf{A}$, there exists a very good polyhedral decomposition for $(\omega,\boldsymbol{\gamma},\mathbf{A})$.
\end{lemma}

\begin{proof} By Lemma \ref{lem_good_poly}, good polyhedral decompositions exist.
The result follows because a good polyhedral decomposition can always be rescaled into a very good one. 
\end{proof}

From now on, we assume that $\scrP$ is a very good polyhedral decomposition.
Let $\tilde{\sigma}$ be a $(d-2)$-dimensional face of $\cM^\trop_{\omega,\boldsymbol{\gamma},\mathbf{A}}(\scrP)$. 
The relative interior of the face $\overline{\RR_{\geq 0}(\tilde{\sigma},1)}$ 
of $\cM_{\omega,\boldsymbol{\gamma},\overline{\mathbf{A}}}^\trop(\overline{\Sigma}_\scrP)$
parametrizes $(\omega,\boldsymbol{\gamma})$-marked tropical curves in $(\overline{M}_\RR,\overline{\Sigma}_\scrP)$ of a given type $\tau_{\tilde{\sigma}}$.
Let $\cM^\log_{\tilde{\sigma}}(\shX_\scrP/\AA^1)$ 
(resp.\, $\cM^\log_{\tilde{\sigma},\overline{\mathbf{H}}}(\shX_\scrP/\AA^1)$) be
the closure in $\cM^\log_{\omega,\boldsymbol{\gamma}}(\shX_\scrP/\AA^1)$
(resp.\ $\cM^\log_{\omega,\boldsymbol{\gamma},\overline{\mathbf{H}}}(\shX_\scrP/\AA^1)$) of the locus of stable log maps of type $\tau_{\tilde{\sigma}}$.
These moduli spaces fit in a fiber diagram 
(both in the schemes and fs log schemes categories)
\begin{equation} \label{eq_diag_deg1}
\begin{tikzcd}
\cM^\log_{\tilde{\sigma},\overline{\mathbf{H}}}(\shX_\scrP/\AA^1)
\arrow[r]
\arrow[d]
&
\cM^\log_{\tilde{\sigma}}(\shX_\scrP/\AA^1)
\arrow[d]
\\
\prod_{i=1}^r \overline{H}_i
\arrow[r,"\iota_{\overline{H}}"]
& 
\prod_{i=1}^r \overline{D}_i\,,
\end{tikzcd}
\end{equation}

Similarly, given a $(d-2)$-dimensional face 
$\rho$ of $\cM^\trop_{\omega,\boldsymbol{\gamma},\mathbf{A}^0}(\Sigma_\scrP)$, we denote by $\cM^\log_{\rho}(\shX_\scrP/\AA^1)$ (resp.\ $\cM^\log_{\rho,\overline{\mathbf{H}}}(\shX_\scrP/\AA^1)$) the closure in $\cM^\log_{\omega,\boldsymbol{\gamma}}(\shX_\scrP/\AA^1)$ 
(resp.\ $\cM^\log_{\omega,\boldsymbol{\gamma},\overline{\mathbf{H}}}(\shX_\scrP/\AA^1)$) of the locus of stable log maps of type 
determined by the interior of the face $\rho$. There is a natural map 
\[\cM^\log_{\rho,\overline{\mathbf{H}}}(\shX_\scrP/\AA^1) \longrightarrow \AA^1\] 
whose fiber over $1$ is the moduli space  
$\cM^\log_{\rho,\mathbf{H}}(X_{\Sigma_\scrP})$ defined in \S \ref{sec_log_gw_invts}. Moreover, 
the fiber over $0$ is the union of the moduli spaces $\cM^\log_{\tilde{\sigma},\overline{\mathbf{H}}}(\shX_\scrP/\AA^1)$ where 
 $\tilde{\sigma}$ are the faces of 
 $\cM^\trop_{\omega,\boldsymbol{\gamma},\mathbf{A}}(\scrP)$
such that $\tilde{\Phi}(\tilde{\sigma})=\rho$.

\begin{lemma} \label{lem_log_smooth1}
The moduli spaces $\cM^\log_{\tilde{\sigma}}(\shX_\scrP/\AA^1)$ and $\cM^\log_{\rho}(\shX_\scrP/\AA^1)$ are log smooth over $\AA^1$.
Moreover, $\cM^\log_{\tilde{\sigma}}(\shX_\scrP/\AA^1)$ is of pure dimension $r$ and $\cM^\log_{\rho}(\shX_\scrP/\AA^1)$ is of pure dimension $r+1$.
\end{lemma}

\begin{proof}
By Lemma \ref{lem_log_smooth}, the moduli space $\cM^\log_{\omega,\boldsymbol{\gamma}}(\shX_\scrP/\AA^1)$ is log smooth over $\AA^1$, of dimension $d-2+r+1$. 
The result follows from the fact that $\cM^\log_{\tilde{\sigma}}(\shX_\scrP/\AA^1)$ and $\cM^\log_{\rho}(\shX_\scrP/\AA^1)$ are log strata of $\cM^\log_{\omega,\boldsymbol{\gamma}}(\shX_\scrP/\AA^1)$.
\end{proof}

By Lemma \ref{lem_log_smooth1}, one can consider the $r$-dimensional fundamental class $[\cM^\log_{\tilde{\sigma}}(\shX_\scrP/\AA^1)]$. We define a $0$-dimensional virtual fundamental class on $\cM^\log_{\tilde{\sigma},\overline{\mathbf{H}}}(\shX_\scrP/\AA^1)$ by
\begin{equation} \label{eq_vclass2}[\cM^\log_{\tilde{\sigma},\overline{\mathbf{H}}}(\shX_\scrP/\AA^1)]^\virt := \iota_{\overline{H}}^! [\cM^\log_{\tilde{\sigma}}(\shX_\scrP/\AA^1)]\end{equation}
and the corresponding log Gromov--Witten invariant 
\begin{equation} \label{eq_N_tilde_sigma}
N_{\tilde{\sigma}}^{\mathrm{toric}} (\shX_\scrP/\AA^1) := \deg [\cM^\log_{\tilde{\sigma},\overline{\mathbf{H}}}(\shX_\scrP/\AA^1)]^\virt\,.\end{equation}

In the following theorem, we express the invariants $N_\rho^{\mathrm{toric}}(X_{\Sigma_\scrP})$ defined in 
    \eqref{eq_N_toric}
in terms of the invariants $N_{\tilde{\sigma}}^{\mathrm{toric}} (\shX_\scrP/\AA^1)$
defined in \eqref{eq_N_tilde_sigma}.
For every $(d-2)$-dimensional face $\rho$ of $\cM^\trop_{\omega,\boldsymbol{\gamma},\mathbf{A}^0}(\Sigma)$, denote by $\tilde{S}_{\rho}$ the set of $(d-2)$-dimensional faces $\tilde{\sigma}$
of $\cM^\trop_{\omega,\boldsymbol{\gamma},\mathbf{A}}(\scrP)$ such that $\Tilde{\Phi}(\Tilde{\sigma})=\rho$. 

\begin{theorem} \label{thm_decomp}
Let $\scrP$ be a very good polyhedral decomposition for $(\omega, \boldsymbol{\gamma},\mathbf{A})$,
and let $\rho$ be a $(d-2)$-dimensional face of $\cM^\trop_{\omega,\boldsymbol{\gamma},\mathbf{A}^0}(\Sigma)$. Then, we have 
\begin{equation} \label{eq_decomp} N_\rho^{\mathrm{toric}}(X_{\Sigma_{\scrP}})
= \sum_{\Tilde{\sigma}\in \Tilde{S}_\rho} N_{\tilde{\sigma}}^{\mathrm{toric}} (\shX_\scrP/\AA^1) \,.\end{equation}
\end{theorem}

\begin{proof}
By Lemma \ref{lem_log_smooth1}, the moduli space $\cM^\log_{\rho}(\shX_\scrP/\AA^1)$ is log smooth over $\AA^1$ and of dimension $r+1$. The general fiber of $\cM^\log_{\rho}(\shX_\scrP/\AA^1)$ is the $r$-dimensional moduli space $\cM^\log_\rho(X_{\Sigma_\scrP})$, whereas the irreducible components of the central fiber over $0 \in \AA^1$ are exactly the $r$-dimensional moduli spaces $\cM^\log_{\tilde{\sigma}}(\shX_\scrP/\AA^1)$ for $\Tilde{\sigma} \in \Tilde{S}_\rho$.
Hence, it follows from the decomposition result of \cite[Corollary 3.2]{ACGSI} for the fibers of log smooth morphisms that 
\begin{equation} \label{eq_decomp_classes}[\cM^\log_\rho(X_{\Sigma_\scrP})]=\sum_{\Tilde{\sigma}\in \Tilde{S}_\rho}[\cM^\log_{\tilde{\sigma}}(\shX_\scrP/\AA^1)]\,.\end{equation}
The multiplicities present in the general formula of \cite[Corollary 3.2]{ACGSI}  are all equal to $1$ in our situation because we are assuming that $\scrP$ is a very good polyhedral decomposition and so Definition \ref{def_very_good}(ii) is satisfied (geometrically, 
the central fiber of $\cM^\log_{\rho}(\shX_\scrP/\AA^1) \rightarrow \AA^1$ is reduced). Now, applying $\iota_{\overline{H}}^{!}$ to both sides of \eqref{eq_decomp_classes}, we obtain by \eqref{eq_vclass1}-\eqref{eq_vclass2} that 
\begin{equation}
[\cM^\log_{\rho,\mathbf{H}}(X_{\Sigma_\scrP})]^\virt=\sum_{\Tilde{\sigma}\in \Tilde{S}_\rho}[\cM^\log_{\tilde{\sigma},\overline{\mathbf{H}}}(\shX_\scrP/\AA^1)]^\virt\,.\end{equation}
Taking the degree on both sides gives \eqref{eq_decomp} by \eqref{eq_N_toric}-\eqref{eq_N_tilde_sigma}.
\end{proof}

\subsection{The log-tropical correspondence for a fixed tropical type}

In this section, we relate the log Gromov--Witten invariants $N_{\Tilde{\sigma}}^{\mathrm{toric}}(\shX_\scrP/\AA^1)$ defined in \eqref{eq_N_tilde_sigma} with the tropical multiplicities $N_\sigma^\trop$ defined in \eqref{eq_N_trop}.

Let $\tilde{\sigma}$ be a $(d-2)$-dimensional face 
of $\cM^\trop_{\omega,\boldsymbol{\gamma},\mathbf{A}}(\scrP)$ such that 
$\dim \fod_{L_{\mathrm{out}}}^{\Tilde{\sigma}}=d-1$ and $\dim \Tilde{\Phi}(\Tilde{\sigma})=d-2$. Tropical curves in the relative interior of $\Tilde{\sigma}$ are of type $\tau_{\tilde{\sigma}}$, and we defined in
\S \ref{sec_decomp} the moduli spaces $\cM^\log_{\tilde{\sigma}}(\shX_\scrP/\AA^1)$ and $\cM^\log_{\tilde{\sigma},\overline{\mathbf{H}}}(\shX_\scrP/\AA^1)$ 
as closures of the loci of stable log maps of type $\tau_{\tilde{\sigma}}$.
 The following lemma
 \ref{lem_closed}
 shows that stable log maps in 
$\cM^\log_{\tilde{\sigma},\overline{\mathbf{H}}}(\shX_\scrP/\AA^1)$ are automatically of type $\tau_{\tilde{\sigma}}$, and so taking the closure was actually unnecessary in this case.
Denote by $\cM^{\log,0}_{\tilde{\sigma}}(\shX_\scrP/\AA^1)$ the open dense locus in $\cM^\log_{\tilde{\sigma}}(\shX_\scrP/\AA^1)$ consisting of stable log maps of type $\tau_{\tilde{\sigma}}$.

\begin{lemma} \label{lem_closed}
Stable log maps in $\cM^\log_{\tilde{\sigma},\overline{\mathbf{H}}}(\shX_\scrP/\AA^1)$  are of type $\tau_{\tilde{\sigma}}$, that is, we have
\[ \cM^\log_{\tilde{\sigma},\overline{\mathbf{H}}}(\shX_\scrP/\AA^1) \subset \cM^{\log,0}_{\tilde{\sigma}}(\shX_\scrP/\AA^1) \,. \]
\end{lemma}

\begin{proof}
If there were a stable log map in $\cM^\log_{\tilde{\sigma},\overline{\mathbf{H}}}(\shX_\scrP/\AA^1)$ of type $\tau \neq \tau_{\Tilde{\sigma}}$, then the closure of the tropical curves of type $\tau$ would be a face of $\cM^\trop_{\omega,\boldsymbol{\gamma},\mathbf{A}}(\scrP)$ 
containing strictly $\tau_{\Tilde{\sigma}}$, and so in particular of dimension $> d-2$. But as the $\boldsymbol\gamma$-constraint $\mathbf{A}$
is general, every face of
$\cM^\trop_{\omega,\boldsymbol{\gamma},\mathbf{A}}$ is of dimension $\leq d-2$ by Definition \ref{def_general_constraints}(ii) and the same is true for   $\cM^\trop_{\omega,\boldsymbol{\gamma},\mathbf{A}}(\scrP)$ because $\cM^\trop_{\omega,\boldsymbol{\gamma},\mathbf{A}}(\scrP)$
is a polyhedral refinement of $\cM^\trop_{\omega,\boldsymbol{\gamma},\mathbf{A}}$
by \S\ref{sec_polyh}. Hence, we obtained a contradiction.
\end{proof}

By Lemma \ref{lem_closed}, it will be enough to study stable log maps of type $\tau_{\Tilde{\sigma}}$. We give below an explicit description of stable log maps of type $\tau_{\Tilde{\sigma}}$.

By Definition \ref{def:good polyhedral} of a good polyhedral decomposition, for each vertex $v$ of $\Gamma_{\Tilde{\sigma}}$, $\foj_v^{\Tilde{\sigma}}$ lies in the interior of an
$n-2$-dimensional cell $\mathcal{P}_v$ of $\scrP$, and for each edge or leg  $E$,
$\fod_E^{\Tilde{\sigma}}$ lies in the interior of an $n-1$-dimensional cell
$\mathcal{P}_E$ of $\scrP$. Let $Z_v, Z_{E}\subseteq \shX_{\scrP}$ be the
closed toric strata corresponding to $\mathcal{P}_v, \mathcal{P}_E$ respectively. The strata $Z_v$ and $Z_E$ are toric subvarieties of $\mathcal{X}_\scrP$ with $\dim Z_v=2$ and $\dim Z_E=1$. Denote by $T_v$ and $T_E$ the integral tangent spaces to $\mathcal{P}_v$ and $\mathcal{P}_E$ respectively. Then, $Z_v$ is a toric surface
given by a fan in the lattice
\[
M_v:=M/T_v\,,
\]
and we have $Z_E \simeq \PP^1$ with fan naturally in $M_E:=M/T_E$.

If $f: C/W \rightarrow \shX_{\scrP}/\AA^1$ is a stable log map of type $\tau_{\Tilde{\sigma}}$, then the dual graph of $C$ is given by the graph $\Gamma_{\Tilde{\sigma}}$. We denote by $\underline{C}_v \simeq \PP^1$ the irreducible component corresponding to $v \in V(\Gamma_{\Tilde{\sigma}})$. Moreover, for every edge or leg $E$ of $\Gamma_{\Tilde{\sigma}}$, we denote by $x_E$ the corresponding special point (node or marked point) on $C$. For every $v \in V(\Gamma_{\Tilde{\sigma}})$, we have $f(C_v) \subset Z_v$ and the only points of $C_v$ mapped by $f$ in the toric boundary of $Z_v$ are the points $x_E$ for $E$ an edge or leg adjacent to $v$. In addition, the point $f(x_E)$ is contained in the interior of the big torus orbit $\GG_m$ of $Z_E \simeq \PP^1$, and the contact order of $f(C_v)$ to $Z_E$ at $f(x_E)$ is specified by the image $\bar{u}_{v,E}$ in $M_v=M/T_v$ of the weighted direction $u_{v,E}\in M$. In particular, the morphism of schemes $\underline{f}_v \colon \underline{C}_v \rightarrow Z_v$ obtained by restricting $f$ to $\underline{C}_v$ is torically transverse as in \cite[Definition 4.1]{NS}. In fact, as the vertices of $\Gamma_{\Tilde{\sigma}}$ are either divalent or trivalent, $\underline{f}_v \colon \underline{C}_v \rightarrow Z_v$ is a line in the sense of \cite[Definition 5.1]{NS}.
Moreover, if $E$ is the leg $L_i$ of $\Gamma_{\Tilde{\sigma}}$ for some $1\leq i \leq r$, then the intersection $Z_{L_i} \cap \overline{H_i}$ consists of a single point, because the defining equation of $\overline{H}_i$ is $z^{\frac{\gamma_i}{|\gamma_i|}}=c_i$ and $\frac{\gamma_i}{|\gamma_i|}$ is primitive. The stable log map $f: C/W \rightarrow \shX_{\scrP}/\AA^1$ matches $\overline{\mathbf{H}}$ if and only if $\{f(x_{L_i})\}=Z_{L_i} \cap \overline{H}_i$ for all $1 \leq i \leq r$.

\begin{lemma}  \label{lem_reduced}
The moduli stack $\cM^\log_{\tilde{\sigma},\overline{\mathbf{H}}}(\shX_\scrP/\AA^1)$ is reduced and $0$-dimensional.
\end{lemma}

\begin{proof}
The deformation theory of stable log maps to $\shX_\scrP/\AA^1$ is controlled by the log tangent bundle $T_{\shX_\scrP/\AA^1}^{\log}=M \otimes \cO_{\shX_\scrP}$. Let $f: C/W \rightarrow \shX_{\scrP}/\AA^1$ be a stable log map in $\cM^\log_{\tilde{\sigma},\overline{\mathbf{H}}}(\shX_\scrP/\AA^1)$. The map $f$ is of type $\tau_{\Tilde{\sigma}}$ by Lemma \ref{lem_closed}.
As $C$ is of genus $0$, we have $H^1(C,
f^\star T_{\shX_\scrP/\AA^1}^{\log})=0$, and so the tangent space to $f$ in $\cM^{\log}_{\omega,\boldsymbol{\gamma},\mathbf{H}}(\shX_\scrP/\AA^1)$ is $T_f := H^0(C,N_f)$, where $N_f:= f^\star T_{\shX_\scrP/\AA^1}^{\log}/T_{C/W}^\log$. As $f$ is a union of lines $\underline{f}_v: \underline{C}_v \rightarrow Z_v$, one checks as in the proof of \cite[Theorem 7.3]{NS} that $T_f=\ker \mathrm{gl}_{\Tilde{\sigma}} \otimes \kk$, where $\mathrm{gl}_{\Tilde{\sigma}}$ is the gluing map
\[\mathrm{gl}_{\Tilde{\sigma}}: 
\prod_{v \in V(\Gamma_{\Tilde{\sigma}})} M 
\longrightarrow \prod_{E\in E(\Gamma_{\Tilde{\sigma}})} M/\Z u_E \times \prod_{i=1}^r M/\gamma_i^{\perp}
\]
as in \eqref{eq_gl}.  
On the other hand, the tangent space $\overline{T}_f$ to $f$ in $\cM^\log_{\tilde{\sigma},\overline{\mathbf{H}}}(\shX_\scrP/\AA^1)$ is obtained by restricting the infinitesimal deformations to those tangent to the strata $Z_v$ and $Z_E$, and so is the kernel of the restricted map $\overline{\mathrm{gl}}_{\Tilde{\sigma}}\otimes \kk$, where 
\[ \overline{\mathrm{gl}}_{\Tilde{\sigma}}: 
\prod_{v \in V(\Gamma_{\Tilde{\sigma}})} M_v 
\longrightarrow \prod_{E\in E(\Gamma_{\Tilde{\sigma}})} M_E \times \prod_{i=1}^r M/\gamma_i^{\perp}\,.
\]
As $\dim \Tilde{\sigma}=d-2$, the map $\mathrm{gl}_{\Tilde{\sigma}} \otimes \kk$ is surjective by Lemma \ref{lem_finite}. The natural projections $M \rightarrow M_v$ and $M \rightarrow M_E$ induce a commutative diagram 
\[\begin{tikzcd}
\prod_{v \in V(\Gamma_{\Tilde{\sigma}})} M \arrow[r,"\mathrm{gl}_{\Tilde{\sigma}}"]
\arrow[d]
&
\prod_{E\in E(\Gamma_{\Tilde{\sigma}})} M \times \prod_{i=1}^r M/\gamma_i^{\perp} 
\arrow[d]
\\
\prod_{v \in V(\Gamma_{\Tilde{\sigma}})} M_v  \arrow[r,"\overline{\mathrm{gl}}_{\Tilde{\sigma}}"] &
\prod_{E\in E(\Gamma_{\Tilde{\sigma}})} M_E \times \prod_{i=1}^r M/\gamma_i^{\perp}\,,
\end{tikzcd}\]
As the right vertical arrow of this diagram is surjective, the surjectivity of $\mathrm{gl}_{\Tilde{\sigma}} \otimes \kk$ implies the surjectivity of $\overline{\mathrm{gl}}_{\Tilde{\sigma}} \otimes \kk$. Finally, because the trivalent graph obtained from $\Gamma_{\Tilde{\sigma}}$ by erasing the divalent vertices has $r+1$ legs, $r-1$ vertices, and $r-2$ edges, the difference between the dimensions of the domain and the codomain of $\overline{\mathrm{gl}}_{\Tilde{\sigma}} \otimes \kk$ is $2(r-1)-(r-2+r)=0$, and so $\overline{\mathrm{gl}}_{\Tilde{\sigma}} \otimes \kk$ is actually an isomorphism. In particular, $\overline{T}_f=0$ and so 
$\cM^\log_{\tilde{\sigma},\overline{\mathbf{H}}}(\shX_\scrP/\AA^1)$ is reduced of dimension $0$ at the point $f$. 
\end{proof}

\begin{lemma}\label{lem_enum}
The log Gromov--Witten invariant $N_{\Tilde{\sigma}}^\mathrm{toric}$ is equal to the count with automorphisms of the $(\omega,\boldsymbol{\gamma})$-marked stable log maps to $\shX_\scrP/\AA^1$ of type $\tau_{\tilde{\sigma}}$ and matching $\overline{\mathbf{H}}$.
\end{lemma}

\begin{proof}
By Lemma \ref{lem_reduced}, $\cM^\log_{\tilde{\sigma},\overline{\mathbf{H}}}(\shX_\scrP/\AA^1)$ is a $0$-dimensional scheme and so \eqref{eq_vclass2} implies that 
$[\cM^\log_{\tilde{\sigma},\overline{\mathbf{H}}}(\shX_\scrP/\AA^1)]^\virt=[\cM^\log_{\tilde{\sigma},\overline{\mathbf{H}}}(\shX_\scrP/\AA^1)]$, and so $N_{\Tilde{\sigma}}^\trop=\deg [\cM^\log_{\tilde{\sigma},\overline{\mathbf{H}}}(\shX_\scrP/\AA^1)]$ by \eqref{eq_N_tilde_sigma}. 
By Lemma \ref{lem_reduced}, $\cM^\log_{\tilde{\sigma},\overline{\mathbf{H}}}(\shX_\scrP/\AA^1)$ is also reduced, and so $N_{\Tilde{\sigma}}^\trop$ is equal to the count with automorphisms of points in $\cM^\log_{\tilde{\sigma},\overline{\mathbf{H}}}(\shX_\scrP/\AA^1)$. 
\end{proof}

Recall from \S\ref{sec_polyh} that $\Tilde{\sigma}$ is contained in a $(d-2)$-dimensional face $\sigma$ of $\cM_{\omega, \boldsymbol{\gamma},\mathbf{A}}$, that we denote by $T_\sigma$ and $T_{\Tilde{\sigma}}$ the integral tangent spaces to $\sigma$ and $\Tilde{\sigma}$ respectively, and that the natural inclusion $T_{\Tilde{\sigma}} \subset T_\sigma$ is of finite index by Lemma \ref{lem:two gammas compare}.
We can now state the log-tropical correspondence theorem for a fixed face $\Tilde{\sigma}$.

\begin{theorem} \label{thm_count}
Fix a tropical data $(\omega,\boldsymbol{\gamma},\mathbf{A})$ with rational $\mathbf{A}$ as in \eqref{Eq:tropical data}, and let $\scrP$ be a very good polyhedral decomposition for $(\omega,\boldsymbol{\gamma},\mathbf{A})$. Let $\tilde{\sigma}$ be a $(d-2)$-dimensional face 
of $\cM^\trop_{\omega,\boldsymbol{\gamma},\mathbf{A}}(\scrP)$ such that 
$\dim \fod_{L_{\mathrm{out}}}^{\Tilde{\sigma}}=d-1$ and $\dim \Tilde{\Phi}(\Tilde{\sigma})=d-2$, and let $\sigma$ be the
$(d-2)$-dimensional face of $\cM_{\omega, \boldsymbol{\gamma},\mathbf{A}}$ containing $\Tilde{\sigma}$.
Then, we have 
\begin{equation}
N_{\tilde{\sigma}}^{\mathrm{toric}} (\shX_\scrP/\AA^1)= \frac{N_\sigma^\trop}{|T_\sigma/T_{\Tilde{\sigma}}|}\,,
\end{equation}
where 
$N_{\Tilde{\sigma}}^{\mathrm{toric}}(\shX_\scrP/\AA^1)$ is the log Gromov--Witten invariant defined in \eqref{eq_N_tilde_sigma} and $N_\sigma^\trop$ is the tropical multiplicity defined in \eqref{eq_N_trop}.
\end{theorem}

\begin{proof}
By Lemma \ref{lem_enum}, $N_{\Tilde{\sigma}}^\mathrm{toric}$ is equal to the count with automorphisms of the $(\omega,\boldsymbol{\gamma})$-marked stable log maps to $\shX_\scrP/\AA^1$ of type $\tau_{\tilde{\sigma}}$ and matching $\overline{\mathbf{H}}$.
Let $\underline{C}$ be the scheme-theoretic domain curve of stable log maps of type $\tau_{\Tilde{\sigma}}$, with irreducible components $\underline{C}_v \simeq \PP^1$ for $v \in V(\Gamma_{\Tilde{\sigma}
})$ and special points (nodes and marked points) $x_E$ for $E \in E(\Gamma_{\Tilde{\sigma}})\cup L(\Gamma_{\Tilde{\sigma}})$. Imitating the terminology of \cite[Definition 4.3]{NS}, we will first count the number of pre-log curves to $\shX_\scrP/\AA^1$ of type $\tau_{\Tilde{\sigma}}$ and matching $\overline{\mathbf{H}}$, that is, the number of scheme morphisms $\underline{f}\colon  \underline{C} \rightarrow \shX_\scrP$, such that:
\begin{itemize}
    \item[(i)] $\underline{f}(\underline{C}_v)\subset Z_v$ for every $v \in V(\Gamma_{\Tilde{\sigma}})$,
    \item[(ii)] the only points of $\underline{C}_v$ mapped by $\underline{f}$ in the toric boundary of $Z_v$ are the points $x_E$ for $E$ an edge or leg adjacent to $v$,  
\item[(iii)] the contact order of $f(C_v)$ to $Z_E$ at $f(x_E)$ is specified by the image $\bar{u}_{v,E}$ in $M_v=M/T_v$ of the weighted direction $u_{v,E}\in M$,
\item[(iv)] 
the point $f(x_E)$ is contained in the interior of the big torus orbit $\GG_m$ of $Z_E \simeq \PP^1$ for every $E \in E(\Gamma_{\Tilde{\sigma}})\cap L(\Gamma_{\Tilde{\sigma}})$,
\item[(v)] 
$\{\underline{f}(x_{L_i})\}=Z_{L_i} \cap \overline{H}_i$
for every $1\leq i \leq r$. 
\end{itemize}
We will then count the number of ways to lift a pre-log curve $\underline{f}:\underline{C}\rightarrow \shX_{\scrP}$ to a stable log curve $f: C/W \rightarrow \shX_{\scrP}/\AA^1$.

To count the number of pre-log curves of type $\tau_{\Tilde{\sigma}}$ and matching $\mathbf{H}$, we proceed by induction following the flow on $\Gamma_{\Tilde{\sigma}}$ starting at the leaves and ending at the root.
We use the notations introduced in the proofs of Lemmas \ref{lem: product for Ntrop}-\ref{lem: product2 for Ntrop}-\ref{lem:coker psi}, but also applied to $\Gamma_{\Tilde{\sigma}}$,
so that a vertex $v$ may be divalent also, so in particular
there may be only a $\Gamma_{\tilde{\sigma}_1}$ and not both a $\Gamma_{\tilde{\sigma}_1}$,
$\Gamma_{\tilde{\sigma}_2}$ attached to a vertex $v$. 

Let $v$ be a vertex of $\Gamma_{\Tilde{\sigma}}$.
If $v$ is divalent, we are 
in a situation similar to \cite[Proposition 5.5]{NS}: we have a $\PP^1$-fibration
structure $Z_v\rightarrow \PP^1$ induced by the map of fans given
by
\[
M/T_v\rightarrow M/\mathcal{L}_{1,v}^{\sat}\simeq \Z\,.
\]
Recall from \S\ref{sec_polyh} that 
we denote $w_v := 
|\overline{u}_E|$, where $E$ is an edge adjacent to $v$
and $\overline{u}_E$ is the image of $u_E$ in the rank two lattice $M/T_v$. 
Then the vertex $v$ corresponds to a component $\underline{C}_v$ of the pre-log curve
which maps $w_v:1$ to a fiber of $Z_v\rightarrow \PP^1$, totally
ramified over the two one-dimensional toric strata of $Z_v$ this fiber 
intersects. Because this stable map has an automorphism group 
$\Z/w_v \Z$, its count with automorphism is $w_v^{-1}$.
If the child edge of $v$ is not a leg, we deduce that $N_{\Tilde{\sigma}_{\mathrm{out}}}=w_v^{-1}N_{\Tilde{\sigma}_1}$.
If the child edge of $v$ is a leg $L_i$, then  $f(x_{L_i})$ is constrained to equal the point $Z_{L_i} \cap \overline{H}_i$, and so $N_{\Tilde{\sigma}_{\mathrm{out}}}=w_v^{-1}$.

Next consider the case with $v$ trivalent.
Assume first that $E_1$ and $E_2$ are not legs.
Then $Z_v$ is a toric surface, and
the edges or legs $E_{1}$, $E_{2}$ and $E_{\out}$ determine three divisors
$Z_{1}=Z_{E_{1}}, Z_2=Z_{E_{2}}$ and $Z_{\out}=Z_{E_{\out}}$ 
of $Z_v$. We then need to count
lines in the sense of \cite[Definition 5.1]{NS} which intersect the boundary 
of $Z_v$ in three points, with maximal tangency order, and
meeting $Z_j$ at one of the $N_{\Tilde{\sigma}_j}$ points where the pre-log curves of type $\tau_{\Tilde{\sigma}_j}$
meet $Z_v$. Given fixed 
input locations,
the number of such curves is 
\[
N_v:=\left\lvert \frac{u_{E_1}}{|u_{E_1}|}\wedge \frac{u_{E_2}}{|u_{E_2}|}\right\rvert \,,
\]
by \cite[Proposition 5.7]{NS}. This can be rewritten as
\begin{equation}
\label{eq:Nv}
N_v:=\big|M/(\mathcal{L}_{1,v}^{\sat}+\mathcal{L}_{2,v}^{\sat})\big|.
\end{equation}
Then we obtain $N_{\Tilde{\sigma}_{\mathrm{out}}}=N_v
N_{\Tilde{\sigma}_1} N_{\Tilde{\sigma}_2}$.
If $E_1$ is an edge and $E_2$ is a leg, we similarly obtain $N_{\Tilde{\sigma}_{\mathrm{out}}}=N_v
N_{\Tilde{\sigma}_1}$, and if both $E_1$ and $E_2$ are legs, we have $N_{\Tilde{\sigma}_{\mathrm{out}}}=N_v$.
Putting this all together, the count with automorphisms of pre-log curves of type $\tau_{\Tilde{\sigma}}$ and matching $\mathbf{H}$ is
\begin{equation}
\label{eq:prelog number}
\left(\prod_{\hbox{$v$ divalent}} w_v^{-1}\right)
\left(\prod_{\hbox{$v$ trivalent}} N_v\right).
\end{equation}
We may now calculate the number of logarithmic enhancements of each
pre-log curve, that is the number of lifts of each pre-log curve to a stable log curve. We use the theory
of \emph{punctured log curves}, introduced in \cite{ACGSII}. Examples of punctured log curves are obtained by restricting the log structure of a log curve to one of its irreducible components. Conversely, punctured log curves can be glued together to produce log curves.
First note that the cones corresponding to $Z_v$ and $Z_E$ in the fan of $\mathcal{X}_\scrP$ are the cones $\mathbf{C}\scrP_v$ and $\mathbf{C}\scrP_E$ over $\scrP_v$ and $\scrP_E$ respectively. We will denote by $P_v$ and $P_E$ the monoids of integral points in the cones $\Hom(\mathbf{C}\scrP_v,\Z)$ and $\Hom(\mathbf{C}\scrP_E,\Z)$, dual to $\mathbf{C}\scrP_v$ and $\mathbf{C}\scrP_E$, respectively. Then, the monoids $P_v$ and $P_E$ are stalks of the ghost sheaf $\overline{\mathcal{M}}_{\mathcal{X}_\scrP}$ at generic points of $Z_v$ and $Z_E$ respectively.

Now, consider one of the pre-log curves $\ul{f}:\ul{C}\rightarrow
\shX_{\scrP}$ constructed above. For every vertex $v \in V(\Gamma_{\Tilde{\sigma}})$, we write
$\ul{f}_v:\ul{C}_v\rightarrow \ul{\shX_{\scrP}}$ for the restriction of $\ul{f}$ to the irreducible component $\ul{C}_v \simeq \PP^1$ of $\ul{C}$ corresponding to $v$.

We first note that we can enhance the stable map $\ul{f}_v$ to a basic
punctured log map $f_v:C_v^\circ/W_v\rightarrow\shX_{\scrP}$ in a unique way so that the
punctured points have contact orders given by the weighted tangent vectors
to the edges adjacent to $v$ (and pointing away from $v$).
The base $W_v$ needs to be the log point $W_v :=
\Spec(Q_v \rightarrow \kk)$ with ghost sheaf $Q_v$ given by the basic monoid. By definition the basic monoid is the dual of the moduli space of tropical maps with domain the dual graph of $
\underline{C}_v$: here the only modulus is the location
of the vertex, which is constrained to lie in ${\bf C}\mathcal{P}_v$, and
trace out a $(d-2)+1$-dimensional subcone $\mathcal{C}_v$ of that cone. The basic monoid $Q_v$ is the monoid of integral points of this cone $\mathcal{C}_v$. Note that, the inclusion $\mathcal{C}_v \subset {\bf C}\mathcal{P}_v$ induces a projection
\begin{equation}
\label{Eq:phiv}
     \phi_v \colon P_v \longrightarrow Q_v \,.
\end{equation}
Consider the curve $\ul C_v$, and recall that we denote by $x_E$ the marked point corresponding to an edge or leg $E$ adjacent to $v$. As in \cite{katoF}, the curve $\ul C_v$ lifts uniquely to a log curve, which is log smooth over $W_v$ with marked points $x_E$. We use the notation $(C_v,\mathcal{M}_{C_v})$ to denote this log curve. 

On the complements of the marked points on $C_v$, the map $\ul{f}_v$ admits a unique lift to a log map $f_v$. Indeed, for $p$ in $\ul C_v$, which is not a marked point, $\ul{f}_v(p)$ lies in the interior of $Z_v$, so we have $\overline{\cM}_{X_\scrP, \ul{f}_v(p)} = P_v$. Then, locally $f_v$ is given by the map:
\begin{align*}
\mathcal{M}_{X_\scrP,\ul f_v(p)} =  \mathcal{O}^*_{\ul X_\scrP,f(p)}  \oplus P_v & \longrightarrow \mathcal{M}_{C_v,p} =  \mathcal{O}^*_{\ul C_v,p} \oplus Q_v \\    
    (s,r) & \longmapsto (\ul{f}_v^{\star}(s),\phi_v(r))
\end{align*}
 
For a marked point $x_E$, we will show in a moment that locally near $x_E$, there exists a unique puncturing $C_v^\circ$ of $C_v$ and the map $\ul{f}_v$ admits a unique lift to a punctured log map $f_v : C_v^\circ \to X_\scrP$. First note that the inclusion $\mathbf{C}\mathcal{P}_v \subset \mathbf{C}\mathcal{P}_E$ induces a projection $P_E \to P_v$ whose composition with $\phi_v$ in \eqref{Eq:phiv} gives a map
\begin{equation} \label{eq_phi_vE}
\phi_{v,E} \colon P_E \longrightarrow Q_v \end{equation}
Moreover, the weighted direction $u_E$ of the edge $E$ is an integral tangent vector to $\mathbf{C}\mathcal{P}_E$, hence can be viewed as an element in $P_E^* \coloneqq \Hom(P_E,\Z)$. Denote by $\mathcal{M}_{(\ul C_v, x_E)}$ the divisorial log structure on $\ul C_v$, defined by the marked point $x_E$ viewed as a divisor. Then, for the stalk of the ghost sheaf, we have $\overline{\mathcal{M}}_{(\ul C_v, x_E),x_E} = \NN$. So, from \cite{katoF} we have
\[ \mathcal{M}_{C_v,x_E} = ( \mathcal{O}^*_{\ul C_v,x_E} \oplus Q_v )  \oplus_{\mathcal{O}^*_{\ul C_v,x_E}} \mathcal{M}_{(C_v,x_E),x_E} \,.\]
Let $x=0$ be a local equation of $x_E$ in $\ul C_v$. One can naturally view $x$ as an element of $\mathcal{M}_{(C_v,x_E),x_E}$. Denote by $\mathcal{M}_{C_v^\circ,x_E} $ the submonoid of $ ( \mathcal{O}^*_{\ul C_v,x_E} \oplus Q_v )  \oplus_{\mathcal{O}^*_{\ul C_v,x_E}} \mathcal{M}^{\gp}_{(C_v,x_E),x_E} $ generated by $ \mathcal{O}^*_{\ul C_v,x_E} \oplus Q_v $ and elements of the form $(s,\phi_{v,E}(r),x^{u_E(r)})$, for $s \in \mathcal{O}^*_{\ul C_v,x_E}$ and $r\in P_E$. Then, $\mathcal{M}_{C_v^\circ,x_E} $ is the unique puncturing of $C_v$ at $x_E$ such that locally $\ul f_v$ lifts to a punctured log map $f_v: C_v^0 \to \mathcal{X}_\scrP$. This map is given by
\begin{align*}
\mathcal{M}_{X_\scrP,\ul f_v(x_E)} =  \mathcal{O}^*_{\ul X_\scrP,\ul{f}_v(x_E)}  \oplus P_E & \longrightarrow \mathcal{M}_{C_v^\circ,x_E} \subset  
( \mathcal{O}^*_{\ul C_v,x_E} \oplus Q_v )  \oplus_{\mathcal{O}^*_{\ul C_v,x_E}} \mathcal{M}^{\gp}_{(C_v,x_E),x_E}  \\    
    (s,r) & \longmapsto (\ul{f}_v^{\star}(s),\phi_{v,E}(r), x^{u_E(r)})\,,
\end{align*}
Here we use that $\ul{f}(x_E)$ is contained in the interior of $Z_E$, and so $\bar{\mathcal{M}}_{X_\scrP,\ul{f}_v(x_E)} = P_E$.

This shows that the restriction of $\ul{f} \colon \ul C \to \ul{\mathcal{X}}_\scrP$, to any irreducible component $\ul C_v$ lifts uniquely to a punctured log map. In the remaining part of the proof, we will count the number of ways to glue these punctured maps to a stable log map $f \colon C \to \mathcal{X}_\scrP$. For this we will follow \cite{gross2023remarks}.

We consider punctured log maps $f_v \colon C_v^\circ/W_v \rightarrow X_{\scrP}$ over the log points $W_v=\Spec (Q_v \rightarrow \kk)$, which we want to glue at marked points $x_E$ such that $\overline{\mathcal{M}}_{X_\scrP,\ul{f}_v(x_E)} = P_E$.
Let $Q_v^*=\Hom(Q_v,\Z)$ and $P_E^*=\Hom(P_E,\Z)$. As in \cite[Defn.4.2]{gross2023remarks}, there is a gluing map
\[
\overline\Psi:\prod_{v\in V(\Gamma_{\Tilde{\sigma}})} Q_v^*\times \prod_{E\in E(\Gamma_{\Tilde{\sigma}})}
\Z\longrightarrow \prod_{E\in E(\Gamma_{\Tilde{\sigma}})} P_E^* \,,
\]
given by
\[
\overline\Psi\left((q_v),(\ell_E)\right)=(\phi_{v_E,E}^{\star}(q_{v_E})+\ell_E u_E 
-\phi_{v'_E, E}^{\star}(q_{v'_E}))_{E\in E(\tilde\Gamma)}\,,
\]
where for every vertex $v$ adjacent to an edge $E$, 
$\phi_{v,E}^{\star} : Q_v^\star \rightarrow  P_E^{\star}$ is the dual of the map $\phi_{v,E}$ in 
\eqref{eq_phi_vE}, where
each edge $E$ is oriented from endpoints $v_E$ to $v'_E$ and
$u_E$ is the weighted tangent vector to the edge $E$ pointing from
$v_E$ to $v'_E$.

We explain how to compare the map $\overline\Psi$ with the map  $\Psi_{\Tilde{\sigma}}$ in \eqref{Eq: WidetildePsi}, following \cite[\S 5.1]{gross2023remarks}. 
The morphism $\shX_{\scrP}\rightarrow \AA^1$ induces maps $Q_v^*\rightarrow\Z$ and
$P_E^*\rightarrow\Z$ with kernels isomorphic to $T_v$ and $T_E$ respectively. 
We then obtain a diagram
\[
\xymatrix@C=30pt
{
0\ar[r]&\prod_v T_v\times \prod_E \Z\ar[r]\ar[d]_{\Psi_{\Tilde{\sigma}}}&
\prod_v Q_v^*\times \prod_E\Z\ar[r]\ar[d]_{\overline\Psi}&
\prod_v \Z\ar[d]_{\partial}\ar[r]& 0\\
0\ar[r]&\prod_E T_E\ar[r]&\prod_E P_E^*\ar[r]&\prod_E\Z\ar[r]&0
}
\]
Here the maps to $\prod\Z$ are induced by the above maps 
$Q_v^*, P_E^*\rightarrow\Z$, $\Psi_{\Tilde{\sigma}}$ is the restriction
of $\overline\Psi$ to the kernels, hence the $\Psi_{\Tilde{\sigma}}$ previously
defined, and $\partial$ can be viewed as the
coboundary map for calculating simplicial cohomology of $\Gamma_{\Tilde{\sigma}}$,
so that $\ker\partial=\Z$ and $\coker\partial=0$. 
Thus by the snake lemma we obtain a long exact sequence
\[
0\rightarrow\ker \Psi_{\Tilde{\sigma}}\rightarrow \ker\overline\Psi\rightarrow
\Z\rightarrow \coker \Psi_{\Tilde{\sigma}} \rightarrow \coker\overline{\Psi}
\rightarrow 0.
\]
Note that $\ker\partial$ is generated by $(1,\cdots, 1)$. As the kernel
$\ker\overline\Psi$ is precisely the integral tangent space $T_{\overline{\RR_{\geq 0}(\Tilde{\sigma},1)}}$
to $\overline{\RR_{\geq 0}(\Tilde{\sigma},1)}$,
the map $\ker\overline\Psi\rightarrow\ker\partial$ is actually surjective
by Definition \ref{def_very_good}(ii) because we are assuming that $\scrP$ is a very good polyhedral decomposition.
Thus $\coker \Psi_{\Tilde{\sigma}} \simeq \coker\overline\Psi$.

By Lemmas \ref{lem:coker psi} and \ref{lem:two gammas compare}, the cokernel $\coker \Psi_{\Tilde{\sigma}} \simeq \coker\overline\Psi$ is finite and so we are in a tropically transverse gluing situation in the sense of Definition 4.8 of \cite{gross2023remarks}. Hence, by Theorem 4.9 of \cite{gross2023remarks}, the moduli space $W$ of glued stable log maps is non-empty. Therefore, one can apply Theorem 4.4 of \cite{gross2023remarks} and obtain that $W$ has $|\coker \Psi_{\Tilde{\sigma}} |$ connected components. Lemma \ref{lem_reduced} implies that $W$ is reduced of dimension $0$, and 
hence there are $|\coker \Psi_{\Tilde{\sigma}} |$ choices of log gluings.

Now by Lemmas \ref{lem:two gammas compare} and \ref{lem:coker psi}, we have
\begin{align*}
|\coker \Psi_{\Tilde{\sigma}}|= {} & |\coker \Psi_\sigma| \frac{\prod_v w_v}{|T_{\sigma}/
T_{\tilde\sigma}|}\\
= {} &\frac{\prod_{v\in\Gamma}|(\cL_{1,v}^{\sat}+\cL_{2,v}^{\sat})/
(\cL_{1,v}+\cL_{2,v})| \prod_v w_v}{|T_{\sigma}/
T_{\tilde\sigma}|}.
\end{align*}
Combining with \eqref{eq:prelog number} and
\eqref{eq:Nv} and Lemma \ref{lem: product for Ntrop}, the total count with automorphisms of the $(\omega,\boldsymbol{\gamma})$-marked stable log maps to $\shX_\scrP/\AA^1$ of type $\tau_{\tilde{\sigma}}$ and matching $\overline{\mathbf{H}}$ is 
\begin{align*}
&\frac{ 
\prod_{v\in\Gamma}|M/(\cL_{1,v}^{\sat}+\cL_{2,v}^{\sat})|\cdot |(\cL_{1,v}^{\sat}+\cL_{2,v}^{\sat})/
(\cL_{1,v}+\cL_{2,v})|}{ |T_{\sigma}/T_{\tilde\sigma}|}\\
= {} &\frac{ 
\prod_{v\in\Gamma}|M/(\cL_{1,v}+\cL_{2,v})|}{ |T_{\sigma}/T_{\tilde\sigma}|}\\
= {} & \frac{N_{\sigma}^\trop}{|T_{\sigma}/T_{\tilde\sigma}|} \,.
\end{align*}
This concludes the proof of Theorem \ref{thm_count}.
\end{proof}

In the purely toric set-up, a general gluing formula for log Gromov--Witten invariants is also proven in \cite{wu2021splitting}. However, in the proof Theorem \ref{thm_count}, we used rather an explicit toric degeneration and then the gluing formula of \cite{gross2023remarks}. The main motivation for this choice is that we expect the explicit description of the log curves in the special fiber of the toric degeneration to be useful to define and study the higher dimensional version of the higher genus log Gromov--Witten invariants considered in \cite{MR3904449}, and which should be related to refined DT invariants.

\subsection{The proof of the log-tropical correspondence}

In this section, we prove our main log-tropical correspondence result comparing the log Gromov--Witten invariants $N_\rho^{\mathrm{toric}}(X_\Sigma)$ with the tropical multiplicities $N_\sigma^\trop$.

Recall that choosing a $\boldsymbol{\gamma}$-constraint $\mathbf{A}$ as in Definition \ref{Def affine constraint} amounts to choosing a point in $[\mathbf{A}]=([A_i])_i \in 
\prod_{i=1}^r M_\RR/(\gamma_i^{\perp})_\RR\simeq \RR^r$. 
The constraint $\mathbf{A}^0$ defined in \eqref{Eq: A0} corresponds to the origin $0 \in \RR^r$. We have a natural map 
\begin{equation}
   \label{Eq: Phi} 
   \Phi \colon \{\mathrm{Faces~of~} \cM^\trop_{\omega,\boldsymbol{\gamma},\mathbf{A}}  \} \longrightarrow \{\mathrm{Faces~of~} \cM^\trop_{\omega,\boldsymbol{\gamma},\mathbf{A}^0}   \}
\end{equation}
defined as follows. Given a face $\sigma$ of $\cM^\trop_{\omega,\boldsymbol{\gamma},\mathbf{A}}$, the relative interior $\Int(\sigma)$ of $\sigma$ parametrize $(\omega,\boldsymbol{\gamma})$-tropical curves in $M_\RR$ of a given type $\tau_\sigma$ matching the constraint $\mathbf{A}$. Let $\sigma'$ be the face of
$\cM^\trop_{\omega,\boldsymbol{\gamma}}$ such that $\Int(\sigma')$ parametrize $(\omega,\boldsymbol{\gamma})$-tropical curves of type $\tau_\sigma$.
As in \eqref{eq_ev_trop}, we have an evaluation map at the legs
\[ \ev \colon \sigma' \longrightarrow \prod_{i=1}^r M_\RR/(\gamma_i^{\perp})_\RR\simeq \RR^r\]
such that $\sigma=\ev^{-1}([\mathbf{A}])$, and we set $\Phi(\sigma):= \ev^{-1}(0)$, which is naturally a face of $\cM^\trop_{\omega,\boldsymbol{\gamma},\mathbf{A}^0}$. Intuitively, $\Phi(\sigma)$ is obtained from $\sigma$ by rescaling the point $[\mathbf{A}] \in \RR^r$ in the space of $\boldsymbol{\gamma}$-constraints by a parameter $t$ and then taking the limit $t \rightarrow 0$. As a face of $\cM^\trop_{\omega,\boldsymbol{\gamma},\mathbf{A}^0}$ is necessarily a cone, we say that $\sigma$ is \emph{asymptotically equivalent} to the cone $\Phi(\sigma)$.

If $\Sigma$ is a $\boldsymbol{\gamma}$-fan and $\rho$ is a $(d-2)$-dimensional face of $\cM^\trop_{\omega, \boldsymbol{\gamma},\mathbf{A}^0}(\Sigma)$, we denote by $S_\rho$ the set of faces $\sigma$ of 
$\cM^\trop_{\omega,\boldsymbol{\gamma},\mathbf{A}}$ such that $\rho \subset \Phi(\sigma)$. 
As $\dim \rho =d-2$ and $\mathbf{A}$ is a general $\boldsymbol{\gamma}$-constraint as in Definition \ref{def_general_constraints}, the condition $\sigma \in S_d$ automatically implies $\dim \sigma =d-2$. If in addition $\dim \fod_{L_{\mathrm{out}}}^\rho=d-1$, then $\dim \fod_{L_{\mathrm{out}}}^\sigma=d-1$, and so one can consider the tropical multiplicity $N_\sigma^\trop$ as in \eqref{eq_N_trop} and the tropical coefficient $k_\sigma$ as in \eqref{eq_coeff}.

\begin{theorem}[\textbf{Theorem \ref{thm_log_trop_intro}}]
\label{Thm:log_trop_thm}
Fix a tropical data $(\omega,\boldsymbol{\gamma},\mathbf{A})$ as in \eqref{Eq:tropical data}, and $\Sigma$ a $\boldsymbol{\gamma}$-fan as in Definition \ref{def_gamma_fan}.
For every $(d-2)$-dimensional face $\rho$ of $\cM^\trop_{\omega,\boldsymbol{\gamma},\mathbf{A}^0}(\Sigma)$ such that $\dim \fod_{L_{\mathrm{out}}}^{\rho}=d-1$, we have 
\begin{equation} \label{eq_main}
k_\rho N_\rho^{\mathrm{toric}}(X_\Sigma)=\sum_{\sigma \in S_\rho}  k_\sigma N_\sigma^\trop \,,\end{equation}
where $N_\rho^{\mathrm{toric}}(X_\Sigma)$ is the log Gromov--Witten invariant of the toric variety $X_{\Sigma}$ as in \eqref{eq_N_toric}, $N_\sigma^\trop$ are the tropical multiplicities as in \eqref{eq_N_trop}, and the tropical coefficients $k_\rho$ and $k_\sigma$ are defined as in \eqref{eq_coeff}.
\end{theorem}

\begin{proof}
If $\Sigma'$ is a refinement of $\Sigma$, then, as in \S\ref{sec_polyh}, $\cM^\trop_{\omega,\boldsymbol{\gamma},\mathbf{A}^0}(\Sigma')$ is a polyhedral refinement of $\cM^\trop_{\omega,\boldsymbol{\gamma},\mathbf{A}^0}(\Sigma)$. If $\rho'$ is a $(d-2)$-dimensional face of  $\cM_{\omega,\boldsymbol{\gamma}}(\Sigma')$ contained in $\rho$, then, by \cite[Theorem 1.4]{johnston2022birational} on the behavior of log Gromov--Witten invariants under log-\'etale transformations, we have 
$N_{\rho'}^{\mathrm{toric}}(X_{\Sigma'})=\frac{1}{|T_{\rho}/T_{\rho'}|} N_\rho^{\mathrm{toric}}(X_\Sigma)$,
and so, using \eqref{eq_coeff},
$k_{\rho'} N_{\rho'}^{\mathrm{toric}}(X_{\Sigma'})=k_\rho N_\rho^{\mathrm{toric}}(X_\Sigma)$.
Hence, it is enough to prove \eqref{eq_main} for a particular $\boldsymbol{\gamma}$-fan to have the result for all $\boldsymbol{\gamma}$-fans.

As $\QQ$ is dense in $\RR$, we can assume without loss of generality that $\mathbf{A}$ is rational.
Let $\scrP$ be a very good polyhedral decomposition of $M_\RR$ for $(\omega,\boldsymbol{\gamma},\mathbf{A})$. Such $\scrP$ exists by Lemma \ref{lem_very_good}. We will prove \eqref{eq_main} for the $\boldsymbol{\gamma}$-fan $\Sigma_\scrP$ asymptotic to $\scrP$. By Theorem \ref{thm_decomp}, we have 
\[N_\rho^{\mathrm{toric}}(X_{\Sigma_{\scrP}})
= \sum_{\Tilde{\sigma}\in \Tilde{S}_\rho} N_{\tilde{\sigma}}^{\mathrm{toric}} (\shX_\scrP/\AA^1)\,, \]
where $\tilde{S}_{\rho}$ is the set of $(d-2)$-dimensional faces $\tilde{\sigma}$
of $\cM^\trop_{\omega,\boldsymbol{\gamma},\mathbf{A}}(\scrP)$ such that $\Tilde{\Phi}(\Tilde{\sigma})=\rho$,
where $\Tilde{\Phi}$ is given by \eqref{Eq: phi}.
Every $(d-2)$-dimensional face $\tilde{\sigma}$ of $\cM^\trop_{\omega,\boldsymbol{\gamma},\mathbf{A}}(\scrP)$ is contained in a unique $(d-2)$-dimensional face $\sigma$ of $\cM^\trop_{\omega,\boldsymbol{\gamma},\mathbf{A}}$, and we have 
\[N_{\tilde{\sigma}}^{\mathrm{toric}} (\shX_\scrP/\AA^1) =\frac{N_\sigma^\trop}{|T_\sigma/T_{\tilde{\sigma}}|}\]
by Theorem \ref{thm_count}.
The maps $\tilde{\Phi}$ in \eqref{Eq: phi} and $\Phi$ in \eqref{Eq: Phi} are compatible with the polyhedral refinements of the moduli spaces induced by $\scrP$ and $\Sigma_\scrP$, and so the conditions $\tilde{\Phi}(\tilde{\sigma})=\rho$ and $\tilde{\sigma}\subset \sigma$ imply $\rho \subset \Phi(\sigma)$.
Conversely, if $\sigma$ is a $(d-2)$-dimensional face of $\cM^\trop_{\omega,\boldsymbol{\gamma},\mathbf{A}}(\scrP)$ such that $\rho \subset \Phi(\sigma)$, there exists a unique
$(d-2)$-dimensional face $\tilde{\sigma}$ of $\cM^\trop_{\omega,\boldsymbol{\gamma},\mathbf{A}}(\scrP)$ such that $\Phi(\tilde{\sigma})=\rho$.
Hence, we obtain
\[N_\rho^{\mathrm{toric}}(X_{\Sigma_{\scrP}})
= \sum_{\sigma\in S_\rho} \frac{N_{\sigma}^\trop}{|T_\sigma/T_{\tilde{\sigma}}|}\,. \]
Finally, we have 
$|T_{\sigma}/T_{\tilde{\sigma}}| = \frac{k_{\tilde \sigma}}{k_\sigma}$ by \eqref{eq_coeff},
and $k_\rho
= k_{\tilde{\sigma}}$ by Definition \ref{def_very_good}(ii) of a very good polyhedral decomposition.

\end{proof}

\section{Quiver DT invariants, flow trees and log curves}
In this section, after shortly reviewing in \S\ref{sec_aft}-\ref{sec_flow_trees} how to calculate quiver DT invariants tropically using attractor flow trees as in \cite{ABflow}, we prove our quiver DT-log Gromov--Witten correspondence in \S\ref{sec_dt_gw}.

\subsection{Attractor flow trees and tropical curves} 
\label{sec_aft}

Let $\omega$ be a skew-symmetric form on $N$. We fix a $r$-tuple 
$\boldsymbol{\gamma}=(\gamma_1,\dots,\gamma_r)$ of elements $\gamma_i \in N$ such that $\iota_{\gamma_i} 
\omega \neq 0$ for every $1\leq i\leq r$, and $\iota_\gamma \omega \neq 0$, where $\gamma:=\sum_{i=1}^r \gamma_i$.

\begin{definition}
A \emph{rooted tree} is a genus $0$ graph, without divalent vertices, and with a single univalent vertex, referred to as the \emph{root} and denoted by $R_T$. 
\end{definition}

Given a rooted tree $T$ with $r$ legs  $L_1,\dots, L_r$ decorated by $\gamma_1,\dots,\gamma_r$, we define for every edge or leg $E$ of $T$, the class $\gamma_E$ of $E$ by $\gamma_E =\sum_i \gamma_i$, where the sum is over the indices $1\leq i \leq r$ such that the leg $L_i$ is a descendant of $E$, that is such that $E$ is contained in the path from the root to $L_i$.

\begin{definition} \label{def_attractor_tree}
Fix a point $\theta \in \gamma^{\perp}$ and a $\boldsymbol{\gamma}$-constraint $\mathbf{A}$. 
A parametrized $(\omega,\boldsymbol{\gamma})$-marked \emph{attractor flow tree} with root $\theta$ and matching $\mathbf{A}$ is a pair $(T,h)$, where 
\begin{itemize}
    \item[(i)] $T$ is a rooted tree with $r$ legs decorated by $\gamma_1,\dots,\gamma_r$, such that, for every edge of leg $E$, we have $\iota_{\gamma_E} \omega \neq 0$, and for every vertex $v \in V(E)$ distinct from the root $R_T$, 
    there exist edges or legs $E_1$ and $E_2$
adjacent to $v$ such that 
\[\omega(\gamma_{E_1},\gamma_{E_2}) \neq 0\,.\] 
    \item[(ii)] $h \colon T \rightarrow M_\RR$
    is a proper continuous map such that $h(R_T)=\theta$, $h(L_i) \subset A_i$ for every $1 \leq i \leq r$, and for every edge or leg $E$, the restriction $h|_E$ is an embedding with image contained in an affine line of oriented direction $\iota_{\gamma_E}\omega$, where the orientation is defined following the flow starting at the root of $T$ and ending at the leaves.
\end{itemize}
\end{definition}

An isomorphism of parameterized attractor flow tree $h:T \to M_\RR$ and $h_0 :T_0 \to M_\RR$
is a homeomorphism $\Psi:T\to T_0$ respecting the marking of the legs, the weights of the edges and legs, and such that $h=h_0\circ \Psi$. An \emph{attractor flow tree} is an isomorphism class of parameterized attractor flow trees.
We denote by $\mathcal{T}_{\omega,\boldsymbol{\gamma},\mathbf{A}}^\theta$ the set of $(\omega,\boldsymbol{\gamma})$-marked attractor flow trees with root $\theta$ and matching $\mathbf{A}$.

Let $h: T\rightarrow M_\RR$ be an $(\omega,\boldsymbol{\gamma})$-marked attractor flow tree with root $\theta$ and matching $\mathbf{A}$. We denote by $\overline{T}$ be the graph obtained from $T$ by removing the root and viewing the edge $E_{\mathrm{out}}$ of $T$ adjacent to the root as a leg $L_{\mathrm{out}}$ of $\overline{T}$. We extend the map $h \colon T\rightarrow M_\RR$ to a map $\overline{h}:\overline{T} \rightarrow M_\RR$
by extending the segment $h(E_{\mathrm{out}})$  to a half-line in the direction of $-\iota_\gamma \omega$.
By construction, $\overline{h}:\overline{T} \rightarrow M_\RR$ is an $(\omega,\boldsymbol{\gamma})$-marked tropical curve in $M_\RR$ matching $\mathbf{A}$, and such that $\theta \in \Int(\overline{h}(L_{\mathrm{out}}))$.

\begin{definition} \label{def_family}
Fix a point $\theta \in \gamma^{\perp}$ and a $\boldsymbol{\gamma}$-constraint $\mathbf{A}$.
Let $(h:T \rightarrow M_\RR) \in \mathcal{T}_{\omega,\boldsymbol{\gamma},\mathbf{A}}^\theta$ be an $(\omega,\boldsymbol{\gamma})$-marked attractor flow tree with root  $\theta$ and matching $\mathbf{A}$. The \emph{family of tropical curves corresponding to $h$} is the face $\rho_h$ of $\mathcal{M}^\trop_{\omega,\boldsymbol{\gamma}, \mathbf{A}}$ whose relative interior parametrize tropical curves with the same type as the tropical curve $\overline{h}: \overline{T} \rightarrow M_\RR$.
\end{definition}

\begin{lemma}\label{lem_flow_trees_nice}
Fix a $\mathbf{\gamma}$-constraint 
$\mathbf{A}$
and a general point $\theta \in (\gamma^\perp)_\RR$
Then, for every $(\omega,\boldsymbol{\gamma})$-marked attractor flow tree $(h:T \rightarrow M_\RR) \in \mathcal{T}_{\omega,\boldsymbol{\gamma},\mathbf{A}}^\theta$, the family $\rho_h$ of tropical curves corresponding to $h$ satisfy the following conditions:
\begin{itemize}
    \item[(i)] For every edge or leg $E$ of $\Gamma_{\rho_h}$, $\dim \fod_E^{\rho_h}=d-1$.
    \item[(ii)] For every vertex $v$ of $\Gamma_{\rho_h}$, $\dim T_v \Gamma_{\rho_h} =2$, and $\dim \foj_v^{\rho_h}=d-2$.
\end{itemize}
\end{lemma}

\begin{proof}
We prove the result by induction following the flow on $\Gamma_{\rho_h}$ starting at the root leg and ending at the leaves.

For the initial step of the induction, note that $\dim \fod_{L_{\mathrm{out}}}^{\rho_h}=d-1$. Indeed, we have $\dim \fod_{L_{\mathrm{out}}}^{\rho_h}\leq d-1$ by Lemma \ref{lem_prep}(i), and if we had $\dim \fod_{L_{\mathrm{out}}}^{\rho_h}\leq d-2$, then $\theta$ would be constrained to lie in a strict hyperplane of $(\gamma^\perp)_\RR$, in contradiction with the assumption that $\theta$ is a general point of $(\gamma^\perp)_\RR$.

For the induction step, let $E$ be a leg or vertex of $\Gamma_{\rho_h}$ such that $\dim \fod_E^{\rho_h}=d-1$. 
Denote by $v$ the child vertex of $E$ and $E_1,\dots,E_k$ the children edges of $v$. By Lemma \ref{lem_locally_planar}, the assumption that $\dim \fod_E^{\rho_h}=d-1$ implies that $\dim T_v \Gamma_{\rho_h}\leq 2$.
On the other hand, by Definition \ref{def_attractor_tree}(i) of an attractor flow tree, there exists edges or legs $E'$ and $E''$ adjacent to $v$ such that $\omega(\gamma_{E'},\gamma_{E''})\neq 0$, and so in particular $\dim T_v \Gamma_{\rho_h} \geq 2$. We conclude that $\dim T_v \Gamma_{\rho_h}=2$ and so $\dim \foj_v^{\rho_h}=d-2$.
Finally, $\dim T_v \Gamma_{\rho_h}=2$ and $\dim \fod_E^{\rho_h}=d-1$ imply that $\dim \fod_{E_i}^{\rho_h}=d-1$ for every child edge $E_i$ of $v$ by Lemma \ref{lem_key}.
\end{proof}

\begin{lemma} \label{lem_flow_tree}
Fix a $\mathbf{\gamma}$-constraint
$\mathbf{A}$
and a general point $\theta \in (\gamma^\perp)_\RR$
Let $(h:T \rightarrow M_\RR) \in \mathcal{T}_{\omega,\boldsymbol{\gamma},\mathbf{A}}^\theta$ be an $(\omega,\boldsymbol{\gamma})$-marked attractor flow tree with root $\theta$ and matching $\mathbf{A}$. Then for every edge $E$ of $T$, with parent vertex $v$ and child vertex $v'$, the point $h(v')$ in $M_\RR$ is the unique intersection point of the half-line $h(v)+\RR_{\geq 0}\iota_{\gamma_E}\omega$ with the affine codimension two plane 
$\cap_{i=1}^k A_{E_j}$, where $E_1,\dots, E_k$ are the children edges of $v'$, and the affine hyperplanes $A_{E_j}$ are defined as in \eqref{eq_AE}.
\end{lemma}

\begin{proof}
By Definition \ref{def_attractor_tree}(ii), we have $h(v') \in h(v)+\RR_{\geq 0}\iota_{\gamma_E} \omega$. On the other hand, we have $h(v') \in \cap_{i=1}^k A_{E_i}$ by Lemma \ref{lem_prep}(ii). As $\iota_{\gamma_E} \omega \neq 0$, $h(v)+\RR_{\geq 0}\iota_{\gamma_E} \omega$ is a half-line line in $A_E$. As $\dim T_v \Gamma_{\rho_h}=2$ by Lemma \ref{lem_flow_trees_nice}, $\cap_{i=1}^k A_{E_i}$ is an affine hyperplane in
$A_E$. Finally, as $\dim \fod_E^{\rho_h}=d-1$ by Lemma \ref{lem_flow_trees_nice}, $h(v)+\RR_{\geq 0}\iota_{\gamma_E} \omega$ is not contained in $\cap_{i=1}^k A_{E_i}$, and so the intersection $(h(v)+\RR_{\geq 0}\iota_{\gamma_E} \omega)\cap (\cap_{i=1}^k A_{E_i})$ consists of at most one point. One concludes that
$(h(v)+\RR_{\geq 0}\iota_{\gamma_E} \omega)\cap (\cap_{i=1}^k A_{E_i})=\{h(v')\}$.
\end{proof}

\begin{lemma}\label{lem_tree_dim}
Fix a $\mathbf{\gamma}$-constraint
$\mathbf{A}$ and a general point $\theta \in (\gamma^\perp)_\RR$.
Let $(h:T \rightarrow M_\RR) \in \mathcal{T}_{\omega,\boldsymbol{\gamma},\mathbf{A}}^\theta$ be an $(\omega,\boldsymbol{\gamma})$-marked attractor flow tree with root $\theta$ and matching $\mathbf{A}$.
Let $v_{\mathrm{out}}$ be the unique vertex of $\Gamma_{\rho_h}$ adjacent to the root leg $L_{\mathrm{out}}$. Then, the evaluation map 
\begin{align*}
    \rho_h &\longrightarrow \foj_{v_{\mathrm{out}}}^{\rho_h}\\
    \overline{h} &\longmapsto \overline{h}(v_{\mathrm{out}})
\end{align*}
is a bijection. In particular, the dimension of the family $\rho_h$ of tropical curves corresponding to $h$ is $\dim \rho_h=d-2$.
\end{lemma}

\begin{proof}
    Following the flow on $\Gamma_{\rho_h}$ starting at the root leg and ending at the leaves, $\overline{h}$ is inductively uniquely determined by $\overline{h}(v_{\mathrm{out}})$ as in Lemma \ref{lem_flow_tree}. Moreover, we have $\dim \foj_{v_{\mathrm{out}}}^{\rho_h}=d-2$ by Lemma \ref{lem_flow_trees_nice}(ii).
\end{proof}

\subsection{The flow tree formula}
\label{sec_flow_trees}

Let $(Q,W)$ be a quiver with potential as in 
\S \ref{sec_dt_intro}. Let $\omega$ be the corresponding skew-symmetric form on $N=\Z^{Q_0}$ as in \eqref{Eq: Euler form}. We fix a dimension vector $\gamma \in N$ and a general stability parameter $\theta \in (\gamma^\perp)_\RR$. As reviewed in \eqref{eq_reconstruction_intro}, the DT invariant $\Omega_\gamma^\theta$ can be expressed in terms of the simpler attractor DT invariants in terms of universal coefficients 
$F_r^\theta(\gamma_1,\dots,\gamma_r)$ indexed by decompositions $\gamma=\sum_{i=1}^r \gamma_i$. These coefficients further decompose as 
\begin{equation}\label{eq_trees}
F_r^\theta(\gamma_1,\dots,\gamma_r)
=\sum_{h \in \mathcal{T}_{\omega,\boldsymbol{\gamma},\mathbf{A}^0}} F_{r,h}^\theta(\gamma_1,\dots,\gamma_r) \,,\end{equation}
where $\boldsymbol{\gamma}=(\gamma_1,\dots,\gamma_r)$, $\mathbf{A}^0$ is the $\boldsymbol{\gamma}$-constraint defined by the linear hyperplanes $(\gamma_i^\perp)_\RR$, and the sum is over the $(\omega,\boldsymbol{\gamma})$-marked attractor flow trees with root $\theta$ and matching $\mathbf{A}^0$ as in Definition \ref{def_attractor_tree}. The decomposition \eqref{eq_trees} is obtained by iterative use of the wall-crossing formula. In general, the attractor flow trees $h \in \mathcal{T}_{\omega,\boldsymbol{\gamma},\mathbf{A}^0}$ may contain vertices of arbitrary high valency, as in \cite[\S3]{KS}. In this section, following \cite{ABflow}, we reformulate the flow tree formula of \cite{ABflow} expressing the coefficients $F_{r,h}^\theta(\gamma_1,\dots,\gamma_r)$ in terms of simpler trivalent tropical curves matching constraints $\mathbf{A}$ which are generic small perturbations of $\mathbf{A}^0$.

For a $\boldsymbol{\gamma}$-constraint $\mathbf{A}$, recall from  \eqref{Eq: Phi} that we have a map
\begin{equation}
   \Phi \colon \{\mathrm{Faces~of~} \cM^\trop_{\omega,\boldsymbol{\gamma},\mathbf{A}}   \} \longrightarrow \{\mathrm{Faces~of~} \cM^\trop_{\omega,\boldsymbol{\gamma},\mathbf{A}^0}   \}\,.
\end{equation}
For every face $\sigma$ of $\cM^\trop_{\omega,\boldsymbol{\gamma},\mathbf{A}}$, $\Phi(\sigma)$ is the face of $\cM^\trop_{\omega,\boldsymbol{\gamma},\mathbf{A}^0}$ obtained from $\sigma$ by rescaling $\mathbf{A}$ to $\mathbf{A}^0$ in the space $\prod_{i=1}^r M_\RR/(\gamma_i^\perp)_\RR \simeq \RR^r$ of $\boldsymbol{\gamma}$-constraints. Finally, recall that for every attractor tree $h \in \mathcal{T}_{\omega,\boldsymbol{\gamma},\mathbf{A}}^\theta$, we defined 
in Definition \ref{def_family}
the corresponding family of tropical curves $\rho_h$, which is a face of $\cM^\trop_{\omega,\boldsymbol{\gamma},\mathbf{A}^0}$. By Lemma \ref{lem_tree_dim}, we have $\dim \rho_h=d-2$. Hence, if $\mathbf{A}$ is a general $\boldsymbol{\gamma}$-constraint as in Definition \ref{def_general_constraints}, and $\sigma$ is a face of $\cM^\trop_{\omega,\boldsymbol{\gamma},\mathbf{A}}$ such that $\Phi(\sigma)=\rho_h$, then $\dim \sigma=d-2$ and $\Gamma_\sigma$ is trivalent.

\begin{theorem} \label{thm_F_trop}
Fix a general point $\theta \in (\gamma^\perp)_\RR$ and an attractor flow tree $(h\colon T\rightarrow M_\RR) \in \mathcal{T}_{\omega,\boldsymbol{\gamma},\mathbf{A}^0}^\theta$. Then, for every general $\boldsymbol{\gamma}$-constraint $\mathbf{A}$, we have
\begin{equation}
    \label{eq:flow_to_trop}
F_{r,h}^\theta(\gamma_1,\dots,\gamma_r) = \sum_{\sigma\in \Phi^{-1}(\rho_h)} \prod_{v \in V(\Gamma_\sigma)} |\omega(\gamma_{E_{1,v}},\gamma_{E_{2,v}}) | \, ,
\end{equation}
where for every vertex $v \in V(\Gamma_\sigma)$, $E_{1,v}$ and $E_{2,v}$ are the two children edges of $v$. 
\end{theorem}

\begin{proof}
We explain below how Theorem \ref{thm_F_trop} follows from the flow tree formula given in \cite[Theorem 5.6]{ABflow}.
In \cite{ABflow}, we introduce a 
bigger lattice  $\mathcal{N}:=\bigoplus_{i=1}^r \Z e_i$ and we consider the map $p \colon \mathcal{N} \rightarrow N$ defined by $e_i \mapsto \gamma_i$. We also define the pullback skew-symmetric form 
$\eta$ on $\mathcal{N}$ by 
$\eta(e_i,e_j) := \omega(\gamma_i, \gamma_j)$.
By duality, we have a map 
$q : M_\RR \rightarrow \mathcal{M}_\RR=\Hom(\mathcal{N},\RR)$. Let $\tilde{\theta}$ be a small generic perturbation of $ q(\theta)$ in
$(\sum_{i=1}^r e_i)^{\perp}$.
Then, \cite[Thm.$5.6$]{ABflow} states that \begin{equation}\label{eq_ftf}[F_r^\theta(\gamma_1,\dots,\gamma_r):=\sum_{T_r} \prod_{v\in V_{T_r}^\circ} -\epsilon_{{T_r},v}^{\tilde{\theta},\omega} 
\eta(e_{v'},e_{v''})  \,,\end{equation}
where the sum is over rooted binary trees ${T_r}$ with $r$ leaves decorated by $\{e_1,\dots,e_r\}$, $V_{T_r}^\circ$ denotes the set of vertices of $T_r$ distinct from the root, for any $v \in V_{T_r}^\circ$, $e_v \in \mathcal{N}$ is the sum of $e_i$'s attached to leaves descendant from $v$, and $\epsilon_{T_r,v}^{\tilde{\theta},\omega} \in \{ -1,0,1 \}$ is defined in \cite[Eq.\! (1.11)]{ABflow} in terms of the discrete attractor flow \cite[Eq.\! (1.10)]{ABflow} embedding the trees $T_r$ in $\mathcal{M}_\RR$. Note that \cite[Theorem 5.6]{ABflow} is actually stated for the refined DT invariants, in which case $F_r^\theta(\gamma_1,\dots,\gamma_r)$ is a function of an additional variable $y$, and we obtained \eqref{eq_ftf} by taking the limit $y \rightarrow -1$ to recover the case of numerical DT invariants $\Omega_\gamma^{+,\theta}$ defined in \eqref{eq_dt_intro} (the limit $y \rightarrow 1$ corresponding to the signed DT invariants $\Omega_\gamma^\theta$).

As explained in \cite[Remark 4.25]{ABflow}, the trees embedded in $\mathcal{M}_\RR$ by the discrete attractor flow are contained in a copy of $M_\RR$ in $\mathcal{M}_\RR$.
Moreover, it follows from the proof of \cite[Eq.\! (4.57)]{ABflow} and from Lemma \ref{lem_flow_tree} that these embedded trees are
exactly attractor flow trees as in Definition \ref{def_attractor_tree}, with root at a general point $\tilde{\theta}$ of $A_{L_{\mathrm{out}}}$ close enough to $\theta$, and matching a small generic perturbation $\mathbf{A}$ of the constraint $\mathbf{A}^0$, if and only if $\epsilon_{T_r,v}^{\tilde{\theta},\omega} \neq 0$, and that in this case,
 $-\epsilon_{{T_r},v}^{\tilde{\theta},\omega}$ has the same sign as $ 
\eta(e_{v'},e_{v''})$ (see the paragraph above \cite[Eq.\! (4.52)]{ABflow}). In particular, a tree $T_r$ with a non-zero contribution to \eqref{eq_ftf} contributes $\prod_{v\in V_{T_r}^\circ}|
\eta(e_{v'},e_{v''})|$, and so \eqref{eq:flow_to_trop}
follows from \eqref{eq_ftf}.

\end{proof}

\begin{remark}
    Theorem \ref{thm_F_trop} could also be deduced from \cite[Theorem 4.4]{mandel2020disks} expressing the functions attached to the walls of a scattering diagram in terms of tropical disks. The proof of \cite[Theorem 5.6]{ABflow} is also based on the study of a scattering diagram in $\mathcal{M}_\RR$ which can be viewed as a universal family for the perturbed scattering diagrams in $M_\RR$ considered in \cite[Theorem 4.4]{mandel2020disks} (see \cite[Remark 4.25]{ABflow}).
\end{remark}

\subsection{The quiver DT-log Gromov--Witten correspondence}
\label{sec_dt_gw}

Let $(Q,W)$ be a quiver with potential as in 
\S \ref{sec_dt_intro}. Let $\omega$ be the corresponding skew-symmetric form on $N=\Z^{Q_0}$. We fix a dimension vector $\gamma \in N$ such that $\iota_\gamma \omega \neq 0$
and a decomposition $\gamma=\sum_{i=1}^r \gamma_i$ such that $\iota_{\gamma_i}\omega \neq 0$ for all $1\leq i\leq r$. We denote $\boldsymbol{\gamma}=(\gamma_1,\dots,\gamma_r)$.
We also fix a general stability parameter $\theta \in (\gamma^\perp)_\RR$. In this section, we prove our main result relating the coefficients  $F_r^\theta(\gamma_1,\dots,\gamma_r)$
in \eqref{eq_reconstruction_intro}-\eqref{eq_trees_intro}, expressing DT invariants in terms of attractor DT invariants, and log Gromov--Witten invariants of toric varieties.

Recall that given an attractor flow tree $(h: T \rightarrow M_\RR)\in \mathcal{T}_{\omega,\boldsymbol{\gamma}, 
\mathbf{A}^0}^\theta$, we defined in Definition \ref{def_family} the corresponding family of tropical curves $\rho_h$: it is a face of $\mathcal{M}_{\omega,\boldsymbol{\gamma},\mathbf{A}^0}^\trop$ such that $\dim \rho=d-2$ by Lemma \ref{lem_tree_dim} and $\dim \fod_{L_{\mathrm{out}}}^{\rho_h}=d-1$ by Lemma \ref{lem_flow_trees_nice}.
Given a $\boldsymbol{\gamma}$-fan $\Sigma$ as in Definition \ref{def_gamma_fan}, we denote by $\widetilde{\rho}_h$ the $(d-2)$-dimensional face of $\mathcal{M}_{\omega,\boldsymbol{\gamma},\mathbf{A}^0}^\trop(\Sigma)$ contained in $\rho_h$ and containing $\theta$. In particular, one can consider as in \eqref{eq_N_toric} the genus $0$ log Gromov--Witten invariant $N_{\widetilde{\rho}_h}^{\mathrm{toric}}(X_\Sigma)$ of the toric variety $X_\Sigma$ of fan $\Sigma$, and the tropical coefficient $k_{\widetilde{\rho}_h}$ defined as in \eqref{eq_coeff}. By Theorem \ref{thm_enum},  $N_{\widetilde{\rho}_h}^{\mathrm{toric}}(X_\Sigma)$ is actually an enumerative count of log curves.

\begin{theorem}[\textbf{Theorem \ref{thm_dt_gw_intro}}] \label{thm_dt_gw}
Fix a general stability parameter $\theta \in (\gamma^\perp)_\RR$ 
and an attractor flow tree $(h: T \rightarrow M_\RR)\in \mathcal{T}_{\omega,\boldsymbol{\gamma}, 
\mathbf{A}^0}^\theta$. 
Then, for every $\boldsymbol{\gamma}$-fan $\Sigma$,
the coefficient  $F_r^\theta(\gamma_1,\dots,\gamma_r)$
in \eqref{eq_reconstruction_intro}-\eqref{eq_trees_intro}, satisfies
\begin{equation} \label{eq_dt_log}
F_{r,h}^\theta(\gamma_1,\dots,\gamma_r)
= \frac{\prod_{i=1}^r |\gamma_i|}{|\gamma|}
k_{\widetilde{\rho}_h}
N_{\widetilde{\rho}_h}^{\mathrm{toric}}(X_\Sigma)\,.\end{equation}
\end{theorem}

\begin{proof}
Fix a general $\boldsymbol{\gamma}$-constraint $\mathbf{A}$, which exists by Lemma 
\ref{lem_dimension}. 
Then, Theorem \ref{thm_F_trop} expresses
the coefficient $F_{r,h}^\theta(\gamma_1,\dots,\gamma_r)$ in terms of the $(d-2)$-dimensional faces $\sigma$ of the moduli space $\mathcal{M}_{\omega,\boldsymbol{\gamma},\mathbf{A}}^\trop$ of $(\omega,\boldsymbol{\gamma})$-marked tropical curves matching $\mathbf{A}$ as
\begin{equation}
    \label{eq_proof_11}
F_{r,h}^\theta(\gamma_1,\dots,\gamma_r) = \sum_{\sigma\in \Phi^{-1}(\rho_h)} \prod_{v \in V(\Gamma_\sigma)} |\omega(\gamma_{E_{1,v}},\gamma_{E_{2,v}}) | \,.
\end{equation}
Using the product formula for tropical multiplicities in Lemma \ref{lem: product2 for Ntrop}, this can be rewritten in terms of the tropical multiplicities $N_\sigma^\trop$ defined in \eqref{eq_N_trop}, and of the tropical coefficients $k_\sigma$ are defined in \eqref{eq_coeff}, as 
\begin{equation}
    \label{eq_proof_12}
F_{r,h}^\theta(\gamma_1,\dots,\gamma_r) = \frac{\prod_{i=1}^r |\gamma_i|}{|\gamma|}\sum_{\sigma\in \Phi^{-1}(\rho_h)} k_\sigma N_\sigma^\trop \, .
\end{equation}
On the other hand, by the log-tropical correspondence of Theorem \ref{Thm:log_trop_thm}, we have 
\begin{equation}\label{eq_proof_13}
k_{\widetilde{\rho}_h}
N_{\widetilde{\rho}_h}^{\mathrm{toric}}(X_\Sigma)=\sum_{\sigma\in S_{\widetilde{\rho}_h}} k_\sigma N_\sigma^{\trop} \,,
\end{equation}
where  $S_{\widetilde{\rho}_h}$ 
is the set of faces $\sigma$ of 
$\cM^\trop_{\omega,\boldsymbol{\gamma},\mathbf{A}}$ such that $\widetilde{\rho}_h \subset \Phi(\sigma)$.
As $\rho_{h}$ is the $(d-2)$-dimensional face of $\mathcal{M}_{\omega,\boldsymbol{\gamma},\mathbf{A}^0}$ containing $\widetilde{\rho}_h$, the set $\Phi^{-1}(\rho_h)$ of $\sigma$'s such that $\Phi(\sigma)=\rho_h$ coincides with the set $S_{\tilde{\rho}_h}$ of $\sigma$'s
such that $\tilde{\rho}_h \subset \Phi(\sigma)$, and so \eqref{eq_dt_log} follows by combining \eqref{eq_proof_12} and \eqref{eq_proof_13}.
\end{proof}

\appendix
\section{Generic log smoothness}

The classical generic smoothness (or Bertini-Sard) theorem states that given $X$ and $Y$ two varieties over a field $\kk$ of characteristic 0, and a morphism $f \colon X \rightarrow Y$, if $X$ is smooth over $\kk$, then there exists an open dense subset $U$ of $Y$ such that for every $u\in U$, the fiber $f^{-1}(u)$ is smooth over $\kk$ (see e.g. \cite[Corollary 10.7]{hartshorne}). We give below a logarithmic version of this result.

\begin{theorem} \label{thm_appendix}
    Let $\kk$ be a field of characteristic $0$. Let $f: X\rightarrow Y$ be a morphism between fine and saturated log schemes over $\kk$. Assume that $X$ is log smooth over the trivial log point $(\Spec \kk, \kk^\star)$, that the scheme underlying $Y$ is integral and that the log structure on $Y$ is generically trivial. Then, there exists an open dense subset $U$ of $Y$ such that for every $u \in U$ the fiber $f^{-1}(u)$, endowed with the log structure restricted from $X$, is log smooth over the trivial log point $(\Spec \kk(u), \kk(u)^\star)$, where $\kk(u)$ is the residue field of $u$.
\end{theorem}

\begin{proof}
As $X$ is log smooth over the trivial log point $(\Spec \kk, \kk^\star)$, $X$ is log regular in the sense of \cite[(2.1)]{kato_toric} by \cite[Proposition 8.3 (1)]{kato_toric}. By \cite[Proposition 7.2]{kato_toric} the localization of a log regular log scheme is log regular. So, the generic fiber $X_\eta$ of $X$ over the generic point $\eta$ of $Y$ is also log regular. As $\kk$ is of characteristic $0$, the function field $\kk(\eta)$ of $Y$ is also of characteristic zero, and so $X_\eta$ log regular implies that $X_\eta$ is log smooth over the trivial log point $(\Spec \kk(\eta), \kk(\eta)^{\star})$ by \cite[Proposition 8.3 (2)]{kato_toric}. As the log structure on $Y$ is assumed to be generically trivial, $(\Spec \kk(\eta), \kk(\eta)^{\star})$ is the generic point of $Y$ endowed with the log structure restricted from $Y$, and so there exists an open dense subset $V$ of $Y$ such that $f|_{f^{-1}(V)}$ is log smooth by the generic smoothness result of \cite[\S 1.8]{lodh2013bertini}.
On the other hand, applying  \cite[\S 1.8]{lodh2013bertini} to $Y$ over $(\Spec \kk, \kk^\star)$, one deduces that there exists an open dense subset $W$ of $Y$ such that the log structure of $Y$ is trivial in restriction to $W$. The intersection $U:=V \cap W$ is an open dense subset of $Y$ with trivial log structure such that $f|_{f^{-1}(U)}$ is log smooth. Hence, for every $u\in U$, the fiber $f^{-1}(u)$ is log smooth over the trivial log point $(\Spec \kk(u), \kk(u)^\star)$.
\end{proof}

\FloatBarrier

\bibliographystyle{plain}
\bibliography{bibliography}


\end{document}